\documentclass[11pt]{amsart}

\usepackage{a4wide, amsmath, amsfonts, amssymb, mathrsfs, amsthm, breqn}
\usepackage[hyperfootnotes=false]{hyperref}
\usepackage{shuffle}
\usepackage{enumitem}
\usepackage{graphicx}
\usepackage{subcaption}
\allowdisplaybreaks[2]

\numberwithin{equation}{section}

\newtheorem{theorem}{Theorem}[section]
\newtheorem{lemma}[theorem]{Lemma}
\newtheorem{corollary}[theorem]{Corollary}
\newtheorem{remark}[theorem]{Remark}
\newtheorem{definition}[theorem]{Definition}

\newtheorem{proposition}[theorem]{Proposition}

\newtheorem{example}[theorem]{Example}
\newtheorem*{gammaRIE}{Property~$\boldsymbol{\gamma}$-\textup{(RIE)}}
\newtheorem*{RIE}{Property~\textup{(RIE)}}

\newcommand{\dd}{\,\mathrm{d}}
\renewcommand{\d}{\mathrm{d}}
\newcommand{\hd}{\hat{d}}
\renewcommand{\epsilon}{\varepsilon}
\renewcommand{\phi}{\varphi}
\newcommand{\R}{\mathbb{R}}
\newcommand{\N}{\mathbb{N}}
\newcommand{\X}{\mathbb{X}}
\renewcommand{\P}{\mathbb{P}}
\newcommand{\bX}{\mathbf{X}}
\newcommand{\cP}{\mathcal{P}}
\newcommand{\p}{\frac{p}{2}}
\newcommand{\QX}{Q^{\gamma,\pi}(X)}
\newcommand{\cC}{\mathcal{C}}
\newcommand{\tbbX}{\widetilde{\mathbb{X}}}
\newcommand{\hW}{\widehat{W}}
\newcommand{\hX}{\widehat{X}}
\newcommand{\hbX}{\widehat{\mathbf{X}}}
\newcommand{\hbbX}{\widehat{\mathbb{X}}}
\newcommand{\hbbY}{\widehat{\mathbb{Y}}}

\DeclareMathOperator{\spn}{span}

\title[Universal approximation with signatures of non-geometric rough paths]{Universal approximation with signatures of non-geometric rough paths}

\author[Ceylan]{Mihriban Ceylan}
\address{Mihriban Ceylan, University of Mannheim, Germany}
\email{mihriban.ceylan@uni-mannheim.de}

\author[Kwossek]{Anna P. Kwossek}
\address{Anna P. Kwossek, University of Vienna, Austria}
\email{anna.paula.kwossek@univie.ac.at}

\author[Pr{\"o}mel]{David J. Pr{\"o}mel}
\address{David J. Pr{\"o}mel, University of Mannheim, Germany}
\email{proemel@uni-mannheim.de}

\date{\today}

\begin{document}

\begin{abstract}
  We establish a universal approximation theorem for signatures of rough paths that are not necessarily weakly geometric. By extending the path with time and its rough path bracket terms, we prove that linear functionals of the signature of the resulting rough paths approximate continuous functionals on rough path spaces uniformly on compact sets. Moreover, we construct the signature of a path extended by its pathwise quadratic variation terms based on general pathwise stochastic integration {\`a} la F{\"o}llmer, in particular, allowing for pathwise It{\^o}, Stratonovich, and backward It{\^o} integration. In a probabilistic setting, we obtain a universal approximation result for linear functionals of the signature of continuous semimartingales extended by the quadratic variation terms, defined via stochastic It{\^o} integration. Numerical examples illustrate the use of signatures when the path is extended by time and quadratic variation in the context of model calibration and option pricing in mathematical finance.
\end{abstract}

\maketitle 

\noindent \textbf{Key words:} signature; universal approximation; rough path; pathwise stochastic integration; semimartingale; model calibration; pricing of financial derivatives.

\noindent \textbf{MSC 2020 Classification:} Primary: 60L10; Secondary: 60H05; 60G17; 91G60.



\section{Introduction}

Many quantities that occur from real-world dynamics or complex mathematical models are analytically intractable, making it practically unavoidable to approximate them by simpler, numerically tractable ones that can be fitted to data. In this regard, it is very relevant, for example, for machine learning tasks, quantitative finance, and data analysis in general, to ``faithfully'' summarize time series data in an efficient and computable way. In response, in recent years, an increasingly active strand of research has been concerned with developing and applying data-driven methods based on the \emph{signature of a path}, which turns out to be a very suitable feature map for streamed data.

In mathematical finance, the applications are manifold and include, among others, asset pricing of European~\cite{Arribas2018,Lyons2019,Lyons2020} and American options~\cite{Bayraktar2022,Bayer2021,Bayer2025}, detection of market anomalies \cite{Akyildirim2022}, optimal execution~\cite{Kalsi2020,Cartea2022}, portfolio optimization~\cite{Futter2023,Cuchiero2025}, and calibration of financial models~\cite{Arribas2021,Cuchiero2023a,Cuchiero2024}. For a comprehensive exposition of this fast-growing field, we refer to~\cite{Bayer2025a}. Beyond mathematical finance, signature-based techniques have been applied to machine learning in a variety of contexts, including computer vision, natural language processing, and medical data analysis; see, for instance,~\cite{Chevyrev2026,McLeod2025} and the references therein.

The signature of a path was first introduced by Chen~\cite{Chen1957,Chen1977} and plays a prominent role in rough path theory, initiated by Lyons~\cite{Lyons1998}, which provides a rich mathematical framework for analyzing complex evolving systems driven by irregular signals. Formally, the signature of a path $X \colon [0,T] \to \mathbb{R}^d$ is defined as the collection of all the iterated integrals of the path $X=(X^1,\dots,X^d)$ against itself, that is,
\begin{equation*}
  \int_{0 < t_1 < \ldots < t_n < T} \d X^{i_1}_{t_1} \cdots \d X^{i_n}_{t_n},
\end{equation*}
for $i_1, \ldots, i_n \in \{1, \ldots, d\}$ and all $n \in \N$. Assuming the integrals are well-defined, using a suitable notion of integration, the signature summarizes the full evolution and interactions of the components of the path effectively: the signature is known to provide an intriguing nonlinear characterization of the path that is unique up to tree-like equivalence, see~\cite{Hambly2010,Boedihardjo2016}. Importantly, due to the rich algebraic structure, that is immediate given a suitable notion of integration, the signature comes with an intriguing universal approximation property: linear functionals of the signature approximate continuous functionals of the path arbitrarily well on compact sets, analogously to polynomials approximating continuous real valued functions; see, e.g.,~\cite{Levin2013,Lyons2020,Cuchiero2023b}. This universality result lies at the heart of most signature-based methods.

When considering smooth paths $X \colon [0,T] \to \mathbb{R}^d$, the signature is canonically defined via Riemann--Stieltjes integration. When, however, considering paths of low regularity, such as the sample paths of Brownian motion, or more generally, of semimartingales, one cannot rely on Riemann--Stieltjes (or Young) integration, and typically turns to stochastic calculus. Here, the notion of integral becomes ambiguous, with It{\^o} and Stratonovich integration being the most common choices. Once a probabilistic structure is fixed, one can then construct the second-order iterated integrals
\begin{equation*}
  \X_{s,t}^{(2)} :=\bigg(\int_{s < t_1 <t_2 < t} \d X^{i_1}_{t_1}  \d X^{i_2}_{t_2}\bigg)_{i_1, i_2 \in \{1, \ldots, d\}}
  \quad \text{for } s,t\in [0,T],
\end{equation*}
for almost every sample path $X$ of a continuous semimartingale. 

As the Stratonovich integral satisfies first-order calculus (as does the Riemann--Stieltjes integral for smooth paths), using that notion of integration gives rise to a \emph{weakly geometric} rough path $(X, \X^{(2)})$ that is to be understood in the sense of rough path theory; see, e.g., \cite{Lyons2007, Friz2010}. Classically, in the rough path literature, the signature of a (weakly geometric) rough path is defined via Lyons' extension theorem~\cite{Lyons2007}; see Definition~\ref{def: Lyons' lift of weakly geometric rough path} and also, for example, ~\cite{Levin2013,Cuchiero2023b}. Enhancing a weakly geometric rough path, we inherently obtain the desired algebraic, and also geometric, properties of the signature. In particular, the signature is a group-like element and, therefore, satisfies the shuffle property. As a consequence, the linear span of the signature components forms a point-separating algebra. This in turn allows to apply the Stone--Weierstrass theorem to deduce the above-mentioned universal approximation property.

In mathematical finance, It{\^o} integration is typically the more natural choice of stochastic integration, as the It{\^o} integral preserves the martingale property, which also underlies the principle of no-arbitrage, and allows for a transparent financial interpretation---for instance when used to model the capital gains process. It{\^o} integration is preferred also in continuous-time econometric analysis, see~\cite{Bandi2025} for an application of signatures, as well as for model order reduction, as pointed out, for instance, in~\cite{Bayer2024}. Moreover, the signature associated with It{\^o} integration can offer statistical advantages; see~\cite{Guo2025} for a comparison of the statistical consistency of the Lasso estimator using the signature based on It{\^o} and based on Stratonovich integration.

If one constructs the second-order iterated integrals $\X^{(2)}$ as above, now using It{\^o} integration, the resulting rough path $(X,\X^{(2)})$ is in general \emph{not} weakly geometric. The associated signature then is not a group-like valued path, but a path that takes values in the extended tensor algebra. In particular, the shuffle property does not hold and the Stone--Weierstrass theorem cannot be immediately applied to deduce the universal approximation property.

The aim of this paper is to present universal approximation theorems for signatures of rough paths that are not necessarily weakly geometric. We thereby provide a theoretical foundation for approximations based on the It{\^o} signature, which is conceptually natural from the perspective of mathematical finance. Moreover, our work contributes to a recent line of research that has been concerned with deriving universal approximation results for signatures of general rough paths. For instance, see \cite{Harang2024} for a universal approximation result for polynomial functionals of the signature of rough paths, and see~\cite{Ali2025}, where they consider the branched signature based on the branched rough path framework in the sense of~\cite{Gubinelli2010}.

In Section~\ref{sec: general pathwise signatures}, we consider signatures of general rough paths which are not necessarily weakly geometric, and in the $p$-variation regularity regime for $p \in (2,3)$. We show that linear functionals of the signature of rough paths approximate continuous functionals on rough path spaces uniformly on subsets of compacts, when the paths is extended by time and its rough path brackets. More precisely, and related to the natural It{\^o}--Stratonovich correction for semimartingales (see also~\cite{Boedihardjo2019,Bellingeri2026}), our approach is to start with a rough path, and to extend the underlying path by time and its rough path bracket terms. We then lift the extended path to a rough path, and define its signature via Lyons' extension theorem. Although this increases the dimension of the path, it ensures that the resulting signature satisfies the so-called \emph{quasi shuffle property} (see Proposition~\ref{prop: quasi shuffle property}), and that the linear span of its components does form a point-separating algebra. As a consequence, we are able to deduce the universal approximation property of linear functionals of the signature; see Theorem~\ref{thm: UAT non-geometric}. This universal approximation theorem extends the classical one for signatures of weakly geometric rough paths, and allows to consider signatures of such general rough paths as a linear regression basis for continuous path functionals. For completeness, we also provide in Appendix~\ref{appendix: proof of UAT geometric} a direct proof of the universal approximation theorem for signatures of time-extended weakly geometric rough paths.

As an example that fits into this framework, in Section~\ref{sec: signature via pathwise integration}, we introduce a notion of a signature constructed via general pathwise stochastic integration that can be seen as a generalization of F{\"o}llmer integration~\cite{Follmer1981}. For this purpose, we make use of the path property $\gamma$-(RIE), which has been studied in \cite{Das2025}, and which ensures that the path, that is continuous and of finite $p$-variation for some $p \in (2,3)$, extends canonically to a (not necessarily weakly geometric) rough path, where the lift is given as limits of Riemann sums. For such paths, we define the \emph{$\gamma$-signature} as Lyons' extension of this canonical rough path, which coincides with the collection of iterated rough integrals of controlled paths with respect to the rough path. (We present a rigorous proof of this statement for general continuous $p$-rough paths, $p \in (2,3)$, in Appendix~\ref{sec: finite p-variation}.) We remark that the pathwise integrals exist as limits of general Riemann sums along suitable sequences of partitions, thus yielding a unifying framework for pathwise Stratonovich, It{\^o} and backward It{\^o} integration. Corollary~\ref{cor: pathwise UAT under gamma RIE} then states the corresponding universal approximation property of linear functionals of the $\gamma$-signature of paths extended by time and their rough path bracket terms. In this setting, the rough path bracket is closely linked to F{\"o}llmer's notion of pathwise quadratic variation.

In Section~\ref{sec: application to continuous semimartingales}, the deterministic theory is translated into the probabilistic setting using stochastic integration. In particular, we consider the It{\^o} signature of continuous semimartingales, that is, the collection of iterated integrals defined via It{\^o} integration, and obtain a universal approximation theorem for linear functionals of the It{\^o} signature of continuous semimartingales extended by time and their quadratic variation terms; see Corollary~\ref{cor: UAT Ito}.

One may now pose the question when it is actually advantageous in practice to extend the path additionally by its quadratic (co)-variation. Therefore, in Section~\ref{sec: numerical results}, we finally provide numerical examples to briefly showcase the implications of using It\^o signatures in applications in finance: we consider calibration to time-series data, payoff approximation, and pricing tasks for options that naturally depend on quadratic variation.

\medskip
\noindent \textbf{Organization of the paper:} In Section~\ref{sec: general pathwise signatures} we develop the framework for signatures of rough paths and derive a universal approximation theorem using rough paths that are not necessarily weakly geometric, and where the path is extended by its rough path bracket terms. Section~\ref{sec: signature via pathwise integration} introduces the notion of $\gamma$-signatures, based on general pathwise stochastic integration, and presents corresponding universal approximation results. In Section~\ref{sec: application to continuous semimartingales} we translate these findings to the probabilistic setting of continuous semimartingales and obtain a universal approximation theorem for It{\^o} signatures of continuous semimartingales. Section~\ref{sec: numerical results} provides numerical experiments on signature-based approaches for calibration and option pricing, using paths extended by time and quadratic variation. Finally, Appendices~\ref{appendix: proof of UAT geometric}--\ref{appendix: some results} contain proofs that have been postponed, and auxiliary results from rough path theory.

\medskip 
\noindent\textbf{Acknowledgments:} The authors would like to thank C.~Cuchiero for her helpful suggestions during the preparation of this paper. M.~Ceylan gratefully acknowledges financial support by the doctoral scholarship programme from the Avicenna-Studienwerk, Germany.

\section{The signature of rough paths}\label{sec: general pathwise signatures}

We will first recall some essentials from the theory of signatures and rough paths, which we divide into the algebraic and analytic concepts. For a more detailed introduction, we refer to~\cite{Lyons2007,Friz2010}.

\subsection{Algebraic setting for signatures}

The \emph{tensor algebra} and the \emph{extended tensor algebra} on $\R^d$ are defined by
\begin{equation*}
  T(\R^d) := \bigoplus_{n=0}^\infty (\R^d)^{\otimes n} \qquad \text{and} \qquad T((\R^d)) := \prod_{n=0}^{\infty} (\R^d)^{\otimes n},
\end{equation*}
where $(\R^d)^{\otimes n}$ denotes the $n$-fold tensor product of $\R^d$, with the convention $(\R^d)^{\otimes 0} := \R$.

We equip $T((\R^d))$ with the standard addition $+$, tensor multiplication $\otimes$ and scalar multiplication, which is defined for $\mathbf{a} = (a^{(n)})_{n=0}^\infty, \mathbf{b} = (b^{(n)})_{n=0}^\infty \in T((\R^d))$, $\lambda \in \R$, by setting
\begin{align*}
  \mathbf{a} + \mathbf{b} := (a^{(n)} + b^{(n)})_{n=0}^\infty, \quad
  \mathbf{a} \otimes \mathbf{b} := \bigg(\sum_{i+j=n} a^{(i)} \otimes b^{(j)}\bigg)_{n=0}^\infty,\quad\text{and}\quad
  \lambda \mathbf{a} := (\lambda a^{(n)})_{n=0}^\infty.
\end{align*}
These operations induce analogous operations on $T(\R^d)$ and $T^N(\R^d)$ defined below.

We observe that $(T((\R^d)),+,\cdot,\otimes)$ is a real non-commutative algebra. The neutral element is $(1,0,\dots,0,\dots)$.

Let $(e_1, \ldots, e_d)$ be the canonical basis of $\R^d$. \sloppy The Lie algebra that is generated by $\{\mathbf{e}_1, \dots, \mathbf{e}_d\}$, where $\mathbf{e}_i := (0,e_i,0,\dots) \in T(\R^d)$, and the commutator bracket
\begin{equation*}
  [\mathbf{a},\mathbf{b}] = \mathbf{a} \otimes \mathbf{b} - \mathbf{b} \otimes \mathbf{a}, \qquad \mathbf{a}, \mathbf{b} \in T(\R^d),
\end{equation*}
is called the \emph{free Lie algebra} $\mathfrak{g}(\R^d)$ over $\R^d$, see e.g.~\cite[Section~7.3]{Friz2010}. It is a subalgebra of $T_0((\R^d))$, where we define for $c \in \R$, the tensor subalgebra $T_c((\R^d)) := \{\mathbf{a} = (a^{(n)})_{n=0}^\infty \in T((\R^d)): a^{(0)} = c\}$. 

The \emph{free Lie group} $G((\R^d)) := \exp(\mathfrak{g}(\R^d))$ is defined as the tensor exponential of $\mathfrak{g}(\R^d)$, i.e., its image under the map
\begin{equation*}
  \exp_{\otimes} \colon T_0((\R^d)) \to T((\R^d)), \qquad \mathbf{a} \mapsto 1 + \sum_{k=1}^{\infty} \frac{1}{k!} \mathbf{a}^{\otimes k}.
\end{equation*}
$G((\R^d))$ is a subgroup of $T_1((\R^d))$. In fact, $(G((\R^d)),\otimes)$ is a group with unit element $(1,0,\dots,0,\dots)$, and for all $\mathbf{g} = \exp_{\otimes}(\mathbf{a}) \in G((\R^d))$, the inverse with respect to $\otimes$ is given by $\mathbf{g}^{-1} = \exp_{\otimes}(-\mathbf{a})$, for $\mathbf{g} = \exp_{\otimes}(\mathbf{a}) \in G((\R^d))$. We call elements in $G((\R^d))$ \emph{group-like} elements. 

For $N \in \N$, the \emph{truncated tensor algebra} on $\R^d$ is defined by
\begin{equation*}
  T^N(\R^d) := \bigoplus_{n=0}^N (\R^d)^{\otimes n}.
\end{equation*}
For any $\mathbf{a} = (a^{(n)})_{n=0}^N \in T^N(\R^d)$, we set 
\begin{equation*}
  |\mathbf{a}|_{T^N(\R^d)} := \max_{n=0,\dots,N} |a^{(n)}|_{(\R^d)^{\otimes n}},
\end{equation*}
where we write $|\cdot|$ for the Euclidean norm, on $\R^d$ or $(\R^d)^{\otimes n}$ for some $n \in \N$. We consider the maps $\Pi_{(n)} \colon T((\R^d)) \to (\R^d)^{\otimes n}$ and $\Pi_{N} \colon T((\R^d)) \to T^N(\R^d)$, where $\Pi_{(n)}(\mathbf{a}) = a^{(n)}$ and $\Pi_{N}(\mathbf{a}) = (a^{(0)}, \dots, a^{(N)})$, for $\mathbf{a} = (a^{(n)})_{n=0}^\infty \in T((\R^d))$. We set for $c \in \R$, $T_c^N(\R^d) := \{\Pi_{N}(\mathbf{a}): \mathbf{a} \in T_c((\R^d))\}$. Then $T_1^N(\R^d)$ is a Lie group under the tensor multiplication $\otimes$, truncated at level $N$. We equip $T_1^N(\R^d)$ with the metric 
\begin{equation*}
  \rho(\mathbf{a},\mathbf{b}) := |\mathbf{a} - \mathbf{b}|_{T^N(\R^d)} = \max_{n=1,\ldots,N}|(a-b)^{(n)}|_{(\R^d)^{\otimes n}},
\end{equation*}
for $\mathbf{a}=(a^{(n)})_{n=0}^N, \mathbf{b}=(b^{(n)})_{n=0}^N \in T_1^N(\R^d)$, which arises from the norm on $T^N(\R^d)$.

The \emph{free nilpotent Lie algebra} and the \emph{free nilpotent Lie group of order $N$} are defined by $\mathfrak{g}^N(\R^d) := \Pi_{N}(\mathfrak{g}(\R^d))$ and $G^N(\R^d) := \Pi_{N}(G((\R^d)))$, respectively. That is, 
\begin{equation*}
  \mathfrak{g}^N(\R^d) = \{0\} \oplus \R^d \oplus [\R^d,\R^d] \oplus \dots \oplus \underbrace{[\R^d,[\R^d,\dots[\R^d,\R^d]]]}_{\text{$N-1$ brackets}} \subseteq T_0^N(\R^d).
\end{equation*}
Then $G^N(\R^d)$ is a subgroup of $T_1^N(\R^d)$ with respect to $\otimes$.

Defining the truncated tensor exponential via the corresponding (finite) power series in the truncated tensor algebra, we have that $G^N(\R^d) = \exp_\otimes^N(\mathfrak{g}^N(\R^d))$.

\medskip 

Now, let $I = (i_1, \ldots, i_{|I|})$ be a multi-index (with entries in $\{1, \dots, d\}$) of length $|I|$. We recall the canonical basis $(e_1, \dots, e_d)$ of $\R^d$, and set $e_I := e_{i_1} \otimes \dots \otimes e_{i_{|I|}}$. If $|I|=1$, set $I' = \emptyset$, if $|I| \geq 1$, $I' = (i_1, \ldots, i_{|I|-1})$. Moreover, we denote by $e_\emptyset$ the basis element of $(\R^d)^{\otimes 0}$ and set $|\emptyset| := 0$. This allows to write $\mathbf{a} \in T((\R^d))$ (and $\mathbf{a} \in T(\R^d)$) as
\begin{equation*}
  \mathbf{a} = \sum_{|I| \geq 0} a_I e_I,
\end{equation*}
for some $a_I \in \R$.

Furthermore, for $\mathbf{a} \in T(\R^d)$ and $\mathbf{b} \in T((\R^d))$, we set
\begin{equation*}
  \langle \mathbf{a}, \mathbf{b} \rangle := \sum_{|I| \geq 0} \langle a_I, b_I \rangle.
\end{equation*}
where $\langle \cdot,\cdot\rangle$ is defined as the inner product of $(\R^d)^{\otimes n}$ for each $n\ge 0$. Then $(e_I)_{\{I: |I|=n\}}$ is the canonical orthonormal basis of $(\R^d)^{\otimes n}$ with respect to this inner product and we write $a_I:=\langle e_I,\mathbf{a}\rangle := \langle e_I, \Pi_{(|I|)}(\mathbf{a}) \rangle$.

Associating $\ell \in T(\R^d)$ with a linear functional $\langle \ell, \cdot \rangle \colon T((\R^d)) \to \R$, we write
\begin{equation*}
  \langle \ell, \mathbf{a} \rangle := \sum_{0 \leq |I| \leq N} \ell_I \langle e_I, \mathbf{a} \rangle, \qquad \mathbf{a} \in T((\R^d)),
\end{equation*}
for $\ell = \sum_{0 \leq |I| \leq N} \ell_I e_I$, where $\ell_I := \langle e_I, \ell \rangle \in \R$ and $N \in \N_0$.

\medskip 

For two multi-indices $I = (i_1, \ldots, i_{|I|})$, $J = (j_1, \ldots, j_{|J|})$ with entries in $\{1,\ldots,d\}$, the \emph{shuffle product} is recursively defined by
\begin{equation*}
  e_I \shuffle e_J := (e_{I'} \shuffle e_J) \otimes e_{i_{|I|}} + (e_I \shuffle e_{J'}) \otimes e_{j_{|J|}},
\end{equation*}
with $e_I \shuffle e_\emptyset := e_\emptyset \shuffle e_I := e_I$. For $\mathbf{a}, \mathbf{b} \in T(\R^d)$, we set
\begin{equation*}
  a \shuffle b = \sum_{|I|, |J| \geq 0} a_I b_J (e_I \shuffle e_J)
\end{equation*}
and for $\mathbf{a}, \mathbf{b} \in T((\R^d))$, we set
\begin{equation*}
  \langle e_I, \mathbf{a} \shuffle \mathbf{b} \rangle = \langle e_I, \Pi_{(|I|)}(\mathbf{a}) \Pi_{(|I|)}(\mathbf{b}) \rangle.
\end{equation*}
For all $\mathbf{a} \in G((\R^d))$, the \emph{shuffle product property} holds, i.e., for two multi-indices $I = (i_1, \ldots, i_{|I|})$, $J = (j_1, \ldots, j_{|J|})$, it holds that
\begin{equation*}
  \langle e_I, \mathbf{a} \rangle \langle e_J, \mathbf{a} \rangle = \langle e_I \shuffle e_J, \mathbf{a} \rangle.
\end{equation*}

\subsection{Essentials on rough path theory}

A \emph{partition} $\mathcal{P}$ of an interval $[s,t]$ is a finite set of points between and including the points $s$ and $t$, i.e., $\mathcal{P} = \{s = u_0 < u_1 < \cdots < u_N = t\}$ for some $N \in \N$, and its mesh size is denoted by $|\mathcal{P}|:= \max\{|u_{i+1} - u_i| \, : \, i = 0, \ldots, N-1\}$.

Throughout, we let $T > 0$ be a fixed finite time horizon. We let $\Delta_T := \{(s,t) \in [0,T]^2 \, : \, s \leq t\}$ denote the standard $2$-simplex.

We shall write $a \lesssim b$ to mean that there exists a constant $C > 0$ such that $a \leq Cb$. The constant $C$ may depend on the normed space, e.g.~through its dimension or regularity parameters.

For a normed space $(E,|\cdot|)$, we let $C([0,T];E)$ denote the set of continuous paths from $[0,T]$ to $E$. For $X \in C([0,T];E)$, the supremum seminorm of the path~$X$ is given by 
\begin{equation*}
  \| X\|_{\infty}:= \sup_{t\in [0,T]} |X_t|,
\end{equation*}
and for $p \geq 1$, the $p$-variation of the path $X$ is given by
\begin{equation*}
  \|X\|_p := \|X\|_{p,[0,T]} \qquad \text{with} \qquad \|X\|_{p,[s,t]} := \bigg(\sup_{\mathcal{P}\subset[s,t]} \sum_{[u,v]\in \mathcal{P}} |X_v - X_u|^p \bigg)^{\frac{1}{p}}, \quad (s,t) \in \Delta_T,
\end{equation*}
where the supremum is taken over all possible partitions $\mathcal{P}$ of the interval $[s,t]$. We recall that, given a path $X$, we have that $\|X\|_p < \infty$ if and only if there exists a control function $c$ such that\footnote{Here and throughout, we adopt the convention that $\frac{0}{0} := 0$.}
\begin{equation*}
  \sup_{(u,v) \in \Delta_T} \frac{|X_v - X_u|^p}{c(u,v)} < \infty.
\end{equation*}
We write $C^{p\textup{-var}} = C^{p\textup{-var}}([0,T];E)$ for the space of paths $X \in C([0,T];E)$ which satisfy $\|X\|_p < \infty$. Moreover, for a path $X \in C([0,T];\R^d)$, we will often use the shorthand notation:
\begin{equation*}
  X_{s,t} := X_t - X_s, \qquad \text{for} \quad (s,t) \in \Delta_T.
\end{equation*}
For $r \geq 1$ and a two-parameter function $\X^{(2)} \colon \Delta_T \to E$, we further define
\begin{equation*}
  \|\X^{(2)}\|_r := \|\X^{(2)}\|_{r,[0,T]} \qquad \text{with} \qquad \|\X^{(2)}\|_{r,[s,t]} := \bigg(\sup_{\mathcal{P} \subset [s,t]} \sum_{[u,v] \in \mathcal{P}} |\X^{(2)}_{u,v}|^r\bigg)^{\frac{1}{r}}, \quad (s,t) \in \Delta_T.
\end{equation*}
We write $C_2^{r\textup{-var}} = C_2^{r\textup{-var}}(\Delta_T;E)$ for the space of continuous functions $\X^{(2)} \colon \Delta_T \to E$ which satisfy $\|\X^{(2)}\|_r < \infty$.

\medskip 

For $p \in [2,3)$, a pair $\bX = (X,\X^{(2)})$ is called a \emph{(continuous) rough path} over $\R^d$ if
\begin{enumerate}
  \item[(i)] $X \in C^{p\textup{-var}}([0,T];\R^d)$ and $\X^{(2)} \in C_2^{\p\textup{-var}}(\Delta_T;\R^{d \times d})$, and
  \item[(ii)] Chen's relation: $\X^{(2)}_{s,t} = \X^{(2)}_{s,u} + \X^{(2)}_{u,t} + X_{s,u} \otimes X_{u,t}$ holds for all $0 \leq s \leq u \leq t \leq T$.
\end{enumerate}
In component form, condition (ii) states that $(\X^{(2)})^{ij}_{s,t} = (\X^{(2)})^{ij}_{s,u} + (\X^{(2)})^{ij}_{u,t} + X^i_{s,u} X^j_{u,t}$ for every $i$ and $j$. We will denote the space of rough paths by $\cC^p = \cC^p([0,T];\R^d)$. On the space $\cC^p([0,T];\R^d)$, we use the natural seminorm
\begin{equation*}
  \|\bX\|_{p} := \|\bX\|_{p,[0,T]} \qquad \text{with} \qquad \|\bX\|_{p,[s,t]} := \|X\|_{p,[s,t]} + \|\X^{(2)}\|_{\p,[s,t]}
\end{equation*}
for $(s,t) \in \Delta_T$.

\medskip

Let $p \in (2,3)$ and $q > 0$ such that $2/p + 1/q > 1$, and $X \in C^{p\text{-var}}([0,T];\R^d)$. We say that a pair $(Y,Y')$ is a \emph{controlled path} (with respect to $X$), if
\begin{equation*}
  Y \in C^{p\text{-var}}([0,T];\R^{d\times n}), \quad Y' \in C^{q\text{-var}}([0,T];\mathcal{L}(\R^d;\R^{d\times n})), \quad \text{and} \quad R^Y \in C^{r\text{-var}}_2(\Delta_T;\R^{d\times n}),
\end{equation*}
where $R^Y$ is defined by
\begin{equation*}
  Y_{s,t} = Y'_s X_{s,t} + R^Y_{s,t} \qquad \text{for all} \quad (s,t) \in \Delta_T,
\end{equation*}
and $1/r = 1/p + 1/q$. We write $\mathscr{C}^{p,q}_X = \mathscr{C}^{p,q}_X([0,T];\R^{d\times n})$ for the space of $\R^{d\times n}$-valued controlled paths, which becomes a Banach space when equipped with the norm
\begin{equation*}
  (Y,Y') \mapsto |Y_0| + |Y'_0| + \|Y'\|_{q,[0,T]} + \|R^Y\|_{r,[0,T]}.
\end{equation*}
When $p = q$, $r = \p$, we write $\mathscr{C}^p_X = \mathscr{C}^{p,\p}_X$.

\medskip

Further, for $p \geq 1$, and $N \in \N$, the $p$-variation of $\X^{N} \colon [0,T] \to T^N(\R^d)$ is given by
\begin{equation*}
  \|\X^{N}\|_{p,[s,t]} := \max_{1 \leq m \leq N} \sup_{\cP \subset [s,t]} \bigg (\sum_{[u,v] \in \cP} |\Pi_{(m)}(\X^{N}_{u,v})|^{\frac{p}{m}}\bigg)^{\frac{m}{p}}, \qquad (s,t) \in \Delta_T,
\end{equation*}
where now $\X^{N}_{s,t} := (\X^{ N}_s)^{-1} \otimes \X^{ N}_t$, $(s,t) \in \Delta_T$, and we write $\|\X^{N}\|_{p} := \|\X^{N}\|_{p,[0,T]}$.

For $\X^{N}, \tbbX^{N} \colon [0,T] \to T^N(\R^d)$, we define the $p$-variation distance
\begin{equation*}
  \|\X^{N}; \tbbX^{N}\|_{p,[s,t]} := \|\X^{ N} - \tbbX^{N}\|_{p,[s,t]}, \qquad (s,t) \in \Delta_T,
\end{equation*}
and we write $\|\X^{N}; \tbbX^{N}\|_{p} = \|\X^{N}; \tbbX^{N}\|_{p,[0,T]}$.

Moreover, a continuous function $\X^{N} \colon [0,T]\to T^N(\R^d)$ is called a \emph{multiplicative functional} if $\X^0_{s,t} = 1$ and Chen's relation holds:
\begin{equation*}
  \X^{N}_{s,u}\otimes \X^{N}_{u,t}=\X^{N}_{s,t},
  \qquad 0\le s\le u\le t\le T.
\end{equation*}

\medskip

We equip $G^N(\R^d)$ with the (inhomogeneous) subspace topology of $T^N(\R^d)$. In the literature, the (homogeneous) $p$-variation of a $G^{N}(\R^d)$-valued path is often defined in terms of the Carnot--Carath{\'e}odory metric, see e.g.~\cite[Chapter~8]{Friz2010}. This is consistent because the induced topology on $G^N(\R^d)$ coincides with the one induced by the Carnot--Carath{\'e}odory metric, see e.g.~\cite[Section~8.1.2 and 8.1.3]{Friz2010}.

\medskip

A continuous path $\X^{\lfloor p \rfloor} \colon [0,T] \to G^{\lfloor p \rfloor}(\R^d)$ is called a \emph{weakly geometric rough path}, if $\X^{\lfloor p \rfloor}_0 = \mathbf{1}$ and $\|\mathbf{1};\X^{\lfloor p \rfloor}\|_{p} < \infty$, where $\mathbf{1} := (1,0,\dots,0) \in T^{\lfloor p \rfloor}(\R^d)$. We will denote the space of weakly geometric continuous rough paths by $C_o^{p\textup{-var}} = C_o^{p\textup{-var}}([0,T];G^{\lfloor p \rfloor}(\R^d))$ and equip it with the distance $\|\cdot \hspace{1pt}; \cdot\|_p$.

\medskip

An algebraic condition for a rough path to be weakly geometric is that the symmetric part of the rough path lift is determined by the increments of the path.

\begin{lemma}\label{lemma: weakly geometric rough path}
  Let $p \in (2,3)$. Let $(X,\X^{(2)}) \in \cC^p([0,T];\R^d)$ be a continuous rough path such that $\mathbb{S}(\X^{(2)}_{0,t}) = \frac{1}{2} X_{0,t} \otimes X_{0,t}$, $t \in [0,T]$, where we consider the decomposition into the symmetric and the antisymmetric part given by 
  \begin{equation*}
    \X^{(2)}_{0,t} = \mathbb{S}(\X^{(2)}_{0,t}) + \mathbb{A}(\X^{(2)}_{0,t}) := \frac{1}{2} (\X^{(2)}_{0,t} + (\X^{(2)}_{0,t})^\top) + \frac{1}{2} (\X^{(2)}_{0,t} - (\X^{(2)}_{0,t})^\top),
  \end{equation*}
  where $(\cdot)^\top$ denotes matrix transposition. Then $\X^{2}$ is a weakly geometric rough path, i.e., $\X^{2} \in C^{p\textup{-var}}_o$, where $\X^{2}$ is defined by 
  \begin{equation*}
	\X^{2}_t := (1, X_{0,t}, \X^{(2)}_{0,t}), \qquad t \in [0,T].
  \end{equation*}
\end{lemma}

\begin{proof}
  Recall that $G^2(\R^d) = \exp_\otimes^2(\mathfrak{g}^2(\R^d))$, where $\mathfrak{g}^2(\R^d) = \{0\} \oplus \R^d \oplus [\R^d,\R^d]$. It holds that $[\R^d, \R^d] = \spn \{e_i \otimes e_j - e_j \otimes e_i: 1 \leq i, j \leq d \}$. Therefore $[\R^d,\R^d]$ equals the set of antisymmetric $d \times d$-matrices and it follows that, for any $t \in [0,T]$,
  \begin{equation*}
    \X^{2}_t = \bigg(1, X_{0,t}, \frac{1}{2} X_{0,t} \otimes X_{0,t} + \mathbb{A}(\X^{(2)}_{0,t})\bigg) = \exp_\otimes^2(0, X_{0,t}, \mathbb{A}(\X^{(2)}_{0,t})) \in G^2(\R^d).
  \end{equation*}
  Finally, since $(X,\X^{(2)}) \in \cC^p([0,T];\R^d)$, it particularly holds that $\|\mathbf{1};\X^{2}\|_p < \infty$.
\end{proof}

\begin{remark}
  This condition is a consequence of ``first order calculus'' and therefore valid in the context of stochastic Stratonovich integration.
\end{remark}

\subsection{Definition and properties of signatures}\label{subsec: signature of non-weakly geometric rough paths}

\sloppy By Lyons' extension theorem, see e.g.~\cite[Theorem~3.7]{Lyons2007}, any multiplicative functional $\X^m\colon [0,T]\to T^m(\R^d)$ with finite $p$-variation for $m \geq \lfloor p \rfloor$---i.e., $\X^{(i)}$, $i \leq m$, is of finite $p/i$-variation, controlled by a control function $c$---admits a unique extension to a multiplicative functional $\X\colon [0,T]\to T((\R^{d}))$ with finite $p$-variation, for $p \ge 1$. More precisely, for any $N > \lfloor p \rfloor$, there exists a unique continuous function $\X^{(N)}\colon[0,T]\to (\R^d)^{\otimes N}$ such that,
\begin{equation*}
  (1,X_{0,\cdot,},\X^{(2)},\ldots,\X^{(\lfloor p\rfloor)},\ldots,\X^{(N)},\ldots)\in T((\R^d))
\end{equation*}
is a multiplicative functional with finite $p$-variation, i.e., $\Pi_N(\X)$ is of finite $p$-variation for any $N$, controlled by $c$, and is called \emph{Lyons' extension}.

In particular, any weakly geometric rough path admits a unique extension to a path of finite $p$-variation with values in $G^N(\R^d)$ with $N > \lfloor p\rfloor$, see e.g.~\cite[Theorem~9.5]{Friz2010}, which allows us to define the signature of $X$ as follows:

\begin{definition}\label{def: Lyons' lift of weakly geometric rough path}
  Let $p\ge 1$ and $\X^{o,\lfloor p \rfloor} \in C_o^{p\textup{-var}}([0,T];G^{\lfloor p \rfloor}(\R^d))$. The \emph{signature} of $X$ is defined as the unique path
  \begin{equation*}
    \X^o \colon [0,T] \to G((\R^d)), 
  \end{equation*}
  such that for all $N > \lfloor p \rfloor$, $\Pi_{N}(\X^o) = \X^{o,N}$, where $\X^{o, N}$ denotes the extension of $\X^{o,\lfloor p \rfloor}$ in $G^N(\R^d)$. In particular, $\X^o$ is the unique path extension of $\X^{o,\lfloor p \rfloor}$ specified by Lyons' extension theorem.
\end{definition}

Now, let $\bX = (X,\X^{(2)}) \in \cC^p([0,T];\R^d)$, $p\in(2,3)$, be a rough path, and let $\X$ be Lyons' extension of $\bX$ to $T((\R^d))$, i.e., $\X \colon [0,T] \to T((\R^d))$, which by Proposition~\ref{prop: Lyons extension coincides with iterated integrals}  coincides with the collection of iterated rough integrals of controlled paths with respect to $\bX$, that is, for $N \geq 3$,
\begin{equation*}
  \X_{s,t}^{(N)}=\int_s^t\X_{s,r}^{(N-1)} \otimes \d \bX_r \qquad (s,t) \in \Delta_T.
\end{equation*}
The rough integral is defined in Lemma~\ref{lemma: integration against controlled path} and Remark~\ref{rem: tensored controlled path}, considering $G = X$ so that $G' = I_d$ is the identity matrix and $R^G = 0$, and the integral reduces to the classical notion of the rough integral of the controlled path $(F,F')$ against the rough path $\bX$.

Further, let $[\bX]_t := X_{0,t} \otimes X_{0,t} - (\X^{(2)}_{0,t} + (\X^{(2)}_{0,t})^\top)$, $t \in [0,T]$, be the rough path bracket of $\bX$, and set
\begin{equation*}
  Q(X) := ([\bX]^{11}, \dots, [\bX]^{1d}, [\bX]^{22}, \dots, [\bX]^{2d}, \dots, [\bX]^{dd}).
\end{equation*}

We define $\hX := (\cdot,X,Q(X)) \in C^{p\textup{-var}}([0,T];\R^\hd)$, with $\hd=1+d+\frac{d(d+1)}{2}$, and note that $\hX$ can be canonically lifted to a rough path $\widehat{\bX} = (\hX,\hbbX^{(2)}) \in \cC^p([0,T];\R^\hd)$ since $Q(X)$ is a path of finite $p/2$-variation so that $(\hbbX^{(2)}_{s,t})^{ij} := \int_s^t \hX^i_{s,r} \dd \hX^j_r$, for $i,j \notin \{1,\dots,d\}$, exist as Young integrals, and $(\hbbX^{(2)}_{s,t})^{ij} :=(\X^{(2)}_{s,t})^{ij}$, for $i,j = 1, \dots, d$. We denote by $\hbbX$ Lyons' extension of $\hbX$.

We write $(e_0, e_1, \ldots, e_d, \epsilon_{11}, \ldots, \epsilon_{1d}, \epsilon_{22}, \ldots, \epsilon_{2d}, \ldots, \epsilon_{dd})$ $:=(e_0, e_1, \ldots, e_d, e_{d+1}, \ldots, e_{\hd-1})$ for the canonical basis of $\R^\hd$, i.e., we use the index $0$ to denote the time component, and $\epsilon_{ij}$ for the component of $\hX$ referring to $[\bX]^{ij}$, so that $\langle \epsilon_{ij}, \hbbX_t \rangle := [\bX]^{ij}_t$, $i, j = 1, \ldots, d$, $i \leq j$, $t \in [0,T]$. We set $\epsilon_{ji}:=\epsilon_{ij}$ and observe that $\langle e_I\otimes e_i,\hbbX_t\rangle=\int_0^t\langle e_{I},\hbbX_s\rangle\dd \hX_s^{i}$, for $i\in \{0,d+1,\ldots,{\hd-1}\}$, is well-defined as an integral with respect to $\hX^i$.

We further note that $t \mapsto \langle e_0, \hbbX_t \rangle$ is strictly monotonically increasing. This is necessary so that $\hbbX_T$ uniquely characterizes $\hbX$, see e.g.~\cite{Hambly2010,Boedihardjo2016}.

See the proof of condition~(iii) in Theorem~\ref{thm: UAT geometric} for a similar argument for signatures $\hbbX^{o}$ of time extended weakly geometric rough paths, and the proof of condition (iii) in Theorem~\ref{thm: UAT non-geometric} for signatures $\hbbX$ of general rough paths extended by time and the rough path bracket terms.

Further, let $\hbbX^o$ be Lyons' extension of $(1,\hX,\hbbX^{(2)} + \frac{1}{2}[\hbX])$, that then is a group-like valued path, i.e., $\hbbX^o \colon [0,T] \to G((\R^{\hd}))$, see Definition~\ref{def: Lyons' lift of weakly geometric rough path}.

\medskip

Extending the path $X$ to $\hX$ by the rough path bracket terms yields that the components of the signature of the non-weakly geometric rough path $\hbbX$ can be represented as linear functionals of the signature of the weakly geometric rough path $\hbbX^o$.

\begin{proposition}\label{prop: linear functional}
  Let $\bX = (X,\X^{(2)}) \in \cC^p([0,T];\R^d)$, $p \in (2,3)$, be a rough path, and $\hX := (\cdot,X,Q(X))$ be the path extended by time and the rough path bracket terms, and $\hbX = (\hX,\hbbX^{(2)}) \in \cC^p([0,T];\R^{\hd})$ be the corresponding rough path. Further, let $\hbbX$ be Lyons' extension of $(1,\hX,\hbbX^{(2)}) \in C_o^{p\textup{-var}}([0,T];T^2(\R^{\hd}))$. Then, for any multi-index $I$, there exists $\ell^{I} \in T(\R^\hd)$ such that 
  \begin{equation*}
    \langle e_I, \hbbX_t \rangle = \langle \ell^{I}, \hbbX^o_t \rangle, \qquad t \in [0,T],
  \end{equation*}
  where $\hbbX^o \colon [0,T] \to G((\R^{\hd}))$ denotes the group-like valued path that is Lyons' extension of $(1,\hX,\hbbX^{(2)} + \frac{1}{2}[\hbX]) \in C_o^{p\textup{-var}}([0,T];G^2(\R^{\hd}))$, and $\ell^I$ is recursively defined by 
  \begin{equation*}
    \ell^I := \ell^{I'} \otimes e_{i_{|I|}} - \frac{1}{2} \ell^{(I')'} \otimes \epsilon_{i_{|I'|}i_{|I|}},
  \end{equation*}
  with $\ell^\emptyset := e_\emptyset$, $\ell^{(i_1)} := e_{(i_1)}$, and $\epsilon_{i_{|I'|}i_{|I|}} := 0$, for $i_{|I'|}, i_{|I|} \notin \{1, \dots, d\}$.
\end{proposition}

\begin{proof}
  We show the statement for $t = T$. For $|I| = 0$ and $|I| = 1$, considering $\ell^{I} := e_I$, we have that $\langle e_I, \hbbX_T \rangle = \langle \ell^I, \hbbX^o_T \rangle$ by definition of Lyons' extension.
	
  For $|I| = 2$, i.e., $I = (i_1,i_2)$, $i_1, i_2 \in \{0, \ldots, \hd-1\}$, we obtain that $\ell^I=e_I-\frac{1}{2}\epsilon_{i_1i_2}$.
   
  Now let $|I| = 3$, i.e., $I=(i_1,i_2,i_3)$. Then, taking the limit along any sequence of partitions $\cP$ of $[0,T]$ with vanishing mesh size, we note that
  \begin{align*}
    &\langle e_I, \hbbX_T \rangle \\
    &\quad = \lim_{|\cP| \to 0} \sum_{[u,v] \in \cP} (\hbbX^{(2)}_{0,u})^{i_1i_2} \hX^{i_3}_{u,v} + \hX^{i_1}_{0,u} (\hbbX^{(2)}_{u,v})^{i_2i_3} \\
    &\quad = \lim_{|\cP| \to 0} \sum_{[u,v] \in \cP} (\hbbX^{o,(2)}_{0,u})^{i_1i_2} - \frac{1}{2} [\hbX]^{i_1i_2}_{0,u})\hX^{i_3}_{u,v} + \hX^{i_1}_{0,u} ((\hbbX^{o,(2)}_{u,v})^{i_2i_3} - \frac{1}{2} [\hbX]^{i_2i_3}_{u,v}) \\
    &\quad = \langle e_I, \hbbX^o_T \rangle - \frac{1}{2} \int_0^T [\hbX]^{i_1i_2}_{0,t} \dd \hX^{i_3}_t - \frac{1}{2} \int_0^T \hX^{i_1}_{0,t} \dd [\hbX]^{i_2i_3}_t \\
    &\quad = \langle e_I, \hbbX^o_T \rangle - \frac{1}{2} \langle \epsilon_{i_1i_2} \otimes e_{i_3}, \hbbX^o_T \rangle - \frac{1}{2} \langle e_{i_1} \otimes \epsilon_{i_2i_3}, \hbbX^o_T \rangle \\
    &\quad = \langle \ell^I, \hbbX^o_T \rangle,
  \end{align*}
  for $\ell^I := e_I - \frac{1}{2} \epsilon_{i_1i_2} \otimes e_{i_3} - \frac{1}{2} e_{i_1} \otimes \epsilon_{i_2i_3} = \ell^{(i_1,i_2)} \otimes e_{i_3} - \frac{1}{2} e_{i_1} \otimes \epsilon_{i_2i_3}$, where the latter two integrals exist as Young integrals because $[\hbX]$ is a path of finite $p/2$-variation. We apply an inductive argument: Assuming that the claim holds true for any $n < N$, for $N > 3$, we now let $n \geq N$. Let $I$ be a multi-index with entries in $\{0, \dots, \hd-1\}$ of length $n$. Then,
  \begin{align*}
    &\langle e_I, \hbbX_T \rangle \\
    &\quad = \lim_{|\cP| \to 0} \sum_{[u,v] \in \cP} \langle e_{I'}, \hbbX_u \rangle \hX^{i_n}_{u,v} + \langle e_{(I')'}, \hbbX_u \rangle (\hbbX^{(2)}_{u,v})^{i_{n-1} i_n} \\
    &\quad = \lim_{|\cP| \to 0} \sum_{[u,v] \in \cP} \langle \ell^{I'}, \hbbX^o_u \rangle  \hX^{i_n}_{u,v} + \langle \ell^{(I')'}, \hbbX^o_u \rangle ((\hbbX^{o,(2)}_{u,v})^{i_{n-1}i_n} - \frac{1}{2}[\hbX]_{u,v}^{i_{n-1} i_n}) \\
    &\quad= \lim_{|\cP|\to 0} \sum_{[u,v] \in \cP} \langle \ell^{(I')'} \otimes e_{i_{n-1}}, \hbbX^o_u \rangle \hX^{i_n}_{u,v} - \frac{1}{2} \langle \ell^{((I')')'} \otimes \epsilon_{i_{n-2}i_{n-1}}, \hbbX^o_u \rangle \hX^{i_n}_{u,v} \\
    &\qquad + \langle \ell^{(I')'}, \hbbX^o_u \rangle (\hbbX^{o,(2)})_{u,v}^{i_{n-1}i_n} - \frac{1}{2} \lim_{|\cP|\to 0} \sum_{[u,v] \in \cP} \langle \ell^{(I')'}, \hbbX^o_u \rangle [\hbX]^{i_{n-1}i_n}_{u,v} \\
    &\quad = \langle \ell^{(I')'} \otimes e_{i_{n-1}} \otimes e_{i_n}, \hbbX^o_T \rangle - \frac{1}{2} \langle \ell^{((I')')'} \otimes \epsilon_{i_{n-2}i_{n-1}} \otimes e_{i_n}, \hbbX^o_T \rangle - \frac{1}{2} \langle \ell^{(I')'} \otimes \epsilon_{i_{n-1}i_n}, \hbbX^o_T \rangle \\
    &\quad = \langle \ell^{I'} \otimes e_{i_{n}} - \frac{1}{2} \ell^{(I')'} \otimes \epsilon_{i_{n-1}i_n}, \hbbX^o_T \rangle,
  \end{align*}
  where we again used that the integral w.r.t.~$[\hbX]$ exists as a Young integral because $[\hbX]$ is a path of finite $p/2$-variation.
\end{proof}

The signature of a weakly geometric rough path $(1,X,\X^{o,(2)})$ is a group-like valued path and therefore satisfies the shuffle property. It turns out that extending the path $X$ to $\hX$ by the rough path bracket terms yields that the signature defined via Lyons' extension admits the so-called quasi-shuffle property, see~\cite{Hoffman2000}. This allows us to prove the universal approximation theorem for the signature of non-weakly geometric rough paths in Theorem~\ref{thm: UAT non-geometric}.

\begin{definition}\label{def: quasi shuffle}
  For every two multi-indices $I$, $J$ with entries in $\{0,\ldots,\hd-1\}$, we define the \emph{quasi-shuffle product} by
  \begin{equation*}
    e_I \widetilde{\shuffle} e_J = (e_{I'} \widetilde{\shuffle} e_J) \otimes e_{i_{|I|}} + (e_I \widetilde{\shuffle} e_{J'}) \otimes e_{j_{|J|}} + (e_{I'} \widetilde{\shuffle} e_{J'}) \otimes \epsilon_{i_{|I|} j_{|J|}},
  \end{equation*}
  with $e_I \widetilde{\shuffle} e_\emptyset = e_\emptyset \widetilde{\shuffle} e_I = e_I$.
\end{definition}

\begin{proposition}[Quasi-shuffle property]\label{prop: quasi shuffle property}
  Let $I, J$ be two multi-indices with entries in  $\{0, \dots, \hd-1\}$. Let $\bX = (X,\X^{(2)}) \in \cC^p([0,T];\R^d)$, $p \in (2,3)$, be a rough path, and $\hX := (\cdot,X,Q(X))$ be the extended path. Further, let $\hbbX$ be Lyons' extension of $(1,\hX,\hbbX^{(2)}) \in C_o^{p\textup{-var}}([0,T];T^2(\R^{\hd}))$. Then,
  \begin{equation*}
    \langle e_I, \hbbX_t \rangle \langle e_J, \hbbX_t \rangle = \langle e_I \widetilde{\shuffle} e_J, \hbbX_t \rangle, \qquad t \in [0,T].
  \end{equation*}
\end{proposition}

\begin{remark}
  For the quasi-shuffle property to hold, we actually do not need to extend the path $X$ by the time-component but only by its rough path bracket terms $Q(X)$.
\end{remark}

\begin{proof}
  Since $\langle e_\emptyset,\hbbX\rangle=1$, the statement immediately holds true for $I=\emptyset$ or $J=\emptyset$. Now let $|I|+|J|=2$, $I, J \neq \emptyset$, i.e., $I=(i)$, $J=(j)$, $i,j\in\{0,\ldots,\hd-1\}$. Then, for $t\in[0,T]$, we have by Proposition~\ref{prop: linear functional} that
  \begin{align*}
    \langle e_i,\hbbX_t\rangle \langle e_j,\hbbX_t\rangle&=\langle e_i,\hbbX^o_t\rangle \langle e_j,\hbbX^o_t\rangle\\
    &=\langle e_i \shuffle e_j,\hbbX^o_t\rangle\\
    &=\langle e_{(i,j)},\hbbX^o_t\rangle + \langle e_{(j,i)},\hbbX^o_t\rangle\\
    &=\langle e_{(i,j)},\hbbX_t\rangle+\frac{1}{2}\langle \epsilon_{ij},\hbbX_t\rangle+ \langle e_{(j,i)} ,\hbbX_t\rangle +\frac{1}{2} \langle \epsilon_{ji},\hbbX_t\rangle\\
    &=\langle (e_i\widetilde\shuffle e_\emptyset)\otimes e_j,\hbbX_t\rangle+\langle (e_{\emptyset}\widetilde\shuffle e_j)\otimes e_i,\hbbX_t\rangle+\langle (e_\emptyset\widetilde\shuffle e_\emptyset)\otimes \epsilon_{ij},\hbbX_t\rangle\\
    &=\langle e_i\widetilde\shuffle e_j,\hbbX_t\rangle,
  \end{align*}
  where we note that $\epsilon_{ji} = \epsilon_{ij}$ for $i \leq j$, and $\epsilon_{ij}=0$ for $i,j\notin\{1,\dots,d\}.$
    
  We apply an inductive argument: We assume that the claim holds true for any $I,J$ such that $|I|+|J|<n$, $n>2$ and let $I,J$ be such that $|I|+|J|\le n$, $I,J\neq \emptyset$.
    
  We start by noting that $\langle e_I, \hbbX \rangle$ is a controlled path w.r.t.~$\hX$.
    
  We consider $(Y,Y') \in \mathscr{C}^p_{\hX}$ to be the controlled path given by $Y \in C^{p\textup{-var}}([0,T];\mathcal{L}(\R^\hd;\R^2))$, where
  \begin{equation*}
    Y^{mn} = 
    \begin{cases}
      \langle e_{I'}, \hbbX \rangle, & \text{if } m = 1,\, n = i_{|I|}, \\
      \langle e_{J'}, \hbbX \rangle , & \text{if } m = 2,\, n = j_{|J|}, \\
      0, & \text{otherwise}.
    \end{cases}
  \end{equation*}
  Since $\hbbX^{N}_t := \int_0^t \hbbX^{N-1}_r \otimes \d \hbX_r$ is a rough integral against $\hbX = (\hX,\hbbX^{(2)}) \in \cC^p([0,T];\R^{\hd})$, the projection onto the tensor product $(\R^d)^{\otimes N}$, $\Pi_N(\hbbX^N)$, is the rough integral of the controlled path $(\Pi_{N-1}(\hbbX^N), \Pi_{N-2}(\hbbX^N))$ against $\hbX$. That is, for any multi-index $I$, we have that
  \begin{equation*}
    \langle e_I, \hbbX_t \rangle = \lim_{|\cP|\to 0} \sum_{[u,v] \in \cP \cap [0,t]} \langle e_{I'}, \hbbX_u \rangle X^{i_{|I|}}_{u,v} + \langle e_{(I')'}, \hbbX_u \rangle \hbbX^{(2),i_{|I'|}i_{|I|}}_{u,v}.
  \end{equation*}
  Particularly, it then is $Y'_t \in \mathcal L(\R^d;\mathcal L(\R^d;\R^2))$, given by $(Y'_t)^{1ij} = \langle e_{(I')'}, \hbbX_t \rangle$ for $i = i_{|I'|}$, $j = i_{|I|}$, and $(Y'_t)^{2ij} = \langle e_{(J')'}, \hbbX_t \rangle$ for $i = j_{|J'|}$, $j = j_{|J|}$, $(Y'_t)^{mij} = 0$ otherwise, $m = 1, 2$, $i, j = 0, \dots, \hd-1$. This gives us that
  \begin{align*}
    Z_t = \lim_{|\cP|\to 0} \sum_{[u,v] \in \cP \cap [0,t]} Y_u X_{u,v} + Y'_u \hbbX^{(2)}_{u,v}  = \lim_{|\cP| \to 0} \sum_{[u,v] \in \cP \cap [0,t]} \mathcal{Z}_{u,v},
  \end{align*}
  with 
  \begin{equation*}
    \mathcal{Z}_{u,v}:=(\langle e_{I'}, \hbbX_u \rangle X^{i_{|I|}}_{u,v} + \langle e_{(I')'}, \hbbX_u \rangle \hbbX^{(2),i_{|I'|}i_{|I|}}_{u,v}, \langle e_{J'}, \hbbX_u \rangle X^{j_{|J|}}_{u,v} + \langle e_{(J')'}, \hbbX_u \rangle \hbbX^{(2),j_{|J'|}j_{|J|}}_{u,v})^\top,
  \end{equation*}
  where we have used that $(Y'_u \hbbX^{(2)}_{u,v})^m = \sum_{i} \sum_{j} (Y'_u)^{mij} \hbbX_{u,v}^{(2),ij}$.
    
  Then, $Z:= \int Y \dd \hbX$ is the rough integral of $Y$ against $\hbX$, which equals $Z = (Z^1, Z^2)^\top =(\langle e_I, \hbbX \rangle, \langle e_J, \hbbX \rangle)^\top$.
       
  As a controlled path, $Z$ can now be canonically lifted to a rough path $\mathbf Z = (Z,\mathbb Z)$, with $\mathbb Z_{s,t} := \int_s^t Z_r \otimes Z_r - Z_s \otimes Z_{s,t}$, $(s,t) \in \Delta_T$, where the integral is defined as in Lemma~\ref{lemma: integration against controlled path}.

  Using the It{\^o}-formula for rough paths, see e.g.~\cite[Proposition~5.8]{Friz2020}, for $f(Z)=Z^1 \cdot Z^2$, we thus obtain that
  \begin{equation}\label{eq: ito formula rough path}
    \langle e_I, \hbbX_t\rangle \langle e_J, \hbbX_t \rangle = \int_0^t (Z^2,Z^1)_r \dd \mathbf Z_r +  \frac{1}{2} [\mathbf Z]^{12}_t + \frac{1}{2} [\mathbf Z]^{21}_t = \int_0^t (Z^2,Z^1)_r \dd Z_r + [\mathbf Z]^{12}_t,
  \end{equation}
  where the latter integral is well defined as a rough integral against the controlled path $(Z,Y)$, see also \cite[Lemma~A.4]{Allan2023a}.

  Due to the associativity of the rough integral, see \cite[Proposition~A.2]{Allan2023a}, we have 
  \begin{equation}\label{eq: quasi shuffle 1}
    \begin{split}
    \int_0^t (Z^2, Z^1)_r \dd Z_r &= \int_0^t (Z^2,Z^1)_r Y_r \dd \hbX_r \\
    &= \int_0^t (0, \dots, 0, \langle e_J, \hbbX_r \rangle \langle e_{I'}, \hbbX_r \rangle, 0, \dots, 0, \langle e_I, \hbbX_r \rangle \langle e_{J'}, \hbbX_r \rangle, 0, \dots, 0) \dd \hbX_r\\
    &= \int_0^t (0, \dots, 0, \langle e_J\widetilde\shuffle e_{I'}, \hbbX_r \rangle, 0, \dots, 0, \langle e_I\widetilde\shuffle e_{J'}, \hbbX_r \rangle, 0, \dots, 0) \dd \hbX_r \\
    &=\langle (e_J\widetilde\shuffle e_{I^\prime})\otimes e_{i_{|I|}},\hbbX_t\rangle+\langle (e_I\widetilde\shuffle e_{J^\prime})\otimes e_{j_{|J|}},\hbbX_t\rangle,
    \end{split}
  \end{equation}
  where we used the induction hypothesis in the second last step. 

  By definition, it holds that $e_J \widetilde{\shuffle} e_{I'} = (e_J \widetilde{\shuffle} e_{(I')'}) \otimes e_{i_{|I'|}} + (e_{J'} \widetilde{\shuffle} e_{I'}) \otimes e_{j_{|J|}} + (e_{J'} \widetilde{\shuffle} e_{(I')'}) \otimes \epsilon_{j_{|J|} i_{|I'|}}$. So setting $U_t := (0, \dots, 0, \langle e_J\widetilde\shuffle e_{I'}, \hbbX_t \rangle, 0, \dots, 0, \langle e_I \widetilde\shuffle e_{J'}, \hbbX_t \rangle, 0, \dots, 0)$, we observe that $U'_t \in \mathcal L(\R^{\hd}; \mathcal{L}(\R^{\hd};\R))$, given by $(U'_t)^{i_{|I'|}j} = \langle e_J \widetilde{\shuffle} e_{(I')'}, \hbbX_t \rangle$, $(U'_t)^{j_{|J|}j} = \langle e_{J'} \widetilde{\shuffle} e_{I'}, \hbbX_t \rangle$, and $(U'_t)^{kj} = \langle e_{J'} \widetilde{\shuffle} e_{(I')'}, \hbbX_t \rangle$, for $k$ such that $e_k = \epsilon_{j_{|J|} i_{|I'|}}$, for $j = i_{|I|}$. Similarly for $j = j_{|J|}$, and otherwise it is $(U'_t)^{ij} =0$. By a similar line of argument as above, it then follows the last step.
    
  Further, by \cite[Lemma~B.1]{Allan2023a}, we have that\footnote{In writing $Y_u \otimes Y_u$, we technically mean the $4$-tensor whose $ijk\ell$ component is given by $[Y_u \otimes Y_u]^{ijk\ell} = (Y_u)^{ij} (Y_u)^{k\ell}$, and we interpret the ``multiplication'' $(Y_u \otimes Y_u) [\hbX]_{u,v}$ as the matrix whose $ik$ component is given by $[(Y_u \otimes Y_u) [\hbX]_{u,v}^{ik} = \sum_{j=1}^d \sum_{\ell=1}^d (Y_u)^{ij} (Y_u)^{k\ell} [\hbX]^{j\ell}_{u,v}$.}
  \begin{equation*}
    [\mathbf Z]_t = \int_0^t (Y_r \otimes Y_r) \dd [\hbX]_r = \lim_{|\cP|\to 0} \sum_{[u,v] \in \cP} (Y_u \otimes Y_u) [\hbX]_{u,v},
  \end{equation*}
  and by definition of $Y$, we obtain
  \begin{equation}\label{eq: quasi shuffle 2}
    [\mathbf Z]_t^{12} = [\mathbf Z]_t^{21} = \int_0^t \langle e_{I'}, \hbbX_r \rangle \langle e_{J'}, \hbbX_r \rangle \dd [\hbX]^{i_{|I|}j_{|J|}}_r = \langle (e_{I'} \widetilde{\shuffle} e_{J'}) \otimes \epsilon_{i_{|I|}j_{|J|}}, \hbbX_t \rangle,
  \end{equation}
  where we again used the induction hypothesis in the last step. 
     
  Finally, combining \eqref{eq: ito formula rough path}, \eqref{eq: quasi shuffle 1} and \eqref{eq: quasi shuffle 2}, yields
  \begin{align*}
    \langle e_I,\hbbX_t\rangle \langle e_J,\hbbX_t\rangle&= \langle (e_J\widetilde\shuffle e_{I^\prime})\otimes e_{i_{|I|}},\hbbX_t\rangle+\langle (e_I\widetilde\shuffle e_{J^\prime})\otimes e_{j_{|J|}},\hbbX_t\rangle+\langle (e_{I'} \widetilde{\shuffle} e_{J'}) \otimes \epsilon_{i_{|I|}j_{|J|}}, \hbbX_t \rangle\\
    &=\langle e_I\widetilde\shuffle e_J,\hbbX_t\rangle.
  \end{align*}
\end{proof}

\subsection{Universal approximation with signatures of general rough paths}

We now present the pathwise universal approximation theorem of the signature of a general, i.e., not necessarily, weakly geometric rough path via Lyons' lift, as an extension of the classical result to a more general class of signatures. The proof is based on an application of the Stone--Weierstrass theorem, which requires that the linear span of the signature forms an algebra, making use of the quasi shuffle property. For this, we formulate the universal approximation theorem on the subspace of rough paths extended by time and the rough path bracket terms, defined by
\begin{equation*}
  \widehat{\cC}^p([0,T];\R^\hd) := \{(\hX,\hbbX^{(2)}) \in \cC^p([0,T];\R^\hd) : (X,\X^{(2)}) \in \cC^p([0,T];\R^d), \hX = (\cdot,X,Q(X)) \}.
\end{equation*}

\begin{theorem}\label{thm: UAT non-geometric}
  Let $p \in (2,3)$, and let $K \subset \widehat{\cC}^p([0,T];\R^{\hd})$ be a compact subset, bounded with respect to the rough path norm, and consider a continuous function $f \colon K \to \R$. Further for some $L > 0$, let $K_L \subset K$ be the subset defined by
  \begin{align*}
    K_L:= \{\hbX=(\hX,\hbbX^{(2)}) \in K: \, 
        &\|(\hX,\hbbX^{(2)})\|_p + \|[\widehat{\bX}]\|_{\p} \leq L
        \}.
  \end{align*}
  Then, for every $\epsilon > 0$, there exists a linear functional $\ell \in T(\R^{\hd})$ such that 
  \begin{equation*}
    \sup_{\widehat{\bX}\in K_L} |f(\hbX) - \langle \ell, \hbbX_T \rangle| < \epsilon,
  \end{equation*}
  where $\hbbX$ denotes the signature of $\widehat{\bX}$ given by Lyons' extension theorem to $T((\R^\hd))$.
\end{theorem}

\begin{proof}
  The result follows by an application of the Stone--Weierstrass theorem to the set
  \begin{equation*}
	\mathcal A := \spn \{K_L \ni \hbX\mapsto \langle e_I,\hbbX_T \rangle \in \R: I \in \{0,\ldots,\hd-1\}^N, N \in \N_0\}.
  \end{equation*}
  Therefore we have to show that $\mathcal A$
  \begin{enumerate}
    \item[(i)] is a vector subspace of $C(K_L;\R)$,
	\item[(ii)] is a subalgebra and contains a non-zero constant function, and
    \item[(iii)] separates points.
  \end{enumerate}

  \emph{(i):} First, we note that any rough path $\hbX=(\hX,\hbbX^{(2)}) \in \cC^p([0,T];\R^\hd)$ canonically extends to a weakly geometric rough path via
  \begin{align*}
    \iota \colon \hbX \mapsto \hbbX^{o, 2} &:= (1,\hX_{0,\cdot},\mathbb{\hX}^{o,(2)}_{0,\cdot}) \\
    &:= (1,\hX_{0,\cdot},\mathbb{\hX}^{(2)}_{0,\cdot} +\frac{1}{2} [\hbX]),
  \end{align*}
  that is, $\iota(\hbX) \in (C_o^{p\textup{-var}}([0,T];G^2(\R^\hd)),d_{p\textup{-var}})$. Further, we observe that for any $\hbX \in K_L$, it holds that
  \begin{equation*}
    \|\hbbX^{o,(2)}\|_{\p} \leq \|\hbbX^{(2)}\|_\p + \|[\hbX]\|_{\p} \leq L,
  \end{equation*}
  thus we can embed $K_L$ into $\iota(K_L) := \{\iota(\hbX): \hbX\in K_L\}$, which is a subset of the compact set $K_L^2 := \{\hbbX^{o, 2}: \|\mathbf{1}; \hbbX^{o, 2}\|_p \leq L\}$ of $C_o^{p\textup{-var}}([0,T];G^2(\R^\hd))$.

  Because
  \begin{equation*}
    (\cC^p([0,T];\R^\hd), d_{p,\p\textup{-var}}) \ni \hbX \mapsto [\hbX] \in (C^{\p\textup{-var}}([0,T];\R^\hd), \|\cdot \, ; \cdot\|_{\p})
  \end{equation*}
  is continuous with respect to $d_{p,\p\textup{-var}} := \|\cdot - \cdot\|_p +  \|\cdot - \cdot\|_\p$, see Proposition~\ref{prop: rough path to bracket is continuous}, we notice that
  \begin{equation*}
    (K, d_{p,\p\textup{-var}}) \ni \hbX = (\hX,\hbbX^{(2)})\mapsto (1,\hX_{0,\cdot},\hbbX^{(2)}_{0,\cdot} + \frac{1}{2}[\hbX]) \in (C_o^{p\textup{-var}}([0,T];T^2(\R^{\hd})), d_{p\textup{-var}})
  \end{equation*}
  is continuous so that
  \begin{equation*}
    (K_L, d_{p,\p\textup{-var}}) \ni \hbX=(\hX,\hbbX^{(2)}) \mapsto \hbbX^{o,2} \in (\iota(K_L),d_{p\textup{-var}})
  \end{equation*}
  is continuous. By~\cite[Corollary~9.11]{Friz2010}, the map $\hbbX^{o, 2} \mapsto \langle e_I,\hbbX^{o}_T \rangle$ is continuous on bounded sets for every multi-index $I$ with respect to $d_{p\textup{-var}}$. More precisely, the map
  \begin{equation*}
    (\iota(K_L),d_{p\textup{-var}}) \ni \hbbX^{o, 2} \mapsto \hbbX^{o, N} \in (C_o^{p\textup{-var}}([0,T];G^N(\R^\hd)),d_{p\textup{-var}}),
  \end{equation*}
  is continuous with respect to $d_{p\textup{-var}}$, for every $N \geq 3$. Moreover, the evaluation map
  \begin{equation*}
	(C_o^{p\textup{-var}}([0,T];G^N(\R^\hd)),d_{p\textup{-var}}) \ni \hbbX^{o, N} \mapsto \hbbX^{o, N}_T \in (G^N(\R^\hd),\rho)
  \end{equation*}
  is continuous, where $\rho$ denotes the metric induced by the norm on $T_1^N(\R^\hd)$. Here, we used that we can equip $G^N(\R^\hd)$ with the metric $\rho$, see e.g.~\cite[Remark~7.31]{Friz2010}.
    
  Since $\hbbX^{o, N}_T \mapsto \langle e_I, \hbbX^{o}_T \rangle$ is continuous for any multi-index $I$, we can thus conclude that the map
  \begin{equation*}
	(G^N(\R^\hd),\rho) \ni \hbbX^{o, N}_T \mapsto \langle e_I, \hbbX_T^{o} \rangle \in \R
  \end{equation*}
  is continuous, and so is
  \begin{equation*}
    (K_L, d_{p,\p\textup{-var}}) \ni \hbX = (\hX,\hbbX^{(2)}) \mapsto \langle e_I, \hbbX^{o}_T \rangle \in \R.
  \end{equation*}
  By Proposition~\ref{prop: linear functional}, for every multi-index $I\in\{0,\ldots,\hd-1\}^N$, $N\in\N_0$, there exists $\ell^I \in T(\R^\hd)$ such that $\langle e_I, \hbbX_T \rangle = \langle \ell^I, \hbbX^{o}_T \rangle$. This finally yields that
  \begin{equation*}
    (K_L, d_{p,\p\textup{-var}}) \ni \hbX \mapsto \langle e_I, \hbbX_T \rangle \in \R
  \end{equation*}
  is continuous with respect to $d_{p,\p\textup{-var}}$.
   
  \emph{(ii):} By Proposition~\ref{prop: quasi shuffle property}, the quasi-shuffle property holds and thus $\mathcal A$ is a subalgebra. Moreover, since $\langle e_\emptyset, \hbbX_T \rangle = 1$, it contains a non-zero constant function.
	
  \emph{(iii):} For the point separation, let us consider $\hbX=(\hX,\hbbX^{(2)})$, $\widehat{\mathbf{Y}} =(\widehat{Y},\hbbY^{(2)})\in K_L$, with $\hbX \neq \widehat{\mathbf{Y}}$. We show that there exists a $k \in \N$, $I \in \{0,\ldots,\hd-1\}^N$, $N \in \{0,1,2\}$, such that
  \begin{equation*}
    \langle (e_I \shuffle e_0^{\otimes k}) \otimes e_0, \hbbX_T \rangle \neq \langle (e_I \shuffle e_0^{\otimes k}) \otimes e_0, \hbbY_T \rangle.
  \end{equation*}
  We proceed with a proof by contradiction. Assume that for all $k \in \N$, $I \in \{0, \ldots, \hd-1\}^N$, $N \in \{0,1,2\}$, we have
  \begin{equation*}
    \langle (e_I \shuffle e_0^{\otimes k}) \otimes e_0, \hbbX_T \rangle = \langle (e_I \shuffle e_0^{\otimes k}) \otimes e_0, \hbbY_T\rangle.
  \end{equation*}
  We first note that
  \begin{equation*}
	\langle e_0^{\otimes k}, \hbbX_t \rangle = \frac{t^k}{k!}.
  \end{equation*}
  Moreover, using the quasi-shuffle property of $\hbbX$, see Definition~\ref{def: quasi shuffle}, we have that
  \begin{equation*}
    \langle e_I \widetilde{\shuffle} e_0^{\otimes k}, \hbbX_t \rangle = \langle e_I \shuffle e_0^{\otimes k}, \hbbX_t \rangle,
  \end{equation*}
  because $\langle \epsilon_{i0}, \hbbX \rangle = 0$ for any $i = 0, \dots, \hd-1$, i.e., the quasi-shuffle product and the shuffle product coincide. Hence, by e.g.~\cite[Proposition~C.5]{Cuchiero2023b},
  \begin{equation*}
    \langle (e_I \widetilde{\shuffle} e_0^{\otimes k}) \otimes e_0, \hbbX_T \rangle 
	= \int_0^T \langle e_I \widetilde{\shuffle} e_0^{\otimes k}, \hbbX_t \rangle \dd t = \int_0^T \langle e_I, \hbbX_t \rangle \langle e_0^{\otimes k}, \hbbX_t \rangle \dd t
	= \int_0^T \langle e_I, \hbbX_t \rangle \frac{t^k}{k!} \dd t.
  \end{equation*}
  Similarly, we have
  \begin{equation*}
    \langle (e_I \shuffle e_0^{\otimes k}) \otimes e_0, \hbbY_T \rangle = \int_0^T \langle e_I, \hbbY_t \rangle \frac{t^k}{k!} \dd t.
  \end{equation*}
  By \cite[Corollary~4.24]{Brezis2011} and because $\langle e_I, \hbbX_0 \rangle = \langle e_I, \hbbY_0 \rangle = 0$, it then follows that
  \begin{equation*}
    \langle e_I, \hbbX_t \rangle = \langle e_I, \hbbY_t \rangle,
  \end{equation*}
  for all $t \in [0,T]$ and all $I \in \{0, \ldots,\hd-1\}^N$, $N \in \{0,1,2\}$. However, this contradicts the assumption that $\hbX, \widehat{\mathbf{Y}}$ are distinct. Thus we can conclude that $\mathcal{A}$ is point-separating.
\end{proof}

\begin{remark}
  Similarly, one can derive the universal approximation property also when considering functionals of stopped rough paths. The space of stopped rough paths is defined by  $\Lambda_T=\bigcup_{t\in[0,T]}\widehat\cC^{p}([0,t];\R^{\hd})$ with the semi-metric $d_{\Lambda_T}(\hbX_{[0,t]},\widehat{\mathbf{Y}}_{[0,s]}):=|t-s|+d_{p\textup{-var};[0,t\vee s]}(\widehat{\bX^t}_{[0,t\vee s]},\widehat{\mathbf{Y}^s}_{[0,t\vee s]})$. Here $\hbX_{[0,t]}$ denotes the restriction of $\hbX$, that is defined on $[0,T]$, to the sub-interval $[0,t]$, $t\le T$, and $\widehat{\bX^t}$ denotes the stopped rough path. More precisely, $\widehat{X^t}$ is the process $(\widehat{X^t})_s=(s,X^t_s,Q(X_s^t))$, $s \in [0,T]$, that is, the process where we stop the path $X$ and the bracket process $Q(X)$, but not the time-extension, and $\widehat{\bX^t}$ denotes the rough path lift of $\widehat{X^t}$, see e.g.~\cite{Kalsi2020,Cuchiero2025,Bayer2025} for similar definitions. This then allows to approximate non-anticipative path-functionals $f\colon \Lambda_T\to \R$ by linear functionals of the signature on subsets of compact sets.
\end{remark}

\begin{remark}\label{rem: path extension} 
  If the bracket is strictly increasing, i.e.,~$t\mapsto [\bX]_t$ is strictly increasing (for at least one component $[\bX]^{ij}$, $i,j\in\{1,\ldots,d\}$), it suffices to extend the path only by the bracket terms, so to consider only $\hX = (X,Q(X))$. In particular, the quasi-shuffle property and Proposition~\ref{prop: linear functional} still hold true. For point separation one may then apply the same arguments using, for some fixed $i,j\in\{1,\dots,d\}$, $\langle (e_I \shuffle \epsilon_{ij}^{\otimes k})\otimes \epsilon_{ij},\,\hbbX_T\rangle$, for $\langle \epsilon_{ij}, \hbbX^{\gamma,\pi} \rangle = [\bX]^{ij}$.
    
  And if the bracket is given as a linear functional of the time component and the path components of $X$, i.e., $[\bX]_t^{ij}=\sum_{|I|\le 1}a^{ij}_I\langle e_I,\hbbX_t\rangle$, where $I$ is a multi-index taking values in $\{0,\dots,d\}$, it suffices to extend the path only by time, so to consider only $\hX = (\cdot, X)$.
\end{remark}

\subsection{Discussion on approximation with weakly geometric rough paths}\label{sec: discussion}

When considering weakly geometric rough paths, Theorem~\ref{thm: UAT non-geometric} clearly can be applied because the rough path bracket of a weakly geometric rough path is equal to zero. Extending the path by its rough path brackets is therefore redundant. In this case, it is already known that the universal approximation theorem holds for only time-extended weakly geometric rough paths. That is, continuous functionals of time-extended weakly geometric rough paths can be approximated arbitrarily well on compact sets by linear functionals of its signature.

For completeness, let us recall and prove the universal approximation theorem for time-extended weakly geometric rough paths. (A direct proof can be found in Appendix~\ref{appendix: proof of UAT geometric}.) For this, we consider the subspace of time-extended weakly geometric rough paths, defined by
\begin{equation*}
  \widehat{C}_o^{p\textup{-var}}([0,T];G^2(\R^{1+d})) := \{\hbbX^{o, 2} \in C_o^{p\textup{-var}}([0,T];G^2(\R^{1+d})): \langle e_0, \hbbX^{o, 2}_t \rangle = t, \, t \in [0,T]\}.
\end{equation*}

\begin{theorem}\label{thm: UAT geometric}
  Let $p \in (2,3)$. Let $K \subset \widehat{C}_o^{p\textup{-var}}([0,T];G^2(\R^{1+d}))$ be a compact subset, bounded with respect to the $p$-variation norm, and consider a continuous function $f \colon K \to \R$. Then for every $\epsilon > 0$, there exists a linear functional $\ell \in T(\R^{1+d})$ such that
  \begin{equation*}
	  \sup_{\hbbX^{o, 2} \in K} |f(\hbbX^{o, 2}) - \langle \ell, \hbbX^o_{T}\rangle| < \epsilon,
  \end{equation*}
  where $\hbbX^o$ denotes the signature of $\hX := \Pi_{(1)}(\hbbX^{o, 2})$.
\end{theorem}

\begin{proof}
  We first notice that it is equivalent to consider the space of weakly geometric rough paths taken as a subset of the rough path space instead of taken as the space of continuous paths with finite $p$-variation and taking values in $G^2(\R^{1+d})$. The proof then follows by an application of Theorem~\ref{thm: UAT non-geometric}: The rough path bracket of a weakly geometric rough path is equal to zero. Thus, for $L > 0$ large enough, $K_L$ does coincide with $K$. It is therefore actually sufficient to reduce the extended path to $\hX = (\cdot, X)$, i.e., $\hd = d+1$, which then implies the claim.
\end{proof}

\section{The signature using general pathwise stochastic integration}\label{sec: signature via pathwise integration}

In this section, we give an example that fits into the framework developed in the previous section. We introduce a notion of signatures using the path assumption Property $\gamma$-\textup{(RIE)}, which allows to construct pathwise (iterated) integrals as limits of general Riemann sums. It is an extension of Property $\textup{(RIE)}$, which has been established in detail in~\cite{Das2025}. We now give the path properties and the statements required in this paper. For the proofs and an equivalent and more intuitive characterization of the path property, we refer to~\cite{Das2025}.

\subsection{On Property~\texorpdfstring{$\boldsymbol{\gamma}$}{γ}-\textup{(RIE)}}

\begin{gammaRIE}
  Let $X \in C([0,T];\R^d)$, $\gamma \in [0,1]$, $p \in (2,3)$, and $\pi = (\pi^n)_{n\in \N}$, with $\pi^n = \{0 = t^n_0 < t^n_1 < \dots < t^n_{N_n} = T\}$, $n \in \N$, be a sequence of partitions of the interval $[0,T]$ such that $\sup\{|X_{t^n_k,t^n_{k+1}}| \,: k=0, \dots, N_n-1\}$ converges to $0$ as $n \to \infty$. We assume that
  \begin{enumerate}
    \item[(i)] the Riemann sums
      \begin{equation*}
        \int_0^t X_r \otimes \d^{\gamma, \pi^n} X_r := \sum_{k=0}^{N_n-1} (X_{t^n_k} + \gamma X_{t^n_k,t^n_{k+1}}) \otimes X_{t^n_k \wedge t, t^n_{k+1} \wedge t}, \quad t \in [0,T],
     \end{equation*}
     converge uniformly as $n \to \infty$ to a limit, which we denote by $\int_0^t X_r \otimes \d^{\gamma,\pi} X_r$,
     \item[(ii)] there exists a control function $c$ such that
     \begin{equation}\label{eq: bounded p/2 variation assumption}
       \sup_{(s,t) \in \Delta_T} \frac{|X_{s,t}|^p}{c(s,t)} + \sup_{n \in \N} \sup_{0 \le k < \ell \le N_n} \frac{| (\int_0^{\cdot} X_r \otimes \d^{\gamma,\pi^n} X_r)_{t^n_k,t^n_\ell} - X_{t^n_k} \otimes X_{t^n_k,t^n_\ell}|^{\p}}{c(t^n_k,t^n_\ell)} \lesssim 1.
    \end{equation}
  \end{enumerate}
\end{gammaRIE}

We say that a path $X \in C([0,T];\R^d)$ satisfies \emph{Property~$\gamma$-\textup{(RIE)}} relative to $\gamma$, $p$ and $\pi$ if $\gamma$, $p$, $\pi$ and $X$ together satisfy Property $\gamma$-\textup{(RIE)}.

\begin{proposition}[Proposition~2.9 in \cite{Das2025}]\label{prop: rough path lift under gamma RIE}
  Suppose that $X \in C([0,T];\R^d)$ satisfies Property $\gamma$-\textup{(RIE)} relative to some $\gamma \in [0,1]$, $p \in (2,3)$ and a sequence of partitions $\pi = (\pi^n)_{n \in \N}$. Then, $X$ canonically extends to a continuous rough path $\bX^{\gamma,\pi} := (X, \X^{\gamma,\pi,(2)}) \in \cC^p([0,T];\R^d)$, where
  \begin{equation}\label{eq: rough path lift under gamma RIE}
    \X^{\gamma,\pi,(2)}_{s,t} := \int_0^t X_r \otimes \d^{\gamma,\pi} X_r - \int_0^s X_r \otimes \d^{\gamma,\pi} X_r - X_s \otimes X_{s,t}, \qquad (s,t) \in \Delta_T.
  \end{equation}
\end{proposition}

We note that $\bX^{0,\pi}$ corresponds to the It{\^o}-rough path lift, $\bX^{\frac{1}{2},\pi}$ corresponds to the Stra\-to\-no\-vich\--rough path lift, and $\bX^{1,\pi}$ corresponds to the backward It{\^o} rough path lift of a stochastic process, since the ``iterated integral'' $\X^{0,\pi,(2)}$, $\X^{\frac{1}{2},\pi,(2)}$, and $\X^{1,\pi,(2)}$ is given as a limit of left-point, mid-point, and right-point Riemann sums, analogously to the stochastic It{\^o}, Stratonovich, and backward It{\^o} integral, respectively, see \cite[Chapter~5.4]{Friz2020} on backward rough integration for Brownian motion, and Section~\ref{sec: application to continuous semimartingales}. 

When assuming Property $\gamma$-\textup{(RIE)} for a path $X$, we will always work with the rough path $\bX^{\gamma,\pi} = (X,\X^{\gamma,\pi,(2)})$ defined via \eqref{eq: rough path lift under gamma RIE}.

\begin{lemma}[Lemma~2.10 in \cite{Das2025}]\label{lemma: quadratic variation under gamma RIE}
  Suppose that $X \in C([0,T];\R^d)$ satisfies Property $\gamma$-\textup{(RIE)} relative to some $\gamma \in [0,1]$, $p \in (2,3)$ and a sequence of partitions $\pi^n = \{0 = t^n_0 < \dots < t^n_{N_n} = T\}$, $n \in \N$. Let $1 \leq i, j \leq d$, and define for $\gamma = \frac{1}{2}$, $[X^i,X^j]^{\gamma,\pi} := 0$, and for $\gamma \neq \frac{1}{2}$,
  \begin{equation*}
    [X^i,X^j]^{\gamma,\pi}_t := X^i_t X^j_t - X^i_0 X^j_0 - \int_0^t X^i_r \dd^{\gamma,\pi} X^j_r - \int_0^t X^j_r \dd^{\gamma,\pi} X^i_r, \qquad t \in [0,T].
  \end{equation*}
  Then, $[X^i,X^j]^{\gamma,\pi}$ is a continuous function and
  \begin{equation*}
    [X^i,X^j]^{\gamma,\pi}_t = \lim_{n \to \infty} [X^i,X^j]^{\gamma,\pi^n}_t := \lim_{n \to \infty} (1-2\gamma) \sum_{k=0}^{N_n-1} X^i_{t^n_k \wedge t, t^n_{k+1} \wedge t} X^j_{t^n_k \wedge t, t^n_{k+1} \wedge t},\quad t \in [0,T].
  \end{equation*}
  The sequence $([X^i,X^j]^{\gamma,\pi^n})_{n \in \N}$ has uniformly bounded $1$-variation and, thus, $[X^i,X^j]^{\gamma,\pi}$ has finite $1$-variation. We write $[X]^{\gamma,\pi} = [X,X]^{\gamma,\pi} = ([X^i,X^j]^{\gamma,\pi})_{1 \leq i,j \leq d}$, and, analogously, $[X]^{\gamma,\pi^n}$, $n \in \N$.
\end{lemma}

By a slight extension (to allow non-nested partitions) of \cite[Proposition~2.18]{Allan2023b}, the rough path bracket $[\bX^{\gamma,\pi}]$ coincides with $(1-2\gamma)[X]$, where $[X]$ denotes the quadratic variation of $X$ along $\pi$ in the sense of F{\"o}llmer~\cite{Follmer1981}, equal to $[X]^{\gamma,\pi}$.

\medskip

We will actually continue working under Property $\gamma$-\textup{(RIE)}, as it is more general, but we briefly want to point out the theoretical relation to Property \textup{(RIE)}, which has been introduced in~\cite{Perkowski2016} and~\cite{Allan2023a}.

\begin{RIE}
  Let $X \in C([0,T];\R^d)$, $p \in (2,3)$, and $\pi = (\pi^n)_{n\in \N}$, with $\pi^n = \{0 = t^n_0 < t^n_1 < \dots < t^n_{N_n} = T\}$, $n \in \N$, be a sequence of partitions of the interval $[0,T]$ such that $\sup\{|X_{t^n_k,t^n_{k+1}}| \,: k=0, \dots, N_n-1\}$ converges to $0$ as $n \to \infty$. We assume that
  \begin{enumerate}
    \item[(i)] the left-point Riemann sums
    \begin{equation*}
      \int_0^t X_r \otimes \d^{\pi^n} X_r := \sum_{k=0}^{N_n-1} X_{t^n_k} \otimes X_{t^n_k \wedge t, t^n_{k+1} \wedge t}, \quad t \in [0,T],
    \end{equation*}
    converge uniformly as $n \to \infty$ to a limit, which we denote by $\int_0^t X_r \otimes \d^\pi X_r$,
    \item[(ii)] there exists a control function $c$ such that
    \begin{equation*}
      \sup_{(s,t) \in \Delta_T} \frac{|X_{s,t}|^p}{c(s,t)} + \sup_{n \in \N} \sup_{0 \le k < \ell \le N_n} \frac{|(\int_0^{\cdot} X_r \otimes \d^{\pi^n} X_r)_{t^n_k,t^n_\ell} - X_{t^n_k} \otimes X_{t^n_k,t^n_\ell}|^{\p}}{c(t^n_k,t^n_\ell)} \lesssim 1.
    \end{equation*}
  \end{enumerate}
\end{RIE}

We say that a path $X \in C([0,T];\R^d)$ satisfies \emph{Property \textup{(RIE)}} relative to $p$ and $\pi$ if $p$, $\pi$ and $X$ together satisfy Property \textup{(RIE)}.

\begin{lemma}[Lemma~2.16 in \cite{Das2025}]\label{lemma: RIE eq to gamma RIE}
  Let $X \in C([0,T];\R^d)$, $\gamma \in [0,1]$, $p \in (2,3)$ and $\pi = (\pi^n)_{n\in \N}$ be a sequence of partitions.
  \begin{enumerate}
    \item[(i)] Suppose $\gamma \neq \frac{1}{2}$. Then, $X$ satisfies Property \textup{(RIE)} (i) relative to $\pi$ if and only if $X$ satisfies Property $\gamma$-\textup{(RIE)} (i) relative to $\gamma$ and $\pi$, and $X$ satisfies Property \textup{(RIE)} (ii) relative to $p$ and $\pi$ if and only if $X$ satisfies Property $\gamma$-\textup{(RIE)} (ii) relative to $\gamma$, $p$ and $\pi$.
    \item[(ii)] Suppose $\gamma = \frac{1}{2}$. If $X$ satisfies Property \textup{(RIE)} (i) relative to $\pi$, then $X$ satisfies Property $\gamma$-{\textup{(RIE)}} (i) relative to $\gamma$ and $\pi$, and if $X$ satisfies Property \textup{(RIE)} (ii) relative to $p$ and $\pi$, then $X$ satisfies Property $\gamma$-\textup{(RIE)} (ii) relative to $\gamma$, $p$ and $\pi$.
  \end{enumerate}
\end{lemma}

Analogously to Property~\textup{(RIE)}, see~\cite[Proposition~2.12]{Allan2023c}, Property $\gamma$-\textup{(RIE)} is stable under perturbation by a path of finite $q$-variation for $q \in (1,2)$, which then falls into the regime of Young integration. The proof of the following lemma can be found in Appendix~\ref{appendix: proof of lemma}. 

\begin{lemma}\label{lemma: stability of gamma RIE}
  Suppose that $X \in C([0,T];\R^d)$ satisfies Property $\gamma$-\textup{(RIE)} relative to some $\gamma \in [0,1]$, $p \in (2,3)$ and a sequence of partitions $\pi = (\pi^n)_{n \in \N}$. Let $\phi \in C^{q\textup{-var}}([0,T];\R^d)$ for some $q \in [1,2)$ such that $1/p + 1/q > 1$ and $\sup\{|\phi_{t^n_k,t^n_{k+1}}| : k = 0, \dots, N_n - 1\}$ converges to $0$ as $n \to \infty$. Then the path $\hX = X + \phi$ satisfies Property $\gamma$-\textup{(RIE)} relative to $\gamma$, $p$ and $\pi$.
\end{lemma}

\subsection{The \texorpdfstring{$\boldsymbol{\gamma}$}{γ}-signature}

We now show that the canonical rough path under Property $\gamma$-\textup{(RIE)} can be corrected to a weakly geometric rough path by adding the pathwise quadratic variation term, which seems natural when comparing stochastic It{\^o} and Stratonovich integration.

\begin{lemma}\label{lemma: weakly geometric rough path under gamma RIE}
  Suppose that $X \in C([0,T];\R^d)$ satisfies Property $\gamma$-\textup{(RIE)} relative to some $\gamma \in [0,1]$, $p \in (2,3)$ and a sequence of partitions $\pi = (\pi^n)_{n \in \N}$. Let $(X,\X^{\gamma,\pi,o,(2)}) \in \cC^p([0,T];\R^d)$ be a continuous rough path, with $\X^{\gamma,\pi,o,(2)} \colon \Delta_T \to \R^d$ given by
  \begin{equation*}
	\X^{\gamma,\pi,o,(2)}_{s,t} := \X^{\gamma,\pi,(2)}_{s,t} + \frac{1}{2} [X]^{\gamma,\pi}_{s,t}, \qquad (s,t) \in \Delta_T,
  \end{equation*}
  where $\X^{\gamma,\pi,(2)}$ is the canonical rough path lift defined in Proposition~\ref{prop: rough path lift under gamma RIE} and $[X]^{\gamma,\pi}$ is defined in Lemma~\ref{lemma: quadratic variation under gamma RIE}. Then, $\X^{\gamma,\pi,o, 2} \colon [0,T] \to G^2(\R^d)$ is a weakly geometric rough path, where we define
  \begin{equation*}
	\X^{\gamma,\pi,o, 2}_t := (1, X_{0,t}, \X^{\gamma,\pi,o,(2)}_{0,t}), \qquad t \in [0,T].
  \end{equation*}
\end{lemma}

\begin{proof}
  Since $\X^{\gamma,\pi,(2)}$ has finite $\p$-variation and $[X^i,X^j]^{\gamma,\pi}$ has finite $1$-variation, $\X^{\gamma,\pi,o,(2)}$ has finite $\p$-variation, see Proposition~\ref{prop: rough path lift under gamma RIE} and Lemma~\ref{lemma: quadratic variation under gamma RIE}, and particularly, $\|\mathbf{1};\X^{\gamma,\pi,o, 2}\|_p < \infty$.
	
  We show that $\mathbb{S}(\X^{\gamma,\pi,o,(2)}_{0,t}) = \frac{1}{2} X_{0,t} \otimes X_{0,t}$, for any $t \in [0,T]$, where $\mathbb{S}$ denotes the symmetric part of the matrix. Then applying Lemma~\ref{lemma: weakly geometric rough path}, the proof is complete.
	
  By definition, it holds that, for any $1 \leq i, j \leq d$ and any $t \in [0,T]$,
  \begin{align*}
	&(\X^{\gamma,\pi,o,(2)}_{0,t})^{ij} + (\X^{\gamma,\pi,o,(2)}_{0,t})^{ji} \\
	&\quad = \int_0^t X_r^i \dd^{\gamma,\pi} X_r^j - X_0^i X_{0,t}^j + \frac{1}{2} [X^i,X^j]^{\gamma,\pi}_t + \int_0^t X_r^j \dd^{\gamma,\pi} X_r^i- X_0^j X_{0,t}^i + \frac{1}{2} [X^j,X^i]^{\gamma,\pi}_t \\
	&\quad = \lim_{n \to \infty} \sum_{k=0}^{N_n-1} (X^i_{t^n_k} + \gamma X^i_{t^n_k, t^n_{k+1}}) X^j_{t^n_k \wedge t, t^n_{k+1} \wedge t} + (\frac{1}{2} - \gamma) X^i_{t^n_k, t^n_{k+1}} X^j_{t^n_k \wedge t, t^n_{k+1} \wedge t} - X_0^i X_{0,t}^j \\
	&\qquad + \lim_{n \to \infty} \sum_{k=0}^{N_n-1} (X^j_{t^n_k} + \gamma X^j_{t^n_k, t^n_{k+1}}) X^i_{t^n_k \wedge t, t^n_{k+1} \wedge t} + (\frac{1}{2} - \gamma) X^j_{t^n_k, t^n_{k+1}} X^i_{t^n_k \wedge t, t^n_{k+1} \wedge t} - X_0^j X_{0,t}^i \\
	&\quad = \lim_{n \to \infty} \sum_{k=0}^{N_n-1} \frac{1}{2} (X^i_{t^n_k} + X^i_{t^n_{k+1}}) X^j_{t^n_k,t^n_{k+1}} - X_0^i X_{0,t}^j \\
	&\qquad + \lim_{n \to \infty} \sum_{k=0}^{N_n-1} \frac{1}{2} (X^j_{t^n_k} + X^j_{t^n_{k+1}}) X^i_{t^n_k,t^n_{k+1}} - X_0^j X_{0,t}^i \\
	&\quad = X_t^i X_t^j - X_0^i X_0^j - X_0^i X_{0,t}^j - X_0^j X_{0,t}^i \\
	&\quad = X_{0,t}^i X_{0,t}^j.
  \end{align*}
\end{proof}

\begin{remark}\label{rem: Stratonovich rough path is weakly geometric}
  Since $[X]^{\frac{1}{2},\pi} = 0$, see Lemma~\ref{lemma: quadratic variation under gamma RIE}, we notice that $\X^{\frac{1}{2},\pi,o,(2)} = \X^{\frac{1}{2},\pi,(2)}$, which implies that $(1,X_{0,\cdot},\X^{\frac{1}{2},\pi,(2)}_{0,\cdot}) \in C^{p\textup{-var}}_o([0,T];G^2(\R^d))$. That is, the Stratonovich-type rough path is indeed a weakly geometric rough path, which is very reasonable.

  Further, for any $\gamma \in [0,1]$, it holds that $\X^{\gamma,\pi,o,(2)} = \X^{\frac{1}{2},\pi,(2)}$.
\end{remark}

\begin{remark}\label{rem: definition of the signature}
  If $X$ satisfies Property $\gamma$-\textup{(RIE)}, one can consider the signature $\X^{\gamma,\pi,o}$ of $X$ defined in Definition~\ref{def: Lyons' lift of weakly geometric rough path} as the unique path extension of $\X^{\gamma,\pi,o, 2}$. This then coincides with $\X^{\frac{1}{2},\pi,o}$, the unique path extension of $(1,X_{0,\cdot},\X^{\frac{1}{2},\pi,(2)})$.
\end{remark}

Now, we define the signature as the collection of all iterated integrals over a fixed interval associated to a sufficiently regular path. Here, we utilize Property $\gamma$-\textup{(RIE)} and the corresponding iterated integral, which allows for a unifying framework for It{\^o}-type and Stratonovich-type signatures.

To that end, we assume that the path $X$ satisfies Property $\gamma$-\textup{(RIE)} relative to $\gamma$, $p$ and $\pi$, and define the signature $\X^{\gamma,\pi}$ of $X$ using Lyons' extension theorem on $T((\R^d))$, see e.g.~\cite[Theorem~3.7]{Lyons2007}. More precisely, we note that $(1, X_{0,\cdot}, \X^{\gamma,\pi,(2)}) \colon [0,T] \to T^2(\R^d)$ is a multiplicative functional of finite $p$-variation controlled by $c$, i.e., $|X_{s,t}| \lesssim c(s,t)$, $|\X^{\gamma,\pi,(2)}_{s,t}| \lesssim c(s,t)$, $(s,t) \in \Delta_T$, where $\X^{\gamma,\pi,(2)}_t := \int_0^t X_{0,r} \otimes \d^{\gamma,\pi} X_r := \int_0^t X_r \otimes \d^{\gamma,\pi} X_r - X_0 \otimes X_{0,t}$, and $c$ is the control function for which \eqref{eq: bounded p/2 variation assumption} holds, see Proposition~\ref{prop: rough path lift under gamma RIE}. Applying Lyons' extension theorem, for every $N \geq 3$, there exists a unique continuous path $\X^{\gamma,\pi,(N)} \colon [0,T] \to (\R^d)^{\otimes N}$ such that
\begin{equation*}
  (1, X_{0,\cdot}, \X^{\gamma,\pi,(2)}, \dots, \X^{\gamma,\pi,(N)}, \dots) \in T((\R^d))
\end{equation*}
is a multiplicative functional with finite $p$-variation, that is, in particular, $|\X_{s,t}^{\gamma,(N)}|\lesssim c(s,t)^{\frac{N}{p}}$, $(s,t) \in \Delta_T$, for any $N \geq 1$.

Proposition~\ref{prop: Lyons extension coincides with iterated integrals} now states that Lyons' extension $\X^{\gamma,\pi} \colon [0,T] \to T((\R^d))$ coincides with the collection of iterated rough integrals of controlled paths with respect to $\bX^{\gamma,\pi}=(X,\X^{\gamma,\pi,(2)})$, that is, for $N \geq 3$,
\begin{equation*}
  \X_{s,t}^{\gamma,\pi,(N)}=\int_s^t\X_{s,r}^{\gamma,\pi,(N-1)} \otimes \d \bX_r^{\gamma,\pi}, \qquad (s,t) \in \Delta_T,
\end{equation*}
where the rough integral is defined in Lemma~\ref{lemma: integration against controlled path}.

We notice that due to Theorem~\ref{thm: rough int against controlled path under RIE}, the integral exists as a limit of Riemann sums along $\pi$. Moreover, for any multi-index $I = (i_1,\dots,i_N)$ of length $N$, because $\langle e_I,\cdot\rangle \colon (\R^d)^{\otimes N} \to\R$ is continuous, we observe that
\begin{align*}
  &\langle e_I,\X^{\gamma,\pi}_{s,t}\rangle = \langle e_I, \X^{\gamma,\pi,(N)}_{s,t} \rangle \\
  &\quad = \langle e_I, \lim_{n\to\infty}\sum_{k=0} ^{N_n-1}
  \bigl(\X^{\gamma,\pi,(N-1)}_{s,t^n_k} + \gamma (\X^{\gamma,\pi,(N-1)}_{s,t^n_{k+1}}-\X^{\gamma,\pi,(N-1)}_{s,t^n_{k}})\bigr)\otimes X_{t^n_k \vee s,t^n_{k+1} \wedge t}\rangle\\
  &\quad= \lim_{n\to\infty}\sum_{k=0}^{N_n-1}
  \langle e_I,(\X^{\gamma,\pi,(N-1)}_{s,t^n_k} + \gamma (\X^{\gamma,\pi,(N-1)}_{s,t^n_{k+1}}-\X^{\gamma,\pi,(N-1)}_{s,t^n_{k}}))\otimes X_{t^n_k \vee s,t^n_{k+1} \wedge t} \rangle,
\end{align*}
for $(s,t) \in \Delta_T$. Hence,
\begin{align*}
  \langle e_I,\X^{\gamma,\pi}_{s,t}\big\rangle
  &= \lim_{n\to\infty}\sum_{k=0}^{N_n-1}
  \langle e_{I'},\X^{\gamma,\pi,(N-1)}_{s,t^n_k} + \gamma (\X^{\gamma,\pi,(N-1)}_{s,t^n_{k+1}}-\X^{\gamma,\pi,(N-1)}_{s,t^n_{k}})\rangle
  \, X^{i_N}_{t^n_k \vee s,t^n_{k+1} \wedge t},
\end{align*}
which yields that
\begin{equation*}
  \langle e_I,\X^{\gamma,\pi}_{s,t}\big\rangle
  = \int_s^t \langle e_{I'},\X^{\gamma,\pi}_{s,r}\big\rangle\,\dd^{\gamma,\pi}X^{i_N}_r,
\end{equation*}
where the integral on the right-hand side is well-defined as a rough integral with respect to a controlled path and exists as a limit of Riemann sums along $\pi$, see Lemma~\ref{lemma: integration against controlled path}, Remark~\ref{rem: tensored controlled path} and Theorem~\ref{thm: rough int against controlled path under RIE}, because $r\mapsto\langle e_{I^\prime},\X^{\gamma,\pi}_{s,r}\rangle$ is a controlled path w.r.t.~$X$ on $[s,t]$ and so is $X^{i_N}$.

\medskip

This allows us to define the \emph{$\gamma$-signature} as follows.

\begin{definition}\label{def: gamma-signature}
  Suppose that $X \in C([0,T];\R^d)$ satisfies Property $\gamma$-\textup{(RIE)} relative to some $\gamma \in [0,1]$, $p \in (2,3)$ and a sequence of partitions $\pi = (\pi^n)_{n \in \N}$. 
  
  We recursively set
  \begin{align*}
    &\langle e_\emptyset, \X^{\gamma,\pi}_t \rangle := 1, \quad \langle e_I, \X^{\gamma,\pi}_t \rangle := X_{0,t}^{i_1}, \quad I = (i_1),\\
	&\langle e_I, \X^{\gamma,\pi}_t \rangle := \int_0^t X^{i_1}_r \dd^{\gamma,\pi} X^{i_2}_r - X_0^{i_1} X_{0,t}^{i_2} = (\X^{\gamma,\pi,(2)}_{0,t})^{i_1i_2}, \quad I = (i_1,i_2), \\
	&\langle e_I, \X^{\gamma,\pi}_t \rangle :=\int_0^t\langle e_{I^\prime},\X_r^{\gamma,\pi} \rangle \dd^{\gamma,\pi} X_r^{i_{|I|}} \\
    &\qquad \qquad = \lim_{n \to \infty} \sum_{k=0}^{N_n-1} (\langle e_{I'}, \X^{\gamma,\pi}_{t^n_k} \rangle + \gamma (\langle e_{I'}, \X^{\gamma,\pi}_{t^n_{k+1}} \rangle - \langle e_{I'}, \X^{\gamma,\pi}_{t^n_k} \rangle)) X^{i_{|I|}}_{t^n_k \wedge t, t^n_{k+1} \wedge t},
  \end{align*}
  where the integral exists as a rough integral of controlled paths, for $I = (i_1, \ldots, i_{|I|})$, $|I| > 2$, and $t \in [0,T]$. Then $\X^{\gamma,\pi} \colon [0,T] \to T((\R^d))$ is well-defined and is  called the \emph{$\gamma$-signature} of $X$. Its projection $\X^{\gamma,\pi, N}$ on $T^N(\R^d)$ is given by
  \begin{equation*}
	  \X^{\gamma,\pi, N}_t = \Pi_{N}(\X^{\gamma,\pi}_t) = \sum_{|I|\leq N} \langle e_I, \X^{\gamma,\pi}_t \rangle e_I,
  \end{equation*}
  and called \emph{$\gamma$-signature of $X$ truncated at level $N$}, which takes values in $T^N(\R^d)$ for all $t \in [0,T]$. The increments of the $\gamma$-signature $\X^{\gamma,\pi}$ are defined by
  \begin{equation*}
    \X^{\gamma,\pi}_{s,t} := (\X^{\gamma,\pi}_s)^{-1} \otimes \X^{\gamma,\pi}_t, \qquad (s,t) \in \Delta_T.
  \end{equation*}
\end{definition}

\begin{remark}
  Remark~\ref{rem: Stratonovich rough path is weakly geometric} states that
  \begin{equation*}
	\Pi_{2}(\X^{\gamma,\pi,o}_t) = \X^{\gamma,\pi,o, 2}_t = (1, X_{0,t}, \X^{\gamma,\pi,o,(2)}_{0,t}) = (1, X_{0,t}, \X^{\frac{1}{2},\pi,(2)}_{0,t}) = \Pi_{ 2}(\X^{\frac{1}{2},\pi}_t),
  \end{equation*}
  that is, if $X$ satisfies Property $\gamma$-\textup{(RIE)}, the signature of the corresponding weakly geometric rough path truncated at level $2$ and the $1/2$-signature of $X$ truncated at level $2$ coincide. Moreover, since the $\gamma$-signature is defined to be the unique Lyons' extension, the $1/2$-signature $\X^{\frac{1}{2},\pi}$ is a group-like valued path, i.e., $\X^{\frac{1}{2},\pi}_t \in G((\R^d))$, and coincides with the signature of the weakly geometric rough path $\X^{\gamma,\pi,o,2} = \X^{\frac{1}{2},\pi,o,2} = (1,X_{0,\cdot},\X^{\frac{1}{2},\pi})$, see Definition~\ref{def: Lyons' lift of weakly geometric rough path}.
\end{remark}

Now, suppose that $X \in C([0,T];\R^d)$ satisfies Property $\gamma$-\textup{(RIE)} relative to some $\gamma \in [0,1]$, $p \in (2,3)$ and a sequence of partitions $\pi = (\pi^n)_{n \in \N}$. We set
\begin{equation}\label{eq: extended path}
  \hX := (\cdot,X,\QX) \in C([0,T];\R^\hd),
\end{equation}
where $\hd = 1 + d + \frac{d(d+1)}{2}$, and
\begin{equation*}
  \QX := ([X^1,X^1]^{\gamma,\pi}, \dots, [X^1,X^d]^{\gamma,\pi}, [X^2,X^2]^{\gamma,\pi}, \dots, [X^2,X^d]^{\gamma,\pi}, \dots, [X^d,X^d]^{\gamma,\pi}).
\end{equation*}
We first notice that $\hX$ is the path extended by its rough path bracket and time, as introduced in Section~\ref{subsec: signature of non-weakly geometric rough paths}, because $[\bX^{\gamma,\pi}] = [X]^{\gamma,\pi}$. We also recall that $[X]^{\frac{1}{2},\pi} = 0$, see Lemma~\ref{lemma: quadratic variation under gamma RIE}, that is, if $\gamma = \frac{1}{2}$, we may consider $\hX = (\cdot,X)$, and we have that $\hd = 1 + d$.

It follows by applying Lemma~\ref{lemma: quadratic variation under gamma RIE}, and Lemma~\ref{lemma: stability of gamma RIE} to $(\cdot,0,0) + (0,X,0) + (0,0,\QX)$ that $\hX$ satisfies Property $\gamma$-\textup{(RIE)} relative to $\gamma$, $p$ and $\pi$.

As before we use $e_0$ for the time component and $\epsilon_{ij}$ for the component of $\hX$ referring to $[X^i,X^j]^{\gamma,\pi}$, i.e.~$\langle \epsilon_{ij},\hbbX_t^{\gamma,\pi}\rangle:=[X^i,X^j]_t^{\gamma,\pi}$, $i,j=1,\ldots,d$, $i\le j$,  $t\in[0,T]$.

\subsection{A universal approximation theorem with \texorpdfstring{$\boldsymbol{\gamma}$}{γ}-signatures}\label{sec: pathwise universal approximation theorem}

As an example of Theorem~\ref{thm: UAT non-geometric}, the universal approximation property holds true for the $\gamma$-signature of paths satisfying Property $\gamma$-\textup{(RIE)} for $\gamma \in [0,1]$, $\gamma \neq \frac{1}{2}$.

\begin{corollary}\label{cor: pathwise UAT under gamma RIE}
  Let $\gamma \in [0,1]$, $\gamma \neq \frac{1}{2}$, $p \in (2,3)$ and $\pi = (\pi^n)_{n \in \N}$ be a sequence of partitions of the interval $[0,T]$. Let $K \subset \widehat{\cC}^p([0,T];\R^\hd) $ be a compact subset, bounded with respect to the rough path norm, and consider a continuous function $f \colon K \to \R$. Further, for some $L > 0$, let $K_{\gamma,\pi,L} \subset K$ be the subset defined by
  \begin{align*}
    K_{\gamma,\pi,L} := \Biggl\{\hbX^{\gamma,\pi}=(\hX,\hbbX^{\gamma,\pi,(2)}) \in K \, : \, 
    \begin{aligned}
    &X \text{ satisfies Property $\gamma$-\textup{(RIE)} relative to $\gamma$, $p$} \\
    &\text{and } \pi \text{ such that }\|(\hX,\hbbX^{\gamma,\pi,(2)})\|_p + \|[\hX]^{\gamma,\pi}\|_{\p} \leq L
    \end{aligned}
     \Biggr\}.
  \end{align*}
  Then for every $\epsilon > 0$, there exists a linear functional $\ell^\gamma \in T(\R^\hd)$ such that
  \begin{equation*}
	\sup_{\hbX^{\gamma,\pi} \in K_{\gamma,\pi,L}} |f(\hbX^{\gamma,\pi}) - \langle \ell^\gamma, \hbbX^{\gamma,\pi}_T \rangle| < \epsilon,
  \end{equation*}
  where $\hbbX^{\gamma,\pi}$ denotes the $\gamma$-signature of $\hX$.
\end{corollary}

\begin{proof}
  We only need to check that $K_{\gamma,\pi,L} \subset K_L$, where $K_L$ denotes the set defined in Theorem~\ref{thm: UAT non-geometric}. Then the claim follows immediately from Theorem~\ref{thm: UAT non-geometric}. 
    
  Let $\hbX^{\gamma,\pi}=(\hX,\hbbX^{\gamma,\pi,(2)}) \in K_{\gamma,\pi,L}$. Since $X$ satisfies Property $\gamma$-\textup{(RIE)} relative to $\gamma$, $p$ and $\pi$, it holds that $\bX^{\gamma,\pi} = (X,\X^{\gamma,\pi,(2)}) \in \cC^p([0,T];\R^d)$, see Proposition~\ref{prop: rough path lift under gamma RIE}. Further, by \cite[Proposition~2.18]{Allan2023b}, we have that the rough path bracket $[\bX^{\gamma,\pi}]$ coincides with $(1-2\gamma)[X]$, where $[X]$ denotes the quadratic variation of $X$ along $\pi$ in the sense of F\"ollmer, equal to $[X]^{\gamma,\pi}$, which yields that $[\bX^{\gamma,\pi}]=[X]^{\gamma,\pi}$, i.e. $Q(X) = Q^{\gamma,\pi}(X)$. Also, $\hX$ satisfies Property $\gamma$-\textup{(RIE)}, see Lemma~\ref{lemma: stability of gamma RIE}, so that $[\hbX^{\gamma,\pi}] = [\hX]^{\gamma,\pi}$, and thus, $\hbbX^{\gamma,\pi} \in K_L$.
\end{proof}

We immediately obtain the following corollary. Assuming that the path satisfies Property $1/2$-\textup{(RIE)}, by Remark~\ref{rem: Stratonovich rough path is weakly geometric}, the Stratonovich-type lift is weakly geometric, and the statement also applies directly to the time-extended Stratonovich-type signature.

\begin{corollary}
  Let $K \subset \widehat{C}_o^{p\textup{-var}}([0,T];G^2(\R^{1+d}))$ be a compact subset, bounded with respect to the $p$-variation norm, and consider a continuous function $f \colon K \to \R$. Further, define
  \begin{align*}
	K_{\frac{1}{2},\pi} := \Biggl\{\hbbX^2 \in K \, : \, 
    \begin{aligned}
    &X \text{ satisfies Property $1/2$-\textup{(RIE)} relative to $p$ and $\pi$}\\
    &\text{such that }\hbbX^{(1)} = \hX_{0,\cdot} \text{ and } \hbbX^{(2)} = \hbbX^{\frac{1}{2},\pi,(2)} \\
    \end{aligned}
    \Biggr\}.
  \end{align*}
  Then for every $\epsilon > 0$, there exists a linear functional $\ell \in T(\R^{1+d})$ such that
  \begin{equation*}
	\sup_{\hbbX^{\frac{1}{2},\pi, 2} \in K_{\frac{1}{2},\pi}} |f(\hbbX^{\frac{1}{2},\pi, 2}) - \langle \ell, \hbbX^{\frac{1}{2},\pi}_{T}\rangle| < \epsilon.
  \end{equation*}
\end{corollary}

\section{Application to continuous semimartingales}\label{sec: application to continuous semimartingales}

In this section, we apply the deterministic theory developed in Section~\ref{sec: general pathwise signatures} and Section~\ref{sec: pathwise universal approximation theorem} to continuous semimartingales.

In fact, continuous semimartingales fit well into the theory of signatures when adopting the notion of stochastic integration. That is, the signature can be defined as the collection of iterated integrals via stochastic integration. Because it is obeying first order calculus, one usually considers Stratonovich integration, which almost surely coincides with Lyons' lift, thus classically implying a universal approximation theorem for continuous path functionals.

\medskip 

Throughout, let $X=(X_t)_{t\in [0,T]}$ be a $d$-dimensional continuous semimartingale, defined on a probability space $(\Omega, \mathcal{F}, \P)$ with a filtration $(\mathcal{F}_t)_{t \in [0,T]}$ satisfying the usual conditions, i.e., completeness and right-continuity.

\begin{definition}
  Let $X$ be a $d$-dimensional continuous semimartingale. Its \emph{Stratonovich signature} is the stochastic process $\X^{\circ} = (\X^{\circ}_t)_{t \in [0,T]}$ with values in $T_1((\R^d))$, whose components are recursively defined by
  \begin{equation*}
	\langle e_\emptyset, \X^{\circ}_t \rangle := 1, \qquad \langle e_I, \X^{\circ}_t \rangle := \int_0^t \langle e_{I'}, \X^{\circ}_r \rangle \circ \d X_r^{i_{|I|}},
  \end{equation*}
  for each $I = (i_1, \dots, i_{|I|})$ and $t \in [0,T]$, where $\circ$ denotes the Stratonovich integral. Its projection $\X^{\circ,\leq N}$ on $T^N(\R^d)$ is given by
  \begin{equation*}
	\X^{\circ, N}_t = \Pi_{N}(\X^{\circ}_t) = \sum_{|I|\leq N} \langle e_I, \X^{\circ}_t \rangle e_I,
  \end{equation*}
  and called \emph{Stratonovich signature of $X$ truncated at level $N$}, which takes values in $G^N(\R^d)$ for all $t \in [0,T]$. The increments of the Stratonovich signature $\X^{\circ}$ are defined by
  \begin{equation*}
	\X^{\circ}_{s,t} := (\X^{\circ}_s)^{-1} \otimes \X^{\circ}_t, \qquad (s,t) \in \Delta_T.
  \end{equation*}
\end{definition}

It turns out that, if the semimartingale $X$ satisfies Property $\gamma$-\textup{(RIE)} relative to $\gamma \in [0,1]$, $p \in (2,3)$ and a suitable sequence of partitions, we obtain a canonical signature which corresponds $\P$-almost surely with the signature defined via Lyons' lift and the Stratonovich signature.

\begin{lemma}
  Let $\gamma \in [0,1]$, let $p \in (2,3)$ and let $\pi^n = \{\tau^n_k\}$, $n \in \N$, be a sequence of adapted partitions (so that each $\tau^n_k$ is a stopping time), such that for almost every $\omega \in \Omega$, $(\pi^n(\omega))_{n \in \N}$ is a sequence of (finite) partitions of $[0,T]$ with vanishing mesh size. 
    
  Let $X$ be a continuous $d$-dimensional semimartingale, and suppose that for almost every $\omega \in \Omega$, $\sup\{|X_{\tau^n_k(\omega),\tau^n_{k+1}(\omega)}(\omega)| \, : k=0, \dots, N_n-1\}$ converges to $0$ as $n \to \infty$, and that the sample path $X(\omega)$ satisfies Property $\gamma$-\textup{(RIE)} relative to $\gamma$, $p$ and $(\pi^n(\omega))_{n \in \N}$.
  \begin{enumerate}
    \item[(i)] The random weakly geometric rough path pathwise defined via Proposition~\ref{prop: rough path lift under gamma RIE} for $\gamma = \frac{1}{2}$ and the random weakly geometric rough path pathwise defined via Lemma~\ref{lemma: weakly geometric rough path under gamma RIE} for $\gamma \in [0,1]$ coincide $\P$-almost surely.
	\item[(ii)] The random weakly geometric rough path pathwise defined via Lemma~\ref{lemma: weakly geometric rough path under gamma RIE} and the Stratonovich signature of $X$ truncated at level $2$ coincide $\P$-almost surely.
	\item[(iii)] The random signature $\X^{o}$ pathwise defined via Definition~\ref{def: Lyons' lift of weakly geometric rough path}, or more precisely, Remark~\ref{rem: definition of the signature}, the random signature $\X^{\frac{1}{2},\pi}$ pathwise defined via Definition~\ref{def: gamma-signature} and the Stratonovich signature $\X^{\circ}$ of $X$ coincide $\P$-almost surely.
  \end{enumerate}
\end{lemma}

\begin{proof}
  \emph{(i):} By Lemma~\ref{lemma: RIE eq to gamma RIE}, we know that if a path satisfies Property $\gamma$-\textup{(RIE)} relative to some $\gamma \in [0,1]$, then it particularly satisfies Property $1/2$-\textup{(RIE)}. Then the claim holds true because of Lemma~\ref{lemma: weakly geometric rough path} and $\X^{\frac{1}{2},\pi,(2)}_{0,t} = \X^{\frac{1}{2},\pi,o,(2)}_{0,t}$, $t \in [0,T]$.
	
  \emph{(ii):} By construction, the pathwise rough integral $\int_0^t X_r(\omega) \otimes \d^{\gamma,\pi} X_r(\omega)$ constructed via Property $\gamma$-\textup{(RIE)} is given by the limit as $n \to \infty$ of Riemann sums:
  \begin{equation*}
    \sum_{k=0}^{N_n-1} (X_{\tau^n_k(\omega)}(\omega) + \gamma X_{\tau^n_k(\omega), \tau^n_{k+1}(\omega)}) \otimes X_{\tau^n_k(\omega) \wedge t, \tau^n_{k+1}(\omega) \wedge t} (\omega).
  \end{equation*}
  Suppose that $\gamma = \frac{1}{2}$. Then it is known that these Riemann sums converge uniformly in probability to the Stratonovich integral $\int_0^t X_r \otimes \circ \d X_r$, see e.g.~\cite[Chapter~II, Theorem~21, Theorem~22]{Protter2005}. And the result follows from the (almost sure) uniqueness of limits; see also part~(i) of~\cite[Lemma~4.3]{Das2025}.
	
  Suppose that $\gamma \neq \frac{1}{2}$. Then adding $[X(\omega)]^{\gamma,\pi^n}_{0,t}$,
  \begin{align*}
    &\sum_{k=0}^{N_n-1} (X_{\tau^n_k(\omega)}(\omega) + \gamma X_{\tau^n_k(\omega), \tau^n_{k+1}(\omega)}) \otimes X_{\tau^n_k(\omega) \wedge t, \tau^n_{k+1}(\omega) \wedge t} (\omega) \\
	&\qquad + \frac{1}{2}(1-2\gamma) X_{\tau^n_k(\omega) \wedge t, \tau^n_{k+1}(\omega) \wedge t} (\omega) \otimes X_{\tau^n_k(\omega) \wedge t, \tau^n_{k+1}(\omega) \wedge t} (\omega) \\
	&\quad = \sum_{k=0}^{N_n-1} (X_{\tau^n_k(\omega)}(\omega) + \frac{1}{2} X_{\tau^n_k(\omega) \wedge t, \tau^n_{k+1}(\omega) \wedge t}) \otimes X_{\tau^n_k(\omega) \wedge t, \tau^n_{k+1}(\omega) \wedge t} (\omega) \\
	&\qquad + \gamma (X_{\tau^n_k(\omega), \tau^n_{k+1}(\omega)}) - X_{\tau^n_k(\omega) \wedge t, \tau^n_{k+1}(\omega) \wedge t} (\omega)) \otimes X_{\tau^n_k(\omega) \wedge t, \tau^n_{k+1}(\omega) \wedge t} (\omega),
  \end{align*}
  which again converges uniformly in probability to the Stratonovich integral $\int_0^t X_r \otimes \circ \dd X_r$.
	
  \emph{(iii):} By (i), we have that $\X_{0,t}^{\frac{1}{2},\pi,(2)}=\X_{0,t}^{\frac{1}{2},\pi,o,(2)}$, $t\in [0,T]$. And since Lyons' lift is unique, the random signatures pathwise defined via Definition~\ref{def: Lyons' lift of weakly geometric rough path} (Lyons' lift of the weakly geometric rough path) and via Definition~\ref{def: gamma-signature} ($1/2$-signature) coincide $\P$-almost surely, noting that the $1/2$-signature coincides with the Lyons’ extension of $(1,\hX,\hbbX^{\frac{1}{2},\pi,(2)})$, see Definition~\ref{def: gamma-signature} and Appendix~\ref{sec: finite p-variation}.

  By (ii), the random weakly geometric rough path and the Stratonovich signature of $X$ truncated at level $2$ coincide $\P$-almost surely, and take values in $G^2(\R^d)$. Since Lyons' lift is unique, see \cite[Theorem~9.5]{Friz2010}, and the Stratonovich signature of $X$ truncated at any level $N \geq 3$ takes values in $G^N(\R^d)$, and so does the random signature truncated at level $N$ pathwise defined via Lyons' lift of the weakly geometric rough path, the proof is complete.
\end{proof}

\begin{corollary}\label{cor: UAT Stratonovich}
  Let $X$ be a $d$-dimensional continuous semimartingale, $\hX := (\cdot, X)$, and let $\mathcal{S}^{(2)} := \{\hbbX^{\circ, 2} (\omega): \omega \in \Omega\}$. Further, let $p \in (2,3)$ and $K \subset \widehat{C}_o^{p\textup{-var}}([0,T];G^2(\R^{1+d}))$ be a compact subset of the subspace of time-extended weakly geometric rough paths, see Theorem~\ref{thm: UAT geometric}, bounded with respect to the $p$-variation norm, and consider a continuous function $f \colon K \to \R$. Then for every $\epsilon > 0$, there exists a linear functional $\ell \in T(\R^{1+d})$ such that for almost every $\omega \in \Omega$,
  \begin{equation*}
    |f(\hbbX^{\circ, 2}(\omega)) - \langle \ell, \hbbX^{\circ}_T(\omega) \rangle | < \epsilon \qquad \text{for all} \quad \hbbX^{\circ, 2}(\omega) \in K \cap \mathcal{S}^{(2)},
  \end{equation*}
  where $\hbbX^{\circ}$ denotes the Stratonovich signature of $\hX$.
\end{corollary}

Analogously to the Stratonovich signature, we now define the It{\^o} signature of a continuous semimartingale via iterated stochastic It{\^o} integration, which may be the preferred choice from a modeling perspective when having, for example, a financial application in mind.

\begin{definition}
  Let $X$ be a $d$-dimensional continuous semimartingale. Its \emph{It{\^o} signature} is the stochastic process $\X = (\X_t)_{t \in [0,T]}$ with values in $T_1((\R^d))$, whose components are recursively defined by
  \begin{equation*}
	\langle e_\emptyset, \X_t \rangle := 1, \qquad \langle e_I, \X^{\infty}_t \rangle := \int_0^t \langle e_{I'}, \X_r \rangle  \dd X_r^{i_{|I|}},
  \end{equation*}
  for each $I = (i_1, \dots, i_{|I|})$ and $t \in [0,T]$, where the integral is given as an It{\^o} integral. Its projection $\X^{ N}$ on $T^N(\R^d)$ is given by
  \begin{equation*}
	\X^{ N}_t = \Pi_{N}(\X_t) = \sum_{|I|\leq N} \langle e_I, \X_t \rangle e_I,
  \end{equation*}
  and called \emph{It{\^o} signature of $X$ truncated at level $N$}. The increments of the signature $\X^{\infty}$ are defined by
  \begin{equation*}
	\X_{s,t} := (\X_s)^{-1} \otimes \X_t, \qquad (s,t) \in \Delta_T.
  \end{equation*}
\end{definition}

Due to Theorem~\ref{thm: UAT non-geometric}, we can now formulate a probabilistic version of the universal approximation theorem, using the notion of the It{\^o} signature.

\begin{corollary}\label{cor: UAT Ito}
  Let $X$ be a $d$-dimensional continuous semimartingale. We define the random variable
  \begin{equation*}
	\hX := (\cdot, X, Q(X)),
  \end{equation*}
  with values in $C([0,T];\R^\hd)$ for $\hd = 1 + d + \frac{d(d+1)}{2}$, where
  \begin{equation*}
	Q(X):= ([X^1,X^1], \ldots, [X^1,X^d], [X^2,X^2], \ldots, [X^2,X^d], \ldots, [X^d,X^d]),
  \end{equation*}
  where $[X] = ([X^i,X^j])_{1\leq i,j \leq d}$ denotes the quadratic (co-)variation of $X$. Let $p \in (2,3)$. For some $L > 0$, let
  \begin{align*}
    \mathcal{S}^{(2)}_L := \left\{\hbX(\omega) = (\hX(\omega),\hbbX^{(2)}(\omega)) \, : \,
    \begin{aligned}
    &\omega \in \Omega,\,(X(\omega), \X^{(2)}(\omega)) \in \cC^p([0,T];\R^d), \\
    &\hX(\omega) = (\cdot,X(\omega),Q(X)(\omega)), \\
    &\|(\hX(\omega),\hbbX^{(2)}(\omega))\|_p + \|[\hX](\omega)\|_{\p} \leq L
   \end{aligned}
    \right\},
  \end{align*}
  where $\X^{(2)}_{s,t} := \int_s^t (X_r - X_s) \otimes \d X_r$, $(s,t) \in \Delta_T$, similarly we define $\hbbX^{(2)}$. Further, let $K \subset \widehat{\cC}^p([0,T];\R^{\hd})$ be a compact subset of the subspace of rough paths extended by time and the bracket terms, bounded with respect to the rough path norm, and consider a continuous function $f \colon K \to \R$. Then for every $\epsilon > 0$, there exists a linear functional $\ell \in T(\R^{\hd})$ such that for almost every $\omega \in \Omega$,
  \begin{equation*}
    |f(\hbX(\omega)) - \langle \ell, \hbbX_T(\omega) \rangle | < \epsilon \qquad \text{for all} \quad \hbX(\omega) \in K \cap \mathcal{S}^{(2)}_L,
  \end{equation*}
  where $\hbbX$ denotes the It\^o signature of $\hX$. 
\end{corollary}

\begin{proof}
  We use that a semimartingale can be lifted to a random rough path via its iterated It{\^o} integrals, see e.g.~\cite[Theorem~6.5]{Friz2018}, and that the corresponding rough path bracket coincides almost surely with the quadratic variation of the semimartingale, see e.g.~\cite[Remark~2.7]{Friz2020}. The claim then immediately follows from the pathwise universal approximation theorem for linear functionals of the signature of general rough paths, which is Theorem~\ref{thm: UAT non-geometric}.
\end{proof}

\begin{remark}\label{rem: Brownian sig}
  For example, in the case where $X=W$ is a (correlated) $d$-dimensional Brownian motion with correlation matrix $\rho$, the quadratic (co-)variation is $[W]^{ij}_t=\rho_{ij} t$, $t\in[0,T]$, $i,j=1,\ldots,d$. That is, the quadratic variation is a linear function in time, and coincides almost surely with the rough path bracket of the random It{\^o} rough path. Consequently, see also Remark~\ref{rem: path extension}, in Brownian settings it is sufficient to consider only the time-extended Brownian motion $\hW=(\cdot,W)$ (i.e., without extending by quadratic variation), and a universal approximation result formulated analogously to Corollary~\ref{cor: UAT Ito} holds true.
\end{remark}

It turns out that, if the semimartingale $X$ satisfies Property $0$-\textup{(RIE)}, which is equivalent to Property $\textup{(RIE)}$, see Lemma~\ref{lemma: RIE eq to gamma RIE}, then the $0$-signature and the It{\^o} signature coincide almost surely. In particular, for almost every $\omega \in \Omega$, the limits of left-point Riemann sums exist $\omega$-wise.

\begin{lemma}\label{lemma: Ito properties}
  Let $p \in (2,3)$ and let $\pi^n = \{\tau^n_k\}$, $n \in \N$, be a sequence of adapted partitions (so that each $\tau^n_k$ is a stopping time), such that for almost every $\omega \in \Omega$, $(\pi^n(\omega))_{n \in \N}$ is a sequence of (finite) partitions of $[0,T]$ with vanishing mesh size. 
    
  Let $X$ be a $d$-dimensional continuous semimartingale, and suppose that for almost every $\omega \in \Omega$, $\sup\{|X_{\tau^n_k(\omega),\tau^n_{k+1}(\omega)}(\omega)| \, : k=0, \dots, N_n-1\}$ converges to $0$ as $n \to \infty$, and that the sample path $X(\omega)$ satisfies Property $0$-\textup{(RIE)} relative to $p$ and $(\pi^n(\omega))_{n \in \N}$.
  \begin{enumerate}
    \item[(i)] The random rough path pathwise defined via Proposition~\ref{prop: rough path lift under gamma RIE} for $\gamma = 0$ and the It{\^o} signature of $X$ truncated at level $2$ coincide $\P$-almost surely.
	\item[(ii)] The random $0$-signature $\X^{0,\pi}$ pathwise defined via Definition~\ref{def: gamma-signature} and the It{\^o} signature $\X$ of $X$ coincide $\P$-almost surely.
  \end{enumerate}
\end{lemma}

\begin{proof}
  \emph{(i):} Since Property $0$-\textup{(RIE)} and Property \textup{(RIE)} are equivalent, see also Lemma~\ref{lemma: RIE eq to gamma RIE}, this is the statement of part~(i) of~\cite[Lemma~3.1]{Allan2023c}.
	
  \emph{(ii):} Since the $0$-signature coincides with then Lyons' lift of the rough path $(1,X_{0,\cdot},\X^{0,\pi,(2)})$, see Definition~\ref{def: gamma-signature} and Appendix~\ref{sec: finite p-variation}, and by (i) the random rough path pathwise defined via Proposition ~\ref{prop: rough path lift under gamma RIE} and the It\^o signature of $X$ truncated at level $2$ coincide $\P$-almost surely, we have by the uniqueness of Lyons' lift, that the signatures coincide $\P$-almost surely.
\end{proof}

Moreover, if the semimartingale $X$ satisfies Property $\gamma$-\textup{(RIE)}, the pathwise quadratic variation and the stochastic quadratic variation coincide almost surely.

\begin{lemma}\label{lemma: stochastic quadratic variation}
  Let $\gamma \in [0,1]$, $p \in (2,3)$ and let $\pi^n = \{\tau^n_k\}$, $n \in \N$, be a sequence of adapted partitions (so that each $\tau^n_k$ is a stopping time), such that for almost every $\omega \in \Omega$, $(\pi^n(\omega))_{n \in \N}$ is a sequence of (finite) partitions of $[0,T]$ with vanishing mesh size.
    
  Let $X$ be a $d$-dimensional continuous semimartingale, and suppose that for almost every $\omega \in \Omega$, $\sup\{|X_{\tau^n_k(\omega),\tau^n_{k+1}(\omega)}(\omega)| \, : k=0, \dots, N_n-1\}$ converges to $0$ as $n \to \infty$, and that the sample path $X(\omega)$ satisfies Property $\gamma$-\textup{(RIE)} relative to $\gamma$, $p$ and $(\pi^n(\omega))_{n \in \N}$. We define the random variable
  \begin{equation*}
	\hX := (\cdot, X, (1-2\gamma)Q(X)),
  \end{equation*}
  with values in $C([0,T];\R^\hd)$ for $\hd = 1 + d + \frac{d(d+1)}{2}$, where
  \begin{equation*}
	Q(X):= ([X^1,X^1], \ldots, [X^1,X^d], [X^2,X^2], \ldots, [X^2,X^d], \ldots, [X^d,X^d]),
  \end{equation*}
  where $[X] = ([X^i,X^j])_{1\leq i,j \leq d}$ denotes the quadratic (co-)variation of $X$. Then $\hX$ and the random variable that is pathwise defined via~\eqref{eq: extended path} coincide $\P$-almost surely.
\end{lemma}

\begin{proof}
  This clearly holds true for $\gamma = \frac{1}{2}$. Therefore suppose that $\gamma \neq \frac{1}{2}$. By definition, the pathwise quadratic variation $[X^i(\omega),X^j(\omega)]^{\gamma,\pi}$ is given by the limit as $n \to \infty$ of:
  \begin{equation*}
	(1-2\gamma)\sum_{k=0}^{N_n-1} X^i_{\tau^n_k(\omega) \wedge t, \tau^n_{k+1}(\omega) \wedge t} (\omega) X^j_{\tau^n_k(\omega) \wedge t, \tau^n_{k+1}(\omega) \wedge t} (\omega).
  \end{equation*}
  We know that these sums converge uniformly (in $t \in [0,T]$) in probability to the quadratic variation $(1-2\gamma) [X^i,X^j]$, see e.g.~\cite[Chapter~II, Theorem~22]{Protter2005}. By taking a subsequence, if necessary, it follows the (almost sure) uniqueness of limits.
\end{proof}

As a consequence of Corollary~\ref{cor: pathwise UAT under gamma RIE} (or Corollary \ref{cor: UAT Ito}), Lemma~\ref{lemma: Ito properties} and Lemma~\ref{lemma: stochastic quadratic variation}, we formulate universality of the It{\^o} signature of a continuous semimartingale, whose sample paths almost surely satisfy Property $0$-\textup{(RIE)} or, equivalently, Property \textup{(RIE)}, that is, so that the It{\^o} signature almost surely exists as collection of iterated integrals, where the integral exists as limit of pathwise left-point Riemann sums. This holds true for various semimartingales relative to suitable sequences of partitions. We refer to~\cite[Section~3]{Allan2023c}.

\begin{corollary}[Universal approximation theorem for the It{\^o} signature]\label{thm: UAT Ito}
  Let $p \in (2,3)$ and let $\pi^n = \{\tau^n_k\}$, $n \in \N$, be a sequence of adapted partitions (so that each $\tau^n_k$ is a stopping time), such that for almost every $\omega \in \Omega$, $(\pi^n(\omega))_{n \in \N}$ is a sequence of (finite) partitions of $[0,T]$ with vanishing mesh size. 
    
  Let $X$ be a $d$-dimensional continuous semimartingale, and suppose that for almost every $\omega \in \Omega$, $\sup\{|X_{\tau^n_k(\omega),\tau^n_{k+1}(\omega)}(\omega)| \, : k=0, \dots, N_n-1\}$ converges to $0$ as $n \to \infty$, and that the sample path $X(\omega)$ satisfies Property $0$-\textup{(RIE)} relative to $p$ and $(\pi^n(\omega))_{n \in \N}$. 
  
  Let $\hX := (\cdot, X, Q(X))$, where $Q(X)$ is defined as in Corollary~\ref{cor: UAT Ito}, and $\mathcal{S}^{(2)} := \{(\hX,\hbbX^{(2)})(\omega): \omega \in \Omega\}$, where $\hbbX^{(2)}_{s,t} := \int_s^t (\hX_r - \hX_s) \otimes \d \hX_r$, $(s,t) \in \Delta_T$. Further, let $K \subset \widehat{\cC}^{p\textup{-var}}([0,T];\R^\hd)$ be a compact subset of the subspace of rough paths extended by time and the bracket terms, bounded with respect to the rough path norm, and consider a continuous function $f \colon K \to \R$. For some $L > 0$, let $K_{0,\pi,L} \subset K$ be the subset defined by
  \begin{align*}
	K_{0,\pi,L} := \Biggl\{(\hX,\hbbX^{0,\pi,(2)})\in K \, : \,
    \begin{aligned}
    &X \text{ satisfies Property $0$-\textup{(RIE)} relative to $p$ and $\pi$} \\
    &\text{such that } \|(\hX,\hbbX^{0,\pi,(2)})\|_p + \|[\hX]^{0,\pi}\|_{\p} \leq L
    \end{aligned}
    \Biggr\}.
  \end{align*}
  Then for every $\epsilon > 0$, there exists a linear functional $\ell \in T(\R^\hd)$ such that for almost every $\omega \in \Omega$,
  \begin{equation*}
    |f((\hX,\hbbX^{(2)})(\omega)) - \langle \ell, \hbbX_T (\omega)\rangle| < \epsilon \qquad \text{for all} \quad (\hX,\hbbX^{(2)})(\omega) \in K_{0,\pi,L} \cap \mathcal{S}^{(2)},
  \end{equation*}
  where $\hbbX$ denotes the It{\^o} signature of $\hX$.
\end{corollary}

\begin{proof}
  We use that for almost every $\omega \in \Omega$, the random $0$-signature of $\hX(\omega)$ and the It{\^o} signature $\hbbX(\omega)$ coincide, see Lemma~\ref{lemma: stochastic quadratic variation} and part~(ii) of Lemma~\ref{lemma: Ito properties}. The claim then immediately follows from the pathwise universal approximation theorem for linear functionals of the $\gamma$-signature, which is Corollary~\ref{cor: pathwise UAT under gamma RIE}.

  Alternatively, one could use Corollary~\ref{cor: UAT Ito} to prove the claim.
\end{proof}

\begin{remark}
  An analogous result also holds true when considering the Stratonovich signature of $X$ instead of the It{\^o} signature of $X$ (also if almost all sample paths only satisfy Property $1/2$-\textup{(RIE)}). This can be shown using the results of the previous sections. This is, however, weaker than the classical universal approximation theorem stated in Corollary~\ref{cor: UAT Stratonovich} since we impose an assumption on the sample paths of the semimartingale to allow for a statement about the It{\^o} signature.
\end{remark} 

\begin{remark}\label{rem: Brownian sig 2}
  Similar to Remark~\ref{rem: Brownian sig}, for example, for a correlated $d$-dimensional Brownian motion $X$, the quadratic (co-)variation is actually a linear function in time. Consequently, it suffices to consider only the time-extended path $(\cdot,X)$, and a universal approximation result formulated analogously to Corollary~\ref{thm: UAT Ito} holds true.
\end{remark}

With Remark~\ref{rem: Brownian sig} (and Remark~\ref{rem: Brownian sig 2}) in mind, when approximating continuous functionals using the Brownian signature, we expect the It{\^o} signature and the Stratonovich signature of the time-extended Brownian motion to perform equally well because both do satisfy the universal approximation property.

\section{Numerical examples}\label{sec: numerical results}

As seen in the previous section, the universal approximation property holds true for (linear functionals of) the Stratonovich signature of the time-extended path, whereas for the It{\^o} signature we need to consider the path extended by time and its quadratic (co)-variation. This naturally leads to the question when it is beneficial in practice to extend the path additionally by its quadratic (co-)variation. Therefore, we conclude the study of It{\^o} signatures with a numerical analysis that illustrates briefly the practical implications of using It{\^o} signatures in the context of mathematical finance.\footnote{The code used to generate the results in this section is available at \url{https://github.com/mihribanceylan/Ito-signatures-Calibration-and-Pricing}.}

We consider calibration to time-series data, payoff approximation, and pricing tasks for options that naturally depend on quadratic variation. In particular, we cover options on realized volatility, as well as covariance and correlation swaps and calls. Throughout these experiments, we find the following: When assuming that the underlying price dynamics are driven by Brownian motion whose quadratic variation is equal to time, here: Heston model, it turns out that---as one may expect---both the It{\^o} and the Stratonovich signature achieve very small approximation errors in this setting, with only minor quantitative differences.

When assuming that the underlying price dynamics are time-changed or when considering payoffs that depend explicitly on realized variance, covariance, or correlation, It{\^o} feature maps perform noticeably better -- due to the quadratic variation extension -- in the sense that we observe substantially lower test errors and more accurate prices.

In each experiment, we compare Stratonovich and It{\^o} feature maps under a common protocol (same truncation level, regularization, and train/test split). For the It{\^o} features, we extend the driving path by time and quadratic variation or, in Brownian settings, only by time as we classically do it for the Stratonovich features. These numerical experiments aim at demonstrating that the It{\^o} signature results in better out-of-sample performance when the quadratic variation actually contains additional information of the path.

A classical problem in mathematical finance, that we do not address here, is the hedging of financial derivatives, for example the mean variance optimal hedging problem. Since the profit of a trading strategy is defined in terms of an It{\^o} integral and the It{\^o} integral can naturally be written as a linear functional of the It{\^o} signature, it might be worth exploring at some point the performance of It{\^o} signatures for the hedging problem, or more precisely, for a linearized version of that problem which has already been studied in \cite{Lyons2020} in the context of signatures.

\subsection{Calibration of signature models}

We first consider the calibration of a signature model to simulated time-series data and take a similar approach as in \cite{Cuchiero2023a}. We fix a time horizon $T>0$, and let $X = (X^1, \dots, X^d)$ be a $d$-dimensional continuous local martingale. We then consider the extended process $\hX$ with values in $\R^{\hd}$, where $\hd$ denotes the dimension of the extended path. We fix a truncation level $N\ge 1$ and denote by $\hbbX^N_t$ the truncated signature of $\hX$ at time $t$ at level $N$. The model we want to calibrate is a signature model, see e.g.~\cite{Cuchiero2023a}, of the form
\begin{equation*}
  S(\ell)_t=S_0+\sum_{0<|I|\le N}\ell_I\langle e_I,\hbbX_t\rangle,\quad \ell_I\in\R,
\end{equation*}
where we set $\ell_\emptyset:=S_0$. We fix a time grid $0=t_0< t_1 < \ldots < t_n = T$, and observe a price path $(S_{t_i})_{i=1}^n$ on this grid, where $n \ge 1$ denotes the number of time steps between $0$ and $T$. The calibration problem consists of finding $\ell^\ast\in \R^{d^\ast}$ such that the loss function
\begin{align*}
  L_{\alpha}(\ell)
  &:=\sum_{i=1}^n\big(S(\ell)_{t_i}-S_{t_i}\big)^2+\alpha\|\ell\|_1\\
  &= \sum_{i=1}^n\bigg(S_0+\sum_{0<|I|\le N}\ell_I\langle e_I,\hbbX_{t_i}\rangle-S_{t_i}\bigg)^2+\alpha\|\ell\|_1,
\end{align*}
is minimized, where $d^\ast:=\frac{\hd^{N+1}-1}{d}$ is the dimension of the signature truncated at level $N$ and $\alpha \|\ell\|_1$ denotes a fixed $L^1$ penalization, so we perform a Lasso regression.

We train on one trajectory on the time interval $[0,T]$ with $T=1$, $n=2000$, and test on $1000$ different realizations on $[0,0.5]$. To assess the accuracy, we report mean squared errors (MSEs) between the predicted and observed trajectories, i.e.,
\begin{equation*}
  \frac{1}{n}\sum_{i=1}^n (S(\ell^\ast)_{t_i}-S_{t_i})^2,
\end{equation*}
on both the training trajectory (in-sample MSE) and the test trajectories (out-of-sample MSE), where the latter is computed as an average of the MSEs on the $1000$ test paths. For the computation of the signature we use available packages, e.g.~the package \verb|iisignature| developed by \cite{Reizenstein2018}, or \verb|esig|\footnote{https://pypi.org/project/esig/}. Throughout, both feature maps use the same penalty $\alpha=10^{-5}$, and the same train and test paths.

\begin{example}[Heston model]\label{ex: heston calibration}
  We generate time series (via an Euler scheme) from the Heston stochastic volatility model under the physical measure $\P$
  \begin{align*}
    \dd S_t&=\mu S_t\dd t +S_t\sqrt{V_t}\dd W_t,\\
    \dd V_t&=\kappa(\theta-V_t)\dd t+\sigma\sqrt{V_t}\dd B_t,
  \end{align*}
  with $[B,W]_t=\rho t$, $\rho\in[-1,1],$ see \cite{Heston1993}. We set the model parameters to be
  \begin{equation*}
    \{S_0,V_0,\mu,\kappa,\theta,\sigma,\rho\}:=\{1,0.08,0.001,0.5,0.15,0.25,-0.5\}.
  \end{equation*}
  We work under an equivalent local martingale measure~$\mathbb Q$, and specify the processes $W^{\mathbb{Q}}$ and $B^{\mathbb{Q}}$ by
  \begin{equation*}
    \dd W^{\mathbb{Q}}_t=\frac{\dd S_t}{S_t\sqrt{V_t}}, \qquad \dd B_t^{\mathbb{Q}}=\frac{\dd V_t}{\sigma\sqrt{V_t}}.
  \end{equation*}
  Then, $W^{\mathbb Q}$ and $B^{\mathbb Q}$ are Brownian motions with respect to $\mathbb Q$ with correlation~$\rho$, and the dynamics of $(S,V)$ can be written as $\dd S_t=S_t\sqrt{V_t}\dd W_t^{\mathbb{Q}}$ and $\dd V_t=\sigma\sqrt{V_t}\dd B^{\mathbb{Q}}_t$. We now aim to approximate these dynamics with signature models with $(W^{\mathbb Q}, B^{\mathbb Q})$ being the underlying process. To this end, we consider two signature models: one based on Stratonovich signatures, the other based on It{\^o} signatures of the time extended path $\hX=(\cdot,W^{\mathbb{Q}},B^{\mathbb{Q}})$. More precisely, we calibrate the following signature models:
  \begin{align*}
    S^{\text{Strat}}(\ell)_t&=S_0+ \sum_{|I|\le N}\ell_I\langle \tilde{e}_I,\hbbX_t^{\text{Strat}}\rangle=S_0+\int_0^t\sum_{|I|\le N}\ell_I\langle e_I,\hbbX^{\text{Strat}}_r\rangle \dd W^{\mathbb Q}_r,\\
    S^{\text{It\^o}}(\ell)_t&=S_0+ \sum_{|I|\le N}\ell_I\langle e_I\otimes e_1,\hbbX_t^{\text{It\^o}}\rangle=S_0+\int_0^t\sum_{|I|\le N}\ell_I\langle e_I,\hbbX^{\text{It\^o}}_r\rangle \dd W^{\mathbb Q}_r,  
  \end{align*}
  where $\tilde{e}_I=e_I\otimes e_1-\frac{1}{2}\rho_{i_{|I|}1}e_{I^\prime}\otimes e_0$, $\rho_{i_{|I|}1}$ is the correlation between $\hX^{i_{|I|}}$ and the Brownian motion $W^{\mathbb Q}$. Here, we choose $N=2$ and $S_0=1$ and obtain the following results:
  \begin{figure}[h]
    \centering
    \begin{subfigure}[h]{0.48\linewidth}
    \centering
    \includegraphics[width=\linewidth]{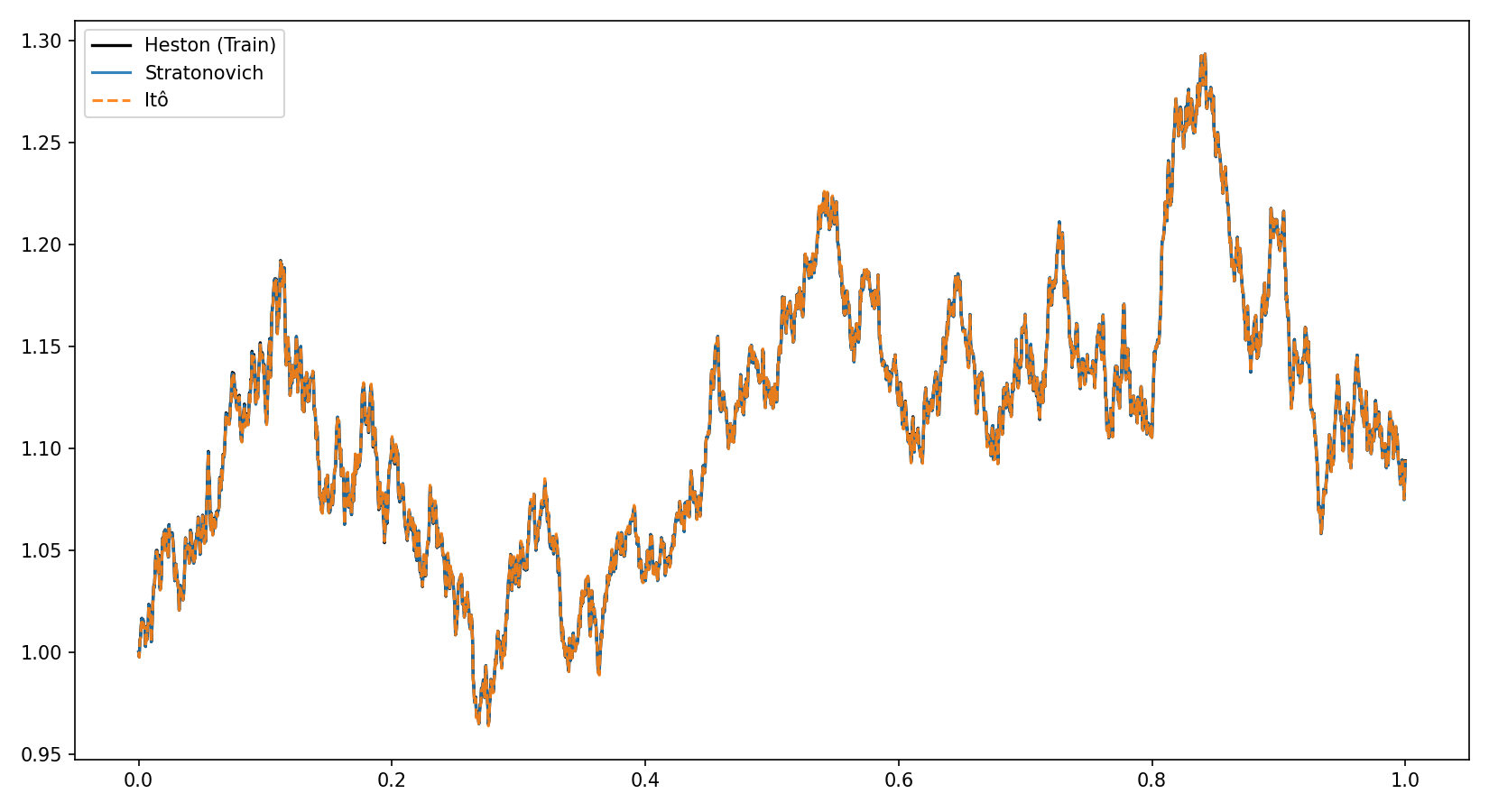}
    \caption{In-sample.}
    \end{subfigure}\hfill
    \begin{subfigure}[h]{0.48\linewidth}
    \centering
    \includegraphics[width=\linewidth]{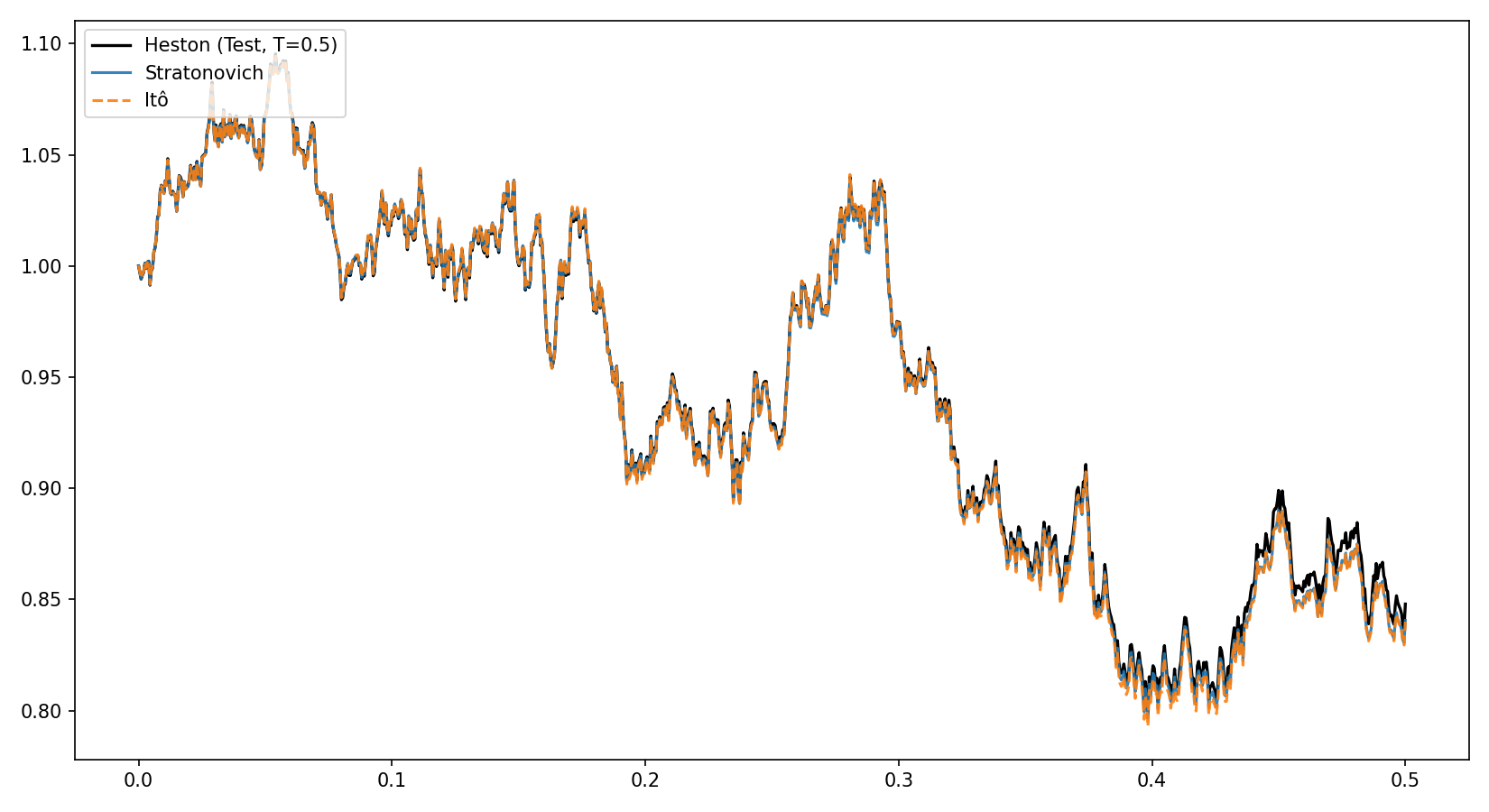}
    \caption{Out-of-sample.}
    \end{subfigure}
    \caption{Regression on the price of the Heston model as defined above.}
    \label{fig: heston_results_calibration}
  \end{figure}

  Figure~\ref{fig: heston_results_calibration} illustrates that the Heston model can be approximated very closely by signature models using the Stratonovich and the It{\^o} signature. We observe an in-sample MSE of order $10^{-8}$ and $10^{-7}$ and out-of-sample MSE of order $10^{-5}$ and $10^{-5}$ for the Stratonovich signature and It{\^o} signature, respectively.

  As stated in Remark~\ref{rem: Brownian sig}, when the model is driven by Brownian motion, that is, when one considers the Brownian signature, it is sufficient to extend the path by time, and both It{\^o} and Stratonovich features achieve comparably small errors (with a slightly better performance for the Stratonovich signature model in our experiments).
\end{example}

\begin{example}[Singular time-changed SDE]
  As a toy example, to mimic volatility dynamics with singular time changes---that appear for instance in hyper-rough Heston-type models with non-absolutely continuous quadratic variation, see e.g.~\cite{Jaber2024}---we consider a one-dimensional SDE driven by a Brownian motion time-changed by a Cantor clock, see~\cite[Example~1 and Section~6]{Aitsahalia2018} and the general framework for SDEs driven by time-changed semimartingales in~\cite{Kobayashi2011}:
  \begin{equation*}
    \dd S_t=\sigma(S_t)\dd W_{C(t)},
  \end{equation*}
  where $C\colon [0,1]\to [0,1]$ is the Cantor function. In our experiment, we choose $\sigma(x)=1+0.3\tanh(x)$. On the time grid, we simulate $C(t)$, generate $W_{C(t)}$ with increments $\Delta W_C\sim \mathcal N(0,\Delta C),$ and use an Euler scheme to simulate $S$. For the Stratonovich signature we take the time-extended path $\widehat X_t=(t,W_{C(t)})\in\mathbb{R}^2$, $t \in [0,1]$; for the It{\^o} signature we additionally include the quadratic variation (QV) and take $\widehat X_t=(t,W_{C(t)},C(t))\in\mathbb{R}^3$, $t \in [0,1]$. We consider the following signature models:
  \begin{align*}
    S^{\text{Strat}}(\ell)_t&=S_0+\sum_{|I|\le N}\ell_I\langle e_I,\hbbX^{\text{Strat}}_t\rangle,\\
    S^{\text{It\^o}}(\ell)_t&=S_0+\sum_{|I|\le N-1}\ell_I\langle e_I\otimes e_{1},\hbbX^{\text{It\^o}}_t\rangle=S_0+\int_0^t\sum_{|I|\le N-1}\ell_I\langle e_I,\hbbX^{\text{It\^o}}_r\rangle \dd W_{C(r)},
  \end{align*}
  where $e_{1}$ corresponds to the component $W_C$ of $\hX$, $\hbbX^{\text{Strat}}$ denotes the Stratonovich signature of $\hX=(\cdot,W_C)$ and $\hbbX^{\text{It\^o}}$ denotes the It{\^o} signature of $\hX=(\cdot,W_C,C)$. We observe that the signature model using the Stratonovich signature cannot be expressed in terms of an It{\^o}-SDE, as was the case in Example~\ref{ex: heston calibration}. This is due to the fact that the Cantor time-change is not absolutely continuous with respect to time. However, using the It{\^o} signature, the corresponding signature model can be written as an It{\^o} SDE driven by the time-changed Brownian motion, which suggests that the It{\^o} signature is the more natural choice when considering such singular time changes.
    
  For our numerical experiment we choose $N=2$ and $S_0=0$ and obtain the following results:
  \begin{figure}[h]
    \centering
    \begin{subfigure}[h]{0.48\linewidth}
    \centering
    \includegraphics[width=\linewidth]{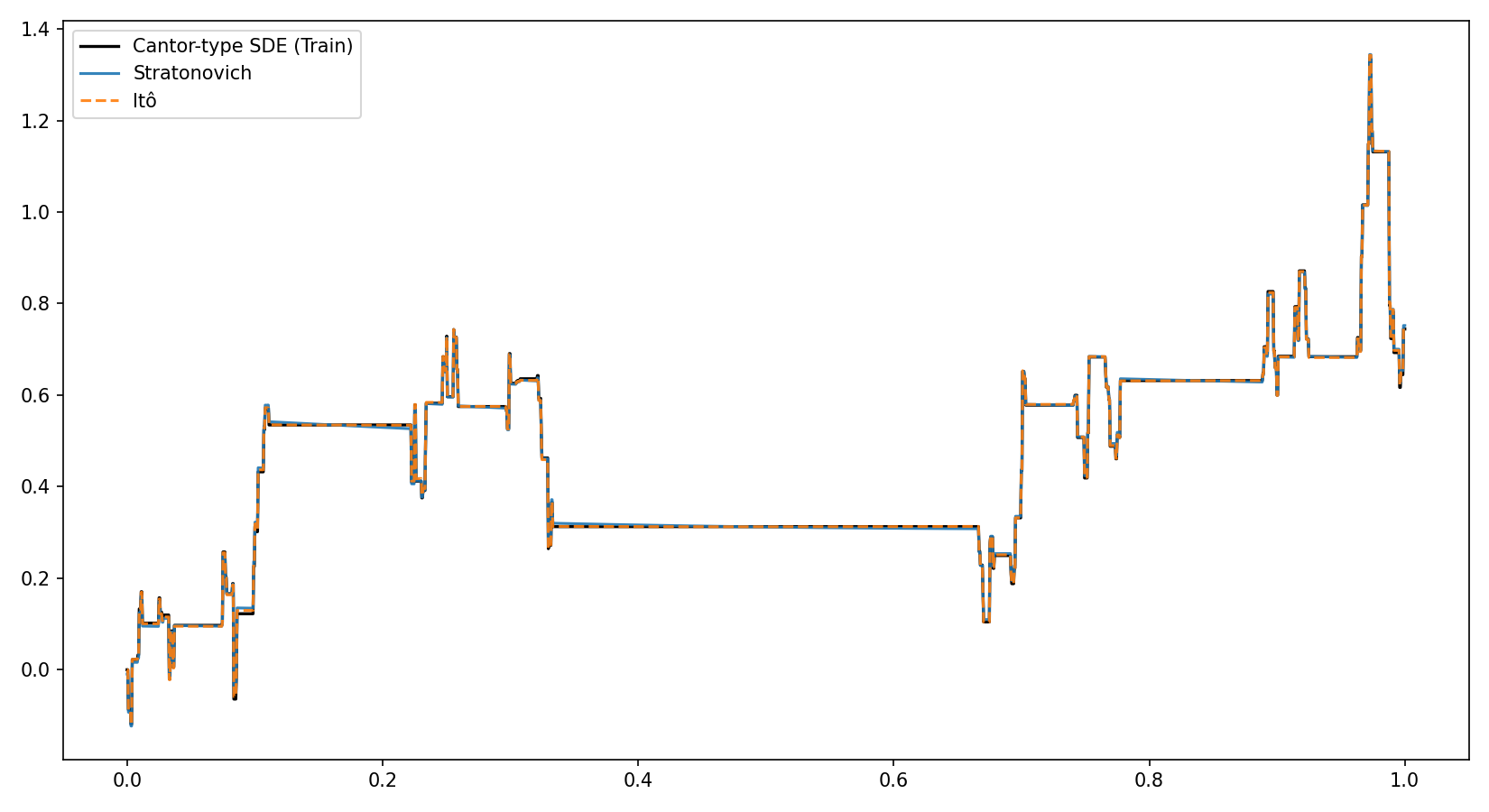}
    \caption{In-sample.}
    \end{subfigure}\hfill
    \begin{subfigure}[h]{0.48\linewidth}
    \centering
    \includegraphics[width=\linewidth]{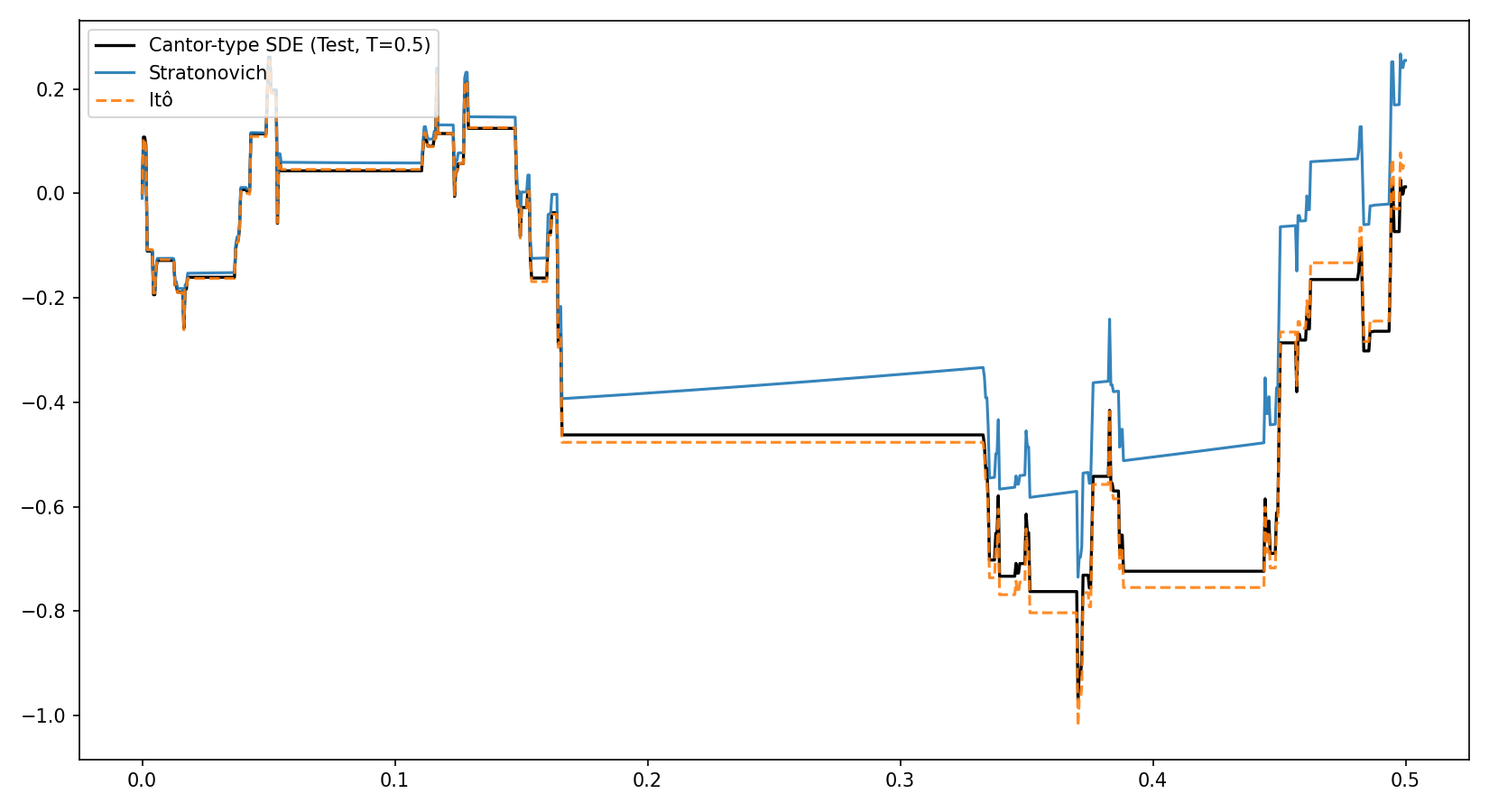}
    \caption{Out-of-sample.}
    \end{subfigure}
    \caption{Regression on the price of a singular time-changed SDE as defined above.}
    \label{fig: Cantor}
  \end{figure}
  
  Figure~\ref{fig: Cantor} illustrates that the signature model using the It{\^o} signature approximates the model more closely than the one using the Stratonovich signature out-of-sample: Both are able to follow the stepped pattern but one can observe that the It{\^o} model results in a path that is more closely aligned with the target path.
     
  The in-sample MSE is of order $10^{-6}$ and $10^{-5}$ and the out-of-sample MSE is of order $10^{-4}$ and $10^{-2}$ when using the It{\^o} signature and the Stratonovich signature, respectively. These findings are consistent with the idea that extending the path by the quadratic variation $[W_C]_t = C(t)$, as we do it for the It{\^o} signature, here actually is advantageous because it contains additional information about the variation of the path that is missing in the Stratonovich signature. This now suggests (and this was the reason we considered this model), and might be worth exploring in future work, that when considering price dynamics that exhibit jumps, signature models based on It{\^o} signatures may also perform better compared to models based on Stratonovich signatures, again because the quadratic variation would contain additional information about the variation of the path (including jumps).
\end{example}

\subsection{Payoff approximation and pricing}

In this subsection we study signature-based approximation and pricing for payoffs that depend explicitly on quadratic (co-)variation, such as options on realized volatility and correlation or covariance swaps and calls.
 
In this setting, since the quadratic variation appears naturally here, we expect the It{\^o} signature to perform particularly well.

This motivated us to compare two set-ups: for single-asset payoffs we consider the Stratonovich signature of the time-extended path $(t,\log S_t^i)$ and the It\^o signature of the path extended by time and quadratic variation $(t,\log S_t^i, [\log S^i]_t)$, $i=1,2$. For correlation and covariance payoffs we analogously consider the Stratonovich signature of $(t,\log S_t^1,\log S_t^2)$ and the It{\^o} signature of the path extended by time and quadratic (co\nobreakdash-)\allowbreak variation $(t,\log S_t^1,\log S_t^2,[\log S^1]_t,[\log S^1,\log S^2]_t,[\log S^2]_t)$.  

The payoffs are defined via log-price increments on an equidistant grid $0=t_0<t_1<\ldots <t_n=T$, where we set $T = 1$ and $n = 252$, thinking of approximately $252$ trading days in one year. We write $X_t^{i}=\log S_t^{i}$, and set $\Delta X_k^{i}=X^{i}_{t_{k+1}} - X^{i}_{t_k}$, $k=0,\ldots,n-1$. The realized variance and realized volatility for asset $i$ are then defined by 
\begin{equation*}
  \mathrm{RVar}^{i}_T:=\sum_{k=0}^{n-1}(\Delta X_k^{i})^2,\qquad \mathrm{RV}^{i}_T:=\sqrt{\sum_{k=0}^{n-1}(\Delta X_k^{i})^2},\qquad i=1,2,
\end{equation*}
and we consider the payoffs
\begin{equation*}
  \mathrm{RVswap}^{i}=\mathrm{RVar}_T^{i}-\mathrm{K}^i_{\mathrm{RVar}},\qquad
  \mathrm{RVcall}^{i}=(\mathrm{RV}_T^{i}-\mathrm{K}^i_{\mathrm{RV}})^{+}, \qquad i=1,2,
\end{equation*}
where the strikes $\mathrm{K}^i_{\mathrm{RVar}}$ and $\mathrm{K}^i_{\mathrm{RV}}$ are determined from the training sample. For the two-asset payoffs, we use
\begin{equation*}
  \mathrm{Cov}_T=\sum_{k=0}^{n-1}\Delta X_k^{1}\Delta X_k^{2},\qquad \mathrm{Corr}_T=\frac{\sum_{k=0}^{n-1}\Delta X_k^{1}\Delta X_k^{2}}{\sqrt{\sum_{k=0}^{n-1}(\Delta X_k^{1})^2}\sqrt{\sum_{k=0}^{n-1}(\Delta X_k^{2})^2}},
\end{equation*}
and consider the payoffs
\begin{align*}
  \mathrm{CovSwap}&=\mathrm{Cov}_T-\mathrm{K_{Cov}},\quad \mathrm{CovCall}=(\mathrm{Cov}_T-\mathrm{K_{Cov}})^{+},\\
  \mathrm{CorrSwap}&=\mathrm{Corr}_T-\mathrm{K_{Corr}},\quad \mathrm{CorrCall}=(\mathrm{Corr}_T-\mathrm{K_{Corr}})^{+},
\end{align*}
where again the strikes $\mathrm{K_{Cov}}$ and $\mathrm{K_{Corr}}$ are determined from the training sample.

Let $\widehat X$ denote the driving path used for signature features (time-extended for Stratonovich; extended by time and quadratic (co-)variation for It{\^o}), and let $\hbbX^{N}$ be its signature truncated at level $N=2$. For a training set of size $N_{\mathrm{train}}=15000$ consisting of i.i.d.~paths $\{\widehat X^{i}\}_{i=1}^{N_{\mathrm{train}}}$ with corresponding payoffs $F(X^{i})$, we fit (separately for each feature map) a linear model on signature features via ridge regression:
\begin{equation*}
  \ell^\ast \;\in\; \arg\min_{\ell\in\mathbb{R}^{d^\ast}}\;
  L_\alpha(\ell),\qquad 
  L_\alpha(\ell):=\frac{1}{N_{\mathrm{train}}}\sum_{i=1}^{N_{\mathrm{train}}}
  \Big(F(X^{(i)})-\langle \ell,\hbbX^{N,i}\rangle\Big)^{2}
  \;+\;\alpha\|\ell\|_2^{2},
\end{equation*}
with $\alpha=10^{-6}$ and $\hbbX^{N,i}$ denoting the truncated signature of the $i$-th path. We then evaluate the model on an independent test set of size $N_{\text{test}}=5000$, consisting of i.i.d.~paths $\{\hX^i\}_{i=1}^{N_{\text{test}}}$ and report the resulting out-of-sample MSEs. Pricing is then performed on an independent sample, generated by Monte Carlo simulation $\{\widehat X^{j}_{\mathrm{MC}}\}_{j=1}^{N_{\mathrm{MC}}}$ of size $N_{\mathrm{MC}}=25000$, via
\begin{equation*}
  \widehat P = \frac{1}{N_{\mathrm{MC}}}\sum_{j=1}^{N_{\mathrm{MC}}}
  \langle \ell^\ast, \hbbX^{N,j}_{\mathrm{MC}}\rangle,
\end{equation*}
and we compare $\widehat P$ to the Monte Carlo price $\frac{1}{N_{\mathrm{MC}}} \sum_{j=1}^{N_{\mathrm{MC}}} F(X^{j}_{\mathrm{MC}})$, which we take as a benchmark, and also include its $95\%$ confidence interval.

\begin{example}
  We consider a two-asset Heston model given by
  \begin{align*}
    \dd S_t^{i}&=\mu^iS_t^i\dd t+S_t^i\sqrt{V_t^i}\dd B_t^i,\\
    \dd V_t^i&=\kappa^i(\theta^i-V_t^i)\dd t+\sigma^i \sqrt{V_t^i}\dd W_t^i, \quad i=1,2,
  \end{align*}
  with $[B^1,B^2]_t=0.3t$, $[W^1,W^2]_t=0.5t$, $[B^1,W^1]_t=-0.6t$ and $[B^2,W^2]_t=-0.5t$. Moreover, we set 
  \begin{align*}    
    &\{S_0^1,S_0^2,V_0^1,V_0^2,\mu^1,\mu^2,\kappa^1,\kappa^2,\theta^1,\theta^2,\sigma^1,\sigma^2\}\\
    &=\{100,80,0.04,0.09,0.0,0.0,2.0,1.8,0.04,0.09,0.5,0.6\}.
  \end{align*}
  \begin{figure}[h]
    \centering
    \begin{subfigure}[h]{0.48\linewidth}
    \centering
    \includegraphics[width=\linewidth]{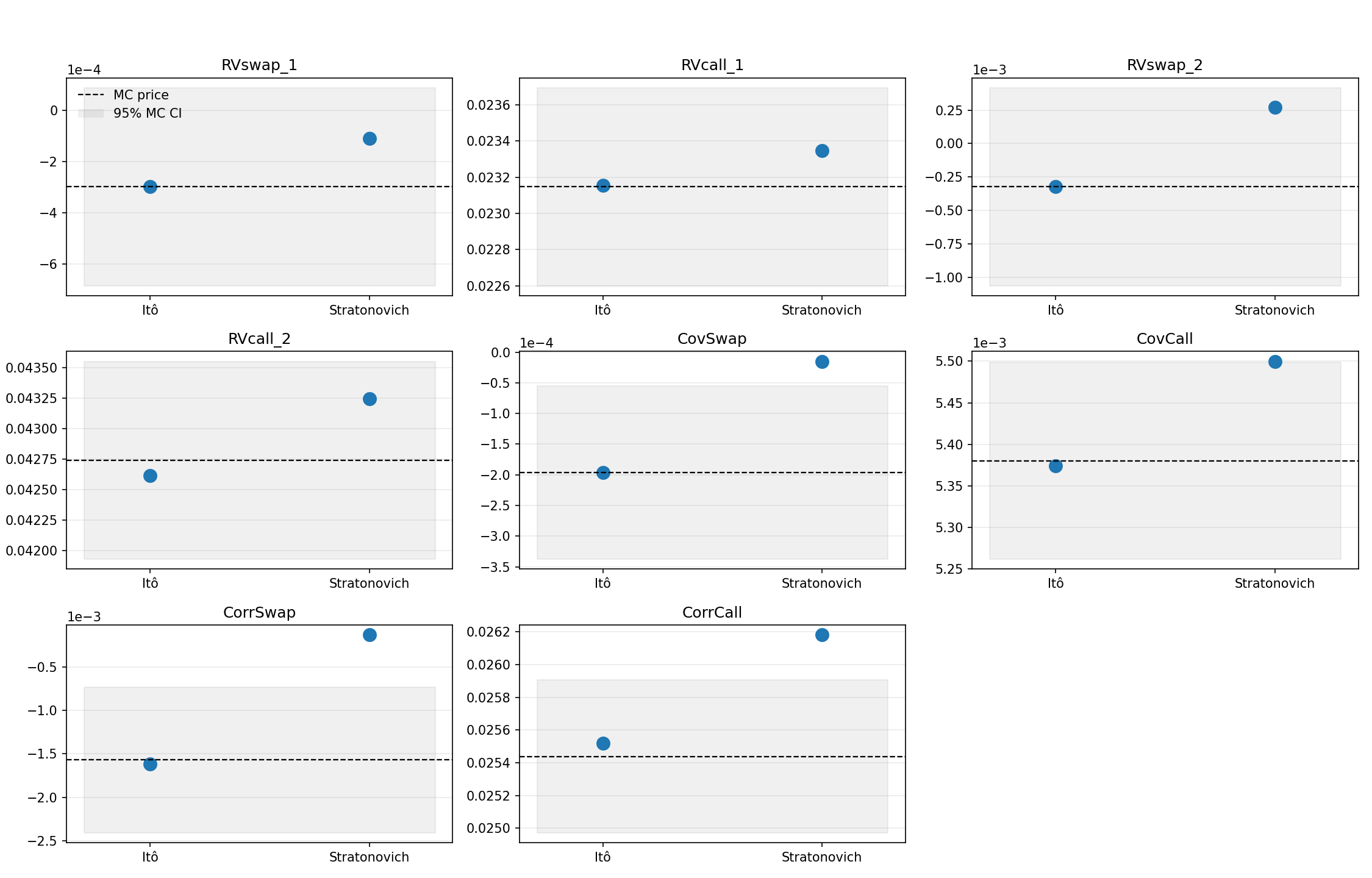}
    \caption{Prices from signature-regressed payoffs and MC price.}
    \label{fig: price_heston}
    \end{subfigure}\hfill
    \begin{subfigure}[h]{0.48\linewidth}
    \centering
    \includegraphics[width=\linewidth]{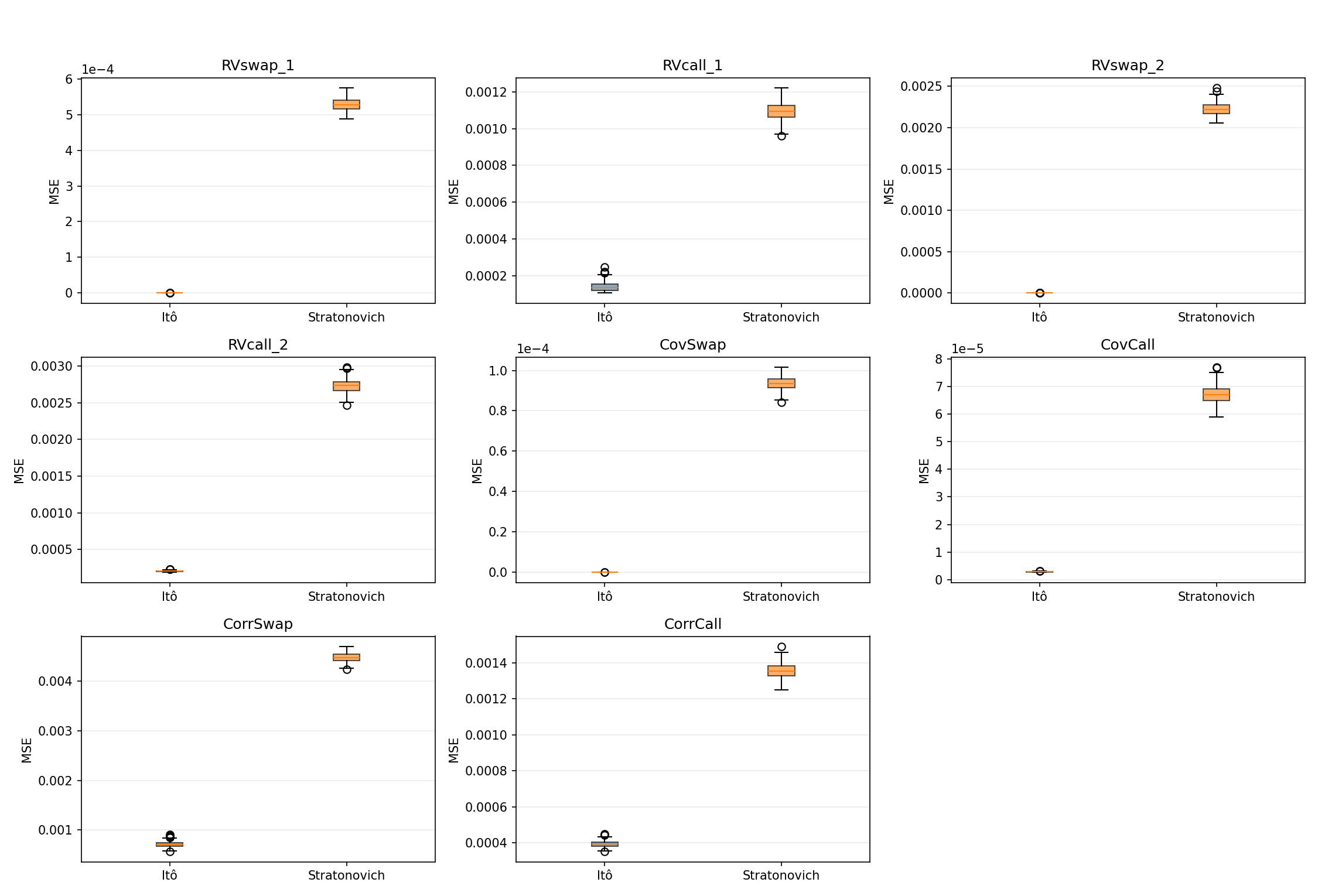}
    \caption{Boxplots of out-of-sample MSEs.}
    \label{fig: mse_heston}
    \end{subfigure}
    \caption{Signature-based payoff regression and pricing under the two-asset Heston model as defined above.}
  \end{figure}
  
  Figure~\ref{fig: mse_heston} shows that, for any payoff, using the It{\^o} signature results in smaller out-of-sample MSEs than using the Stratonovich signature. The boxplots also indicate both a smaller bias and a smaller out-of-sample variance. The prices, see Figure~\ref{fig: price_heston}, estimated via the It{\^o} signature, essentially coincide with the Monte Carlo price, and lie within the $95\%$ confidence interval. The prices, estimated via the Stratonovich signature, are systematically biased---especially for the two-asset payoffs (CovSwap/CovCall and CorrSwap/CorrCall), where some prices also lie outside the confidence interval. This underlines the advantage of extending the path by quadratic variation, as we do it for the It{\^o} features, because this allows to directly access $[\log S^1]$, $[\log S^2]$, $[\log S^1,\log S^2]$ in the regression, which are exactly the statistics the payoffs we consider depend on. Using the Stratonovich signature of the time-extended path, this is not the case and the information has to be inferred from the iterated integrals of order up to $2$.
\end{example}

\begin{example}
  We consider a two-asset singular time-changed SDE (with time-change given by the Cantor clock), 
  \begin{equation*}
    \dd S_t^i=\sigma(S_t^i)\dd W_{C(t)}^i,\quad i=1,2,
  \end{equation*}
  with $[W_C^1,W_C^2]_t=\rho\, C(t)$, $\rho=0.6$, $\sigma(S_t^i)=\nu_i S_t^i$, $\nu_i\in\R$. In our experiments, we set
  \begin{equation*}
    S^{1}_0=100,\qquad S^{2}_0=80, \qquad 
    \nu_1=0.20, \qquad \nu_2=0.30.
  \end{equation*}
  \begin{figure}[h!]
    \centering
    \begin{subfigure}[h]{0.48\linewidth}
    \centering
    \includegraphics[width=\linewidth]{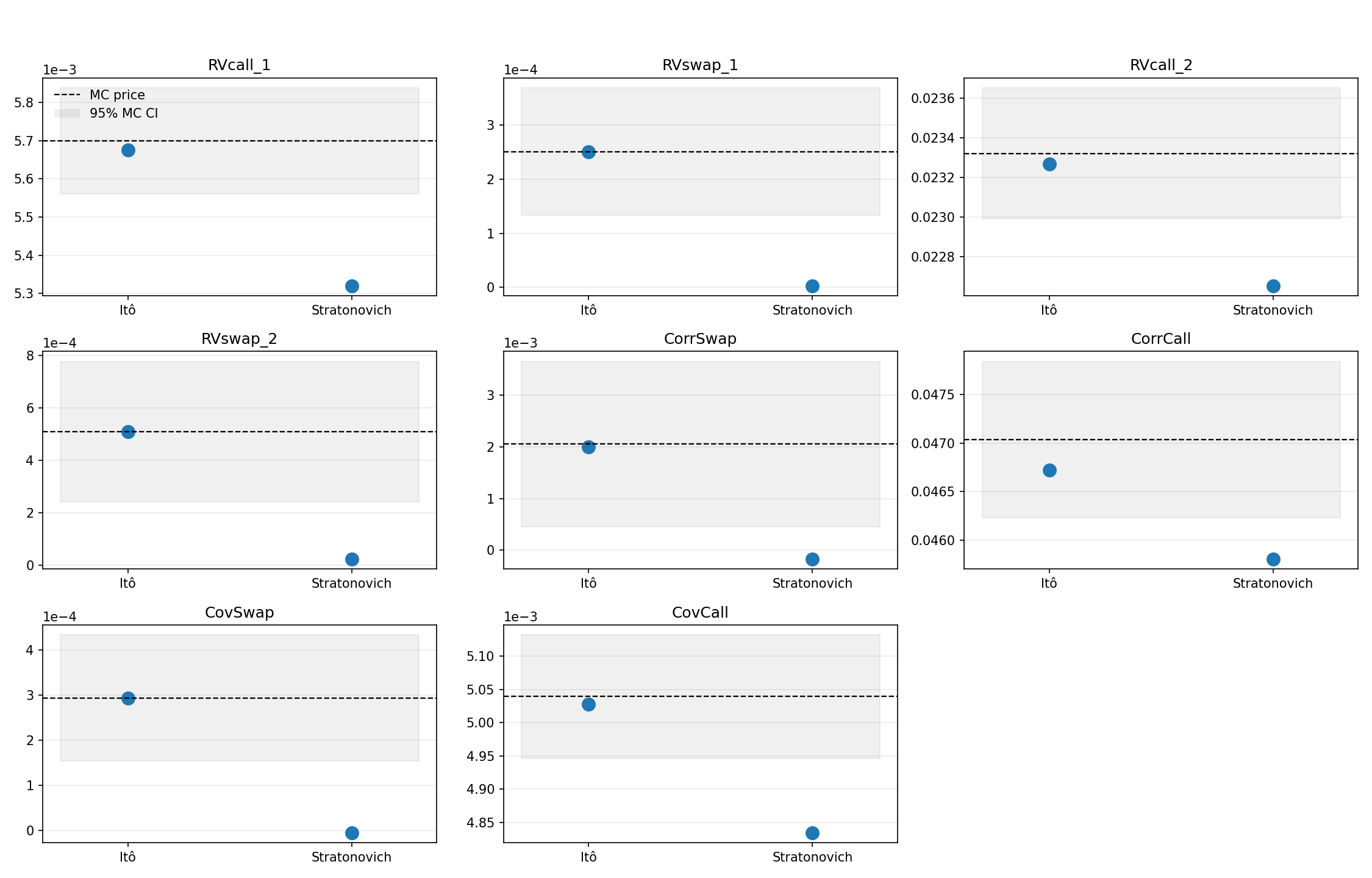}
    \caption{Prices from signature-regressed payoffs and MC price.}
    \label{fig: price_cantor}
    \end{subfigure}\hfill
    \begin{subfigure}[h]{0.48\linewidth}
    \centering
    \includegraphics[width=\linewidth]{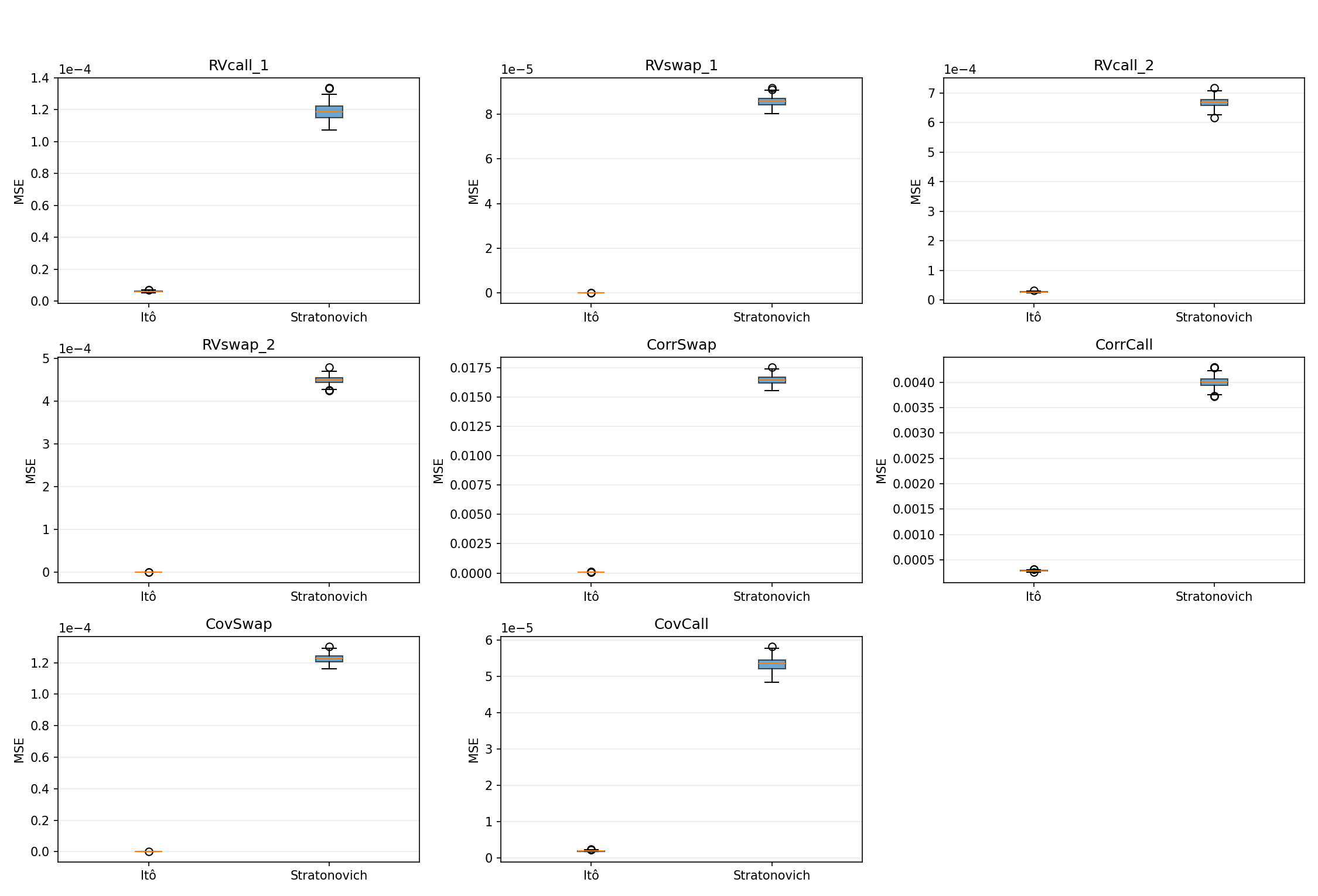}
    \caption{Boxplots of out-of-sample MSEs.}
    \label{fig: mse_cantor}
    \end{subfigure}
    \caption{Signature-based payoff regression and pricing under the two-asset singular-time changed SDE described above.}
  \end{figure}
  
  In this model, the differences between using It{\^o} signatures and Stratonovich signatures become even more clear. When using the It\^o signature, the out-of-sample MSEs are minimal, and the prices lie within the $95\%$ MC confidence interval, for any payoff considered. When using the Stratonovich signature, however, the out-of-sample MSEs are significantly larger, see Figure~\ref{fig: mse_cantor}, and  the prices do not lie within the confidence interval for any payoff considered, see Figure~\ref{fig: price_cantor}.
\end{example}

\appendix
\section{Proof of Theorem~\ref{thm: UAT geometric}}\label{appendix: proof of UAT geometric}

In this appendix, for completeness, we present the proof of the universal approximation theorem using time-extended weakly geometric rough paths (Theorem~\ref{thm: UAT geometric}).

\begin{proof}[Proof of Theorem~\ref{thm: UAT geometric}]
  The result follows by an application of the Stone--Weierstrass theorem to the set
  \begin{equation*}
	\mathcal A := \spn \{K \ni \hbbX^{o, 2} \mapsto \langle e_I,\hbbX^o_T \rangle \in \R: I \in \{0,\ldots,d\}^N, N \in \N_0\}.
  \end{equation*}
    
  Therefore, we have to show that $\mathcal A$
  \begin{enumerate}
    \item[(i)] is a vector subspace of $C(K;\R)$,
	\item[(ii)] is a subalgebra and contains a non-zero constant function, and
	\item[(iii)] separates points.
  \end{enumerate}
  
  \emph{(i):} By~\cite[Corollary 9.11]{Friz2010}, the map $\hbbX^{o, 2}\mapsto \langle e_I,\hbbX^o_T\rangle$ is continuous on bounded sets for every multi-index $I$ with respect to $d_{p\textup{-var}} := \|\cdot \, ; \cdot\|_p$. More precisely, the map
  \begin{equation*}
	(K,d_{p\textup{-var}}) \ni \hbbX^{o, 2} \mapsto \hbbX^{o, N} \in (C_o^{p\textup{-var}}([0,T];G^N(\R^{1+d})),d_{p\textup{-var}}),
  \end{equation*}
  is continuous on $K$ with respect to $d_{p\textup{-var}}$, for every $N \geq 3$. Moreover, the evaluation map
  \begin{equation*}
	(C_o^{p\textup{-var}}([0,T];G^N(\R^{1+d})),d_{p\textup{-var}}) \ni \hbbX^{o, N} \mapsto \hbbX^{o, N}_T\in (G^N(\R^{1+d}),\rho)
  \end{equation*}
  is continuous, where $\rho$ denotes the metric induced by the norm on $T_1^N(\R^{1+d})$. Here, we used that we can equip $G^N(\R^{1+d})$ with the metric $\rho$, see e.g.~\cite[Remark~7.31]{Friz2010}. This yields that the map
  \begin{equation*}
	(K,d_{p\textup{-var}})\ni\hbbX^{o, 2} \mapsto \hbbX^{o, N}_T\in (G^N(\R^{1+d}),\rho)
  \end{equation*}
  is continuous. Since $\hbbX^{o, N}_T \mapsto \langle e_I, \hbbX^o_T \rangle$ is continuous for any multi-index $I$, we can thus conclude that the map
  \begin{equation*}
	(K,d_{p\textup{-var}}) \ni \hbbX^{o, 2} \mapsto \langle e_I, \hbbX_T^o \rangle \in \R
  \end{equation*}
  is continuous with respect to $d_{p\textup{-var}}$.
	
  \emph{(ii):} Since $\hbbX^o_T$ is a group-like element, i.e.,~$\hbbX^o_T \in G((\R^{1+d}))$, the shuffle property holds, and thus $\mathcal A$ is a subalgebra. Moreover, since $\langle e_\emptyset, \hbbX^o_T \rangle = 1$, it contains a non-zero constant function.
	
  \emph{(iii):} For the point separation, let us consider $\hbbX^{o, 2}$, $\hbbY^{o, 2}\in K$ with $\hbbX^{o, 2} \neq \hbbY^{o, 2}$. We show that there exists a $k \in \N$, $I \in \{0,\ldots,d\}^N$, $N \in \{0,1,2\}$, such that
  \begin{equation*}
	\langle (e_I \shuffle e_0^{\otimes k}) \otimes e_0, \hbbX^o_T \rangle \neq \langle (e_I \shuffle e_0^{\otimes k}) \otimes e_0, \hbbY^o_T \rangle.
  \end{equation*}
  We proceed with a proof by contradiction. Assume that for all $k \in \N$, $I \in \{0, \ldots, d\}^N$, $N \in \{0,1,2\}$, we have
  \begin{equation*}
	\langle (e_I \shuffle e_0^{\otimes k}) \otimes e_0, \hbbX^o_T \rangle = \langle (e_I \shuffle e_0^{\otimes k}) \otimes e_0, \hbbY^o_T\rangle.
  \end{equation*}
  We first note that 
  \begin{equation*}
	\langle e_0^{\otimes k}, \hbbX^o_t \rangle = \frac{t^k}{k!}.
  \end{equation*}
  Moreover, using the shuffle property, we have by e.g.~\cite[Proposition~C.5]{Cuchiero2023b}, that
  \begin{equation*}
	\langle (e_I \shuffle e_0^{\otimes k}) \otimes e_0, \hbbX^o_T \rangle 
	= \int_0^T \langle e_I, \hbbX^o_t \rangle \langle e_0^{\otimes k},\hbbX^o_t \rangle \dd t
	= \int_0^T \langle e_I, \hbbX^o_t \rangle \frac{t^k}{k!} \dd t.
  \end{equation*}
  Similarly, we have
  \begin{equation*}
	\langle (e_I \shuffle e_0^{\otimes k}) \otimes e_0, \hbbY^o_T \rangle = \int_0^T \langle e_I, \hbbY^o_t \rangle \frac{t^k}{k!} \dd t.
  \end{equation*}
  By \cite[Corollary~4.24]{Brezis2011} and because $\langle e_I, \hbbX^o_0 \rangle = \langle e_I, \hbbY^o_0 \rangle = 0$, it then follows that 
  \begin{equation*}
	\langle e_I, \hbbX^o_t \rangle = \langle e_I, \hbbY^o_t \rangle,
  \end{equation*}
  for all $t \in [0,T]$ and all $I \in \{0, \ldots, d\}^N$, $N \in \{0,1,2\}$. However, this contradicts the assumption that $\hbbX^{o, 2}$, $\hbbY^{o, 2}$ are distinct. Thus, we can conclude that $\mathcal{A}$ is point-separating.
\end{proof}

\section{Proof of Lemma~\ref{lemma: stability of gamma RIE}}\label{appendix: proof of lemma}

In this appendix we present the proof of the auxiliary Lemma~\ref{lemma: stability of gamma RIE}.

\begin{proof}[Proof of Lemma~\ref{lemma: stability of gamma RIE}]
  For $\gamma \neq \frac{1}{2}$, the statement follows from Lemma~\ref{lemma: RIE eq to gamma RIE} and~\cite[Proposition~2.12]{Allan2023c}.
	
  Suppose that $\gamma = \frac{1}{2}$. We need to verify that the integral 
  \begin{equation*}
    \int_0^t \hX_r \otimes \d^{\frac{1}{2},\pi^n} \hX_r = \int_0^t X_r \otimes \d^{\frac{1}{2},\pi^n} X_r + \int_0^t X_r \otimes \d^{\frac{1}{2},\pi^n} \phi_r + \int_0^t \phi_r \otimes \d^{\frac{1}{2},\pi^n} X_r + \int_0^t \phi_r \otimes \d^{\frac{1}{2},\pi^n} \phi_r,
  \end{equation*}
  converges as $n \to \infty$ to the limit
  \begin{equation*}
	\int_0^t \hX_r \otimes \d^{\frac{1}{2},\pi} X_r = \int_0^t X_r \otimes \d^{\frac{1}{2},\pi} X_r + \int_0^t X_r \otimes \d^{\frac{1}{2},\pi} \phi_r + \int_0^t \phi_r \otimes \d^{\frac{1}{2},\pi} X_r + \int_0^t \phi_r \otimes \d^{\frac{1}{2},\pi} \phi_r,
  \end{equation*}
  uniformly in $t \in [0,T]$, where the latter three integrals are given as Young integrals.
	
  Since $X$ satisfies Property $1/2$-\textup{(RIE)}, we have that
  \begin{equation*}
	\bigg \|\int_0^\cdot X_r \otimes \d^{\frac{1}{2},\pi^n} X_r - \int_0^\cdot X_r \otimes \d^{\frac{1}{2},\pi} X_r \bigg \|_\infty \longrightarrow 0 \qquad \text{as} \qquad n \to \infty.
  \end{equation*}
  Define $\bar{X}^n$ and $\bar{\phi}^n$ as the piecewise linear interpolation of $X$ and $\phi$, respectively, along $\pi = (\pi^n)_{n \in \N}$. Then, it holds for any $t \in [0,T]$ that
  \begin{equation*}
	\int_0^t X_r \otimes \d^{\frac{1}{2},\pi^n} \phi_r = \sum_{k=0}^{N_n-1} (X_{t^n_k} + \frac{1}{2} X_{t^n_k, t^n_{k+1}}) \otimes \phi_{t^n_k \wedge t, t^n_{k+1} \wedge t} = \int_0^t \bar{X}^n_r \otimes \d \phi_r.
  \end{equation*}
  Let $p' > p$ such that $1/p' + 1/q > 1$. By the standard estimate for Young integrals---see e.g.~\cite[Proposition~2.4]{Friz2018}---we have for all $t \in [0,T]$, that
  \begin{equation*}
	\bigg \lvert \int_0^t X_r \otimes \d^{\frac{1}{2},\pi^n} \phi_r - \int_0^t X_r \otimes \d^{\frac{1}{2},\pi} \phi_r \bigg \rvert \lesssim \|\bar{X}^n - X\|_{p'} \|\phi\|_q.
  \end{equation*}
  It follows by interpolation---see e.g.~\cite[Proposition~5.5]{Friz2010}---that
  \begin{equation*}
	\|\bar{X}^n - X\|_{p'} \leq \|\bar{X}^n - X\|_\infty^{1 - \frac{p}{p'}} \|\bar{X}^n - X\|_p^{\frac{p}{p'}}.
  \end{equation*}
  Since $\bar{X}^n$ converges uniformly to $X$ as $n \to \infty$, and $\sup_{n \in \N} \|\bar{X}^n\|_p < \infty$, we deduce that
  \begin{equation*}
	\bigg\|\int_0^\cdot X_r \otimes \d^{\frac{1}{2},\pi^n} \phi_r - \int_0^\cdot X_r \otimes \d^{\frac{1}{2},\pi} \phi_r \bigg \|_\infty \, \longrightarrow \, 0 \qquad \text{as} \quad n \, \longrightarrow \, \infty.
  \end{equation*}
  Similarly, for each $t \in [0,T]$, it holds that
  \begin{equation*}
	\bigg|\int_0^t \phi_r \otimes \d^{\frac{1}{2},\pi^n} X_r - \int_0^t \phi_r \otimes \d^{\frac{1}{2},\pi} X_r \bigg| \lesssim \|\bar{\phi}^n - \phi\|_q \|X\|_p,
  \end{equation*}
  and
  \begin{equation*}
	\bigg|\int_0^t \phi_r \otimes \d^{\frac{1}{2},\pi^n} \phi_r - \int_0^t \phi_r \otimes \d^{\frac{1}{2},\pi} \phi_r \bigg| \lesssim \|\bar{\phi}^n - \phi\|_q \|\phi\|_q,
  \end{equation*}
  and, since $\|\bar{\phi}^n - \phi\|_q \to 0$ as $n \to \infty$, we infer the required convergence.
	
  We further aim to find a control function $c$ such that
  \begin{equation}\label{eq: condition (iii) for perturbed path}
	\sup_{(s,t) \in \Delta_T} \frac{|\hX_{s,t}|^p}{c(s,t)} + \sup_{n \in \N} \, \sup_{0 \leq k < \ell \leq N_n} \frac{|\int_{t^n_k}^{t^n_\ell} \hX_u \otimes \d^{\frac{1}{2},\pi^n} \hX_u - \hX_{t^n_k} \otimes \hX_{t^n_k,t^n_{k+1}}|^{\p}}{c(t_k^n,t_\ell^n)} \lesssim 1,
  \end{equation}
  where
  \begin{align*}
	&\int_{t^n_k}^{t^n_\ell} \hX_u \otimes \d^{\frac{1}{2},\pi^n} \hX_u - \hX_{t^n_k} \otimes \hX_{t^n_k,t^n_{k+1}} = \int_{t^n_k}^{t^n_\ell} \hX_{t^n_k,u} \otimes \d^{\frac{1}{2},\pi^n} \hX_u \\
	&\quad = \int_{t^n_k}^{t^n_\ell} X_{t^n_k,u} \otimes \d^{\frac{1}{2},\pi^n} X_u + \int_{t^n_k}^{t^n_\ell} X_{t^n_k,u} \otimes \d^{\frac{1}{2},\pi^n} \phi_u \\
	&\qquad + \int_{t^n_k}^{t^n_\ell} \phi_{t^n_k,u} \otimes \d^{\frac{1}{2},\pi^n} X_u + \int_{t^n_k}^{t^n_\ell} \phi_{t^n_k,u} \otimes \d^{\frac{1}{2},\pi^n} \phi_u.
  \end{align*}
  Let $c_X$ be the control function with respect to which $X$ satisfies Property $\gamma$-\textup{(RIE)}, and define moreover the control function $c_\phi$, given by $c_{\phi}(s,t) = \|\phi\|_{q,[s,t]}^q$ for $(s,t) \in \Delta_T$.
	
  We have from Property $1/2$-\textup{(RIE)} that
  \begin{equation*}
	\sup_{(s,t) \in \Delta_T} \frac{|\hX_{s,t}|^p}{c_X(s,t) + c_\phi(s,t)} \lesssim \sup_{(s,t) \in \Delta_T} \frac{|X_{s,t}|^p}{c_X(s,t)} + \sup_{(s,t) \in \Delta_T} \frac{|\phi_{s,t}|^p}{c_\phi(s,t)} \lesssim 1,
  \end{equation*}
  and that
  \begin{equation*}  
	\sup_{n \in \N} \, \sup_{0 \leq k < \ell \leq N_n} \frac{|\int_{t_k^n}^{t_\ell^n} X_u \otimes \d^{\frac{1}{2},\pi^n} X_u - X_{t^n_k} \otimes X_{t^n_k,t^n_{k+1}}|^{\p}}{c_X(t_k^n,t_\ell^n)} \lesssim 1.
  \end{equation*}
  By the standard estimate for Young integrals (see e.g.~\cite[Proposition~2.4]{Friz2018}), for every $n \in \N$ and $0 \leq k < \ell \leq N_n$, we have
  \begin{align*}
	\bigg|\int_{t_k^n}^{t_\ell^n} \bar{X}_{t_k^n,u}^n \otimes \d \phi_u \bigg|^{\p} &\lesssim \|\bar{X}^n\|_{p,[t_k^n,t_\ell^n]}^{\p} \|\phi\|_{q,[t_k^n,t_\ell^n]}^{\p}\\
	&\leq \|X\|_{p,[t_k^n,t_\ell^n]}^{\p} \|\phi\|_{q,[t_k^n,t_\ell^n]}^{\p} \leq c_X(t_k^n,t_\ell^n)^{\frac{1}{2}} c_\phi(t_k^n,t_\ell^n)^{\frac{p}{2q}},
  \end{align*}
  and we can similarly obtain
  \begin{equation*}
	\bigg|\int_{t_k^n}^{t_\ell^n} \bar{\phi}^n_{t_k^n,u} \otimes \d X_u \bigg|^{\p} \lesssim c_X(t_k^n,t_\ell^n)^{\frac{1}{2}} c_\phi(t_k^n,t_\ell^n)^{\frac{p}{2q}}
  \end{equation*}
  and
  \begin{equation*}
	\bigg|\int_{t_k^n}^{t_\ell^n} \bar{\phi}_{t_k^n,u}^n \otimes \d \phi_u\bigg|^{\p} \lesssim c_\phi(t_k^n,t_\ell^n)^{\frac{p}{q}}.
  \end{equation*}
  Since $p \in (2,3)$ and $q \in [1,2)$, we have that $1/2 + p/2q > 1$ and $p/q > 1$, and it follows that the maps $(s,t) \mapsto c_X(s,t)^{\frac{1}{2}} c_\phi(s,t)^{\frac{p}{2q}}$ and $(s,t) \mapsto c_\phi(s,t)^{\frac{p}{q}}$ are superadditive and thus control functions. We deduce that \eqref{eq: condition (iii) for perturbed path} holds with a control function $c$ of the form
  \begin{equation*}  
	c(s,t) = C\Big(c_X(s,t) + c_\phi(s,t) + c_X(s,t)^{\frac{1}{2}} c_\phi(s,t)^{\frac{p}{2q}} + c_\phi(s,t)^{\frac{p}{q}}\Big), \qquad (s,t) \in \Delta_T,
  \end{equation*}
  where $C > 0$ is a suitable constant which depends only on $p$ and $q$.
\end{proof}

\section{On Lyons' extension theorem}\label{sec: finite p-variation}

We prove that Lyons' extension of a rough path (that is not necessarily weakly geometric) for $p \in (2,3)$, see e.g.~\cite[Theorem~3.7]{Lyons2007}, coincides with the collection of iterated integrals defined via rough integration with respect to controlled paths.

\begin{proposition}\label{prop: Lyons extension coincides with iterated integrals}
  Let $p \in (2,3)$. Let $\bX = (X,\X^{(2)})$ be a rough path such that
  \begin{equation*}
    |X_{s,t}| \lesssim c(s,t)^{\frac 1 p}, \qquad |\X^{(2)}_{s,t}| \lesssim c(s,t)^{\frac 2 p}, \qquad (s,t) \in \Delta_T,
  \end{equation*}
  for some control function $c$. For $(s,t) \in \Delta_T$, $N > 2$, consider the iterated integral of order $N$, i.e., $\X^{(N)}_{s,t} := \int_s^t (\X^{(N-1)}_{s,\cdot})_r \otimes \d \bX$ as a rough integral of $r \mapsto (\X^{(N-1)}_{s,\cdot})_r = \X^{(N-1)}_{s,r}$, which is a controlled path w.r.t.~$X$ on $[s,t]$, with respect to $\bX$. (For the definition of the rough integral, we particularly refer to Remark~\ref{rem: tensored controlled path}.) Then,
  \begin{equation*}
    (s,t) \mapsto (1,X_{s,t},\X^{(2)}_{s,t},\dots,\X^{(N)}_{s,t}) \in T^N(\R^d)
  \end{equation*}
  satisfies $|\X^{(N)}_{s,t}| \lesssim c(s,t)^{\frac N p}$, $(s,t) \in \Delta_T$, and coincides with Lyons' extension of $(1,X_{s,t},\X^{(2)}_{s,t})$ to $T^N(\R^d)$, as given in e.g.~\cite[Theorem~3.7]{Lyons2007}, for any $N > 2$.
\end{proposition}

\begin{proof}
  Let $N = 3$. We note that 
  \begin{equation*}
    \X^{(N-1)}_{s,v} - \X^{(N-1)}_{s,u} = \X^{(2)}_{s,v} - \X^{(2)}_{s,u} = \X^{(2)}_{u,v} + X_{s,u} \otimes X_{u,v}, \qquad s \leq u \leq v \leq t.
  \end{equation*}
  That is, $r \mapsto \X^{(2)}_{s,r}$ is a controlled path w.r.t.~$X$ on $[s,t]$, with $(\X^{(2)}_{s,\cdot})'_u = X_{s,u}$ and $R^{\X^{(2)}_{s,\cdot}}_{u,v} = \X^{(2)}_{u,v}$.
    
  To show the existence of the rough integral of a controlled path with respect to a rough path, we set $A_{u,v} := \X^{(2)}_{s,u} \otimes X_{u,v} + (\X^{(2)}_{s,\cdot})'_u \otimes \X^{(2)}_{u,v}$ and $\delta A_{u,v,w} := A_{u,w} - A_{u,v} - A_{v,w}$ for $s \leq u \leq v \leq w \leq t$. We then have that
  \begin{align*}
    &|\delta A_{u,v,w}| \\
    &\quad = |- R^{\X^{(2)}_{s,\cdot}}_{u,v} \otimes X_{v,w} - ((\X^{(2)}_{s,\cdot})'_v - (\X^{(2)}_{s,\cdot})'_u) \otimes \X^{(2)}_{v,w}| \\
    &\quad = |- \X^{(2)}_{u,v} \otimes X_{v,w} - X_{u,v} \otimes \X^{(2)}_{v,w}| \\ 
    &\quad \lesssim c(u,v)^{\frac 2 p} c(v,w)^{\frac 1 p} + c(u,v)^{\frac 1 p} c(v,w)^{\frac 2 p} \lesssim c(u,w)^{\frac 3 p}.
  \end{align*}
  Therefore, by the sewing lemma, we obtain the estimate
  \begin{equation*}
    |\X^{(3)}_{s,t}| 
    =  \bigg|\int_s^t (\X^{(2)}_{s,\cdot})_r \otimes \d \bX_r  \bigg|
    =   \bigg|\int_s^t (\X^{(2)}_{s,\cdot})_r \otimes \d \bX_r - (\X^{(2)}_{s,\cdot})_s \otimes X_{s,t} - (\X^{(2)}_{s,\cdot})'_s \otimes \X^{(2)}_{s,t}   \bigg| \lesssim c(s,t)^{\frac 3 p},
  \end{equation*}
  for $(s,t) \in \Delta_T$.
    
  We apply an inductive argument: Assuming that the claim holds true for any $n < N$, for $N > 3$, we now let $n \leq N$. We begin by noting that $r \mapsto Y_r := \X^{(n-1)}_{s,r}$ is a controlled path w.r.t.~$X$ on $[s,t]$ (as a rough integral) and we observe that, by Chen's relation,
  \begin{equation*}
    Y_v - Y_u = \X^{(n-1)}_{s,v} - \X^{(n-1)}_{s,u} = \sum_{j=0}^{n-3} \X^{(j)}_{s,u} \otimes \X^{(n-1-j)}_{u,v} + \X^{(n-2)}_{s,u} \otimes X_{u,v},
  \end{equation*}
  for $s \leq u \leq v \leq t$. That is, $Y'_u = \X^{(n-2)}_{s,u}$, and $R^Y_{u,v} = \sum_{j=0}^{n-3} \X^{(j)}_{s,u} \otimes \X^{(n-1-j)}_{u,v}$. Analogously to above, we derive for $A_{u,v} := Y_u \otimes X_{u,v} + Y'_u \otimes \X^{(2)}_{u,v}$ and $\delta A_{u,v,w} := A_{u,w} - A_{u,v} - A_{v,w}$ that
  \begin{align*}
    &|\delta A_{u,v,w}| = |-R^Y_{u,v} \otimes X_{v,w} - Y'_{u,v} \otimes \X^{(2)}_{v,w}| \\
    &\quad = \bigg|-(\sum_{j=0}^{n-3} \X^{(j)}_{s,u} \otimes \X^{(n-1-j)}_{u,v}) \otimes X_{v,w} - (\sum_{j=0}^{n-3} \X^{(j)}_{s,u} \otimes \X^{(n-2-j)}_{u,v}) \otimes \X^{(2)}_{v,w} \bigg| \\
    &\quad = \bigg|-(\sum_{j=0}^{n-3} \X^{(j)}_{s,u} \otimes (\X^{(n-1-j)}_{u,v} \otimes X_{v,w} + \X^{(n-2-j)}_{u,v} \otimes \X^{(2)}_{v,w}) \bigg| \\
    &\quad \lesssim  \sum_{j=0}^{n-3}c(s,u)^{\frac{j}{p}}(c(u,v)^{\frac{n-1-j}{p}} c(v,w)^{\frac{1}{p}}+c(u,v)^{\frac{n-2-j}{p}}c(v,w)^{\frac{2}{p}})\\
    &\quad\le \sum_{j=0}^{n-3}c(s,t)^{\frac{j}{p}}c(u,v)^{\frac{n-1-j}{p}} c(v,w)^{\frac{1}{p}}+c(s,t)^{\frac{j}{p}}c(u,v)^{\frac{n-2-j}{p}}c(v,w)^{\frac{2}{p}}\\
    &\quad= \sum_{j=0}^{n-3} w_{1,1,j}^{\alpha_{1,1}}(u,v) w_{1,2}^{\alpha_{1,2}}(v,w) + w_{2,1,j}^{\alpha_{2,1}}(u,v) w_{2,2}^{\alpha_{2,2}}(v,w),
  \end{align*}
  where
  \begin{align*}
    w_{1,1,j} = c(s,t)^{\frac{p}{n-1-j} \frac{j}{p}} c, \quad \alpha_{1,1} = \frac{n-1-j}{p}, \quad w_{1,2} = c, \quad \alpha_{1,2} = \frac{1}{p}, \\
    w_{2,1,j} = c(s,t)^{\frac{p}{n-2-j} \frac{j}{p}} c, \quad \alpha_{2,1} = \frac{n-2-j}{p}, \quad w_{2,2}=c, \quad \alpha_{2,2} = \frac{2}{p}.
  \end{align*}
  Then, by the generalized sewing lemma, see e.g.~\cite[Theorem~2.5]{Friz2018}, we have
  \begin{align*}
    |\X^{(n)}_{s,t}| 
    &= \bigg|\int_s^t (\X^{(n-1)}_{s,\cdot})_r \otimes \d \bX_r\bigg| 
    \\
    &= \bigg|\int_s^t (\X^{(n-1)}_{s,\cdot})_r \otimes \d \bX_r - (\X^{(n-1)}_{s,\cdot})_s \otimes X_{s,t} - (\X^{(n-1)}_{s,\cdot})'_s \otimes \X^{(n-1)}_{s,t}\bigg| \\
    &\lesssim \sum_{j=0}^{n-3} w_{1,1,j}^{\alpha_{1,1}}(s,t) w_{1,2}^{\alpha_{1,2}}(s,t) + w_{2,1,j}^{\alpha_{2,1}}(s,t) w_{2,2}^{\alpha_{2,2}}(s,t)\\
    &=\sum_{j=0}^{n-3}c(s,t)^{\frac{j}{p}}c(s,t)^{\frac{n-1-j}{p}} c(s,t)^{\frac{1}{p}}+c(s,t)^{\frac{j}{p}}c(s,t)^{\frac{n-2-j}{p}}c(s,t)^{\frac{2}{p}}\\
    &\lesssim_n c(s,t)^{\frac{n}{p}}.
  \end{align*} 
    
  Further, we notice that $(s,t) \mapsto (1,X_{s,t},\X^{(2)}_{s,t},\dots,\X^{(N)}_{s,t}) \in T^N(\R^d)$ is a multiplicative functional by definition.
    
  Now take $\widetilde{\X}^N$ to be Lyons' extension of $(1,X_{s,t},\X^{(2)}_{s,t})$ to $T^{N}(\R^d)$ for any $N > 3$. We know that both $\X^3$ and $\widetilde{\X}^3$ are of finite $p$-variation controlled by $c$. Since $\X^2 = \widetilde{\X}^2$, it holds that
  \begin{equation*}
    |\X^3_{s,t} - \widetilde{\X}^3_{s,t}| = |\X^{(3)}_{s,t} - \widetilde{\X}^{(3)}_{s,t}| \leq C c(s,t)^{\frac{3}{p}}, \qquad (s,t) \in \Delta_T,
  \end{equation*}
  where $C$ denotes the sum of the two implicit multiplicative constants in the regularity estimates of $\X^{(3)}$ and $\widetilde{\X}^{(3)}$, respectively. By~\cite[Lemma~3.4]{Lyons2007}, the map $(s,t) \mapsto \X^{(3)}_{s,t} - \widetilde{\X}^{(3)}_{s,t}$ is additive in $(\R^d)^{\otimes 3}$, so that the path $\X^{(3)}_{0,\cdot} - \widetilde{\X}^{(3)}_{0,\cdot}$ is of finite $p/3$-variation in $(\R^d)^{\otimes 3}$ starting at zero. Due to the regularity of $c$ and since $p/3 < 1$, it then follows that $(s,t) \mapsto \X^{(3)}_{s,t} - \widetilde{\X}^{(3)}_{s,t}$ is equal to zero, that is $\X^3$ equals $\X^{(3)}$. We apply an induction argument on $N$: Because the arguments carry through, we conclude that $\X^N = \widetilde{\X}^N$ for any $N \geq 1$, which completes the proof.
\end{proof}

\section{Some additional results in rough path theory}\label{appendix: some results}

We recall and extend some essential results on rough integration and rough paths.

\begin{lemma}[Proposition~2.4 in \cite{Allan2023b}]\label{lemma: integration against controlled path}
  Let $\bX = (X,\X) \in \cC^p([0,T];\R^d)$ and let $(F,F'), (G,G') \in \mathscr{C}^p_X$ be controlled paths with remainders $R^F$ and $R^G$, respectively. Then the limit
  \begin{equation}\label{eq: integration against controlled path}
    \int_0^T F_r \dd G_r := \lim_{|\cP| \to 0} \sum_{[s,t] \in \mathcal{P}} F_s G_{s,t} + F'_s G'_s \X_{s,t}
  \end{equation}
  exists along every sequence of partitions $\mathcal{P}$ of $[0,T]$ with mesh size $|\mathcal{P}| \to 0$, and comes with the estimate
  \begin{align*}
    &\bigg|\int_s^t F_r \dd G_r - F_s G_{s,t} - F'_s G'_s \X_{s,t} \bigg|\\
    &\quad \leq C\Big(\|F'\|_\infty (\|G'\|_{p,[s,t)}^p + \|X\|_{p,[s,t)}^p)^{\frac{2}{p}} \|X\|_{p,[s,t]} + \|F\|_{p,[s,t)} \|R^G\|_{\p,[s,t]}\\
    &\quad\quad\quad\quad + \|R^F\|_{\p,[s,t)} \|G'\|_\infty \|X\|_{p,[s,t]} + \|F'G'\|_{p,[s,t)} \|\X\|_{\p,[s,t]}\Big),
  \end{align*}
  for every $(s,t) \in \Delta_T$, where the constant $C$ depends only on $p$.
\end{lemma}

\begin{remark}\label{rem: tensored controlled path}
  For $m, n > 1$, suppose that $F \in C^{p\textup{-var}}([0,T];(\R^d)^{\otimes m-1})$, $F' \in C^{p\textup{-var}}([0,T];(\R^d)^{\otimes m})$, $G \in C^{p\textup{-var}}([0,T];(\R^d)^{\otimes n-1})$, $G' \in C^{p\textup{-var}}([0,T];(\R^d)^{\otimes n})$. Then,
  \begin{equation*}
    \int_0^T F_r \dd G_r := \int_0^T F_r \otimes \d G_r := \lim_{|\cP|\to 0} \sum_{[s,t] \in \cP} F_s \otimes G_{s,t} + (F'_s \otimes G'_s) \X_{s,t},
  \end{equation*}
  relative to the rough path $\bX = (X,\X)$. In writing $F'_s \otimes G'_{s,t}$, we technically mean the $m+n$-tensor whose component is given by $[F'_s \otimes G'_s]^{i_1 \dots i_m j_1 \dots j_n} = (F'_s)^{i_1\dots i_m} (G'_s)^{j_1 \dots j_n}$, and we interpret the ``multiplication'' $(F'_s \otimes G'_s) \X_{s,t}$ as the matrix whose $(m-1)+(n-1)$ component is given by $[(F'_s \otimes G'_s) \X_{s,t}]^{i_1 \dots i_{m-1} j_1 \dots j_{n-1}} = \sum_i \sum_j (F'_s)^{i_1 \dots i_{m-1} i} (G'_s)^{j_1 \dots j_{n-1} j} \X^{ij}$.
\end{remark}

Property $\gamma$-\textup{(RIE)} not only ensures the existence of a suitable rough path lift of a path, but also allows the rough integral to be expressed as a limit of Riemann sums, depending on $\gamma$. The next theorem is a slight generalization of \cite[Theorem~2.12]{Das2025}.

\begin{theorem}\label{thm: rough int against controlled path under RIE}
  Let $p \in (2,3)$, and let $\pi^n = \{0 = t^n_0 < t^n_1 < \dots < t^n_{N_n} = T\}$, $n \in \N$, be a sequence of partitions such that $|\pi^n| \to 0$ as $n \to \infty$. Suppose that $X \in C([0,T];\R^d)$ satisfies Property $\gamma$-\textup{(RIE)} relative to some $\gamma \in [0,1]$, $p$ and $\pi = (\pi^n)_{n \in \N}$, and let $\bX^{\gamma,\pi}$ be the canonical rough path lift of $X$, as constructed in Proposition~\ref{prop: rough path lift under gamma RIE}. Let $q > 0$ be such that $2/p + 1/q > 1$ and let $(F,F'), (G,G') \in \mathscr{C}^{p,q}_X$ be controlled paths with respect to $X$. Then, the limit
  \begin{equation*}
    \int_0^t F_r \dd G_r = \lim_{n \to \infty} \sum_{k=0}^{N_n-1} (F_{t^n_k} + \gamma F_{t^n_k, t^n_{k+1}}) G_{t^n_k \wedge t, t^n_{k+1} \wedge t},
  \end{equation*}
  exists, where the convergence holds uniformly for $t \in [0,T]$, and it coincides with the rough integral of $(F,F')$ against $(G,G')$ as defined in \eqref{eq: integration against controlled path}.
\end{theorem}

\begin{proof}
  By Lemma~\ref{lemma: integration against controlled path}, the rough integral of $(F,F')$ against $(G,G')$ (relative to $\bX^{\gamma,\pi}$) exists. 
    
  We denote by $(\bar{F}^n)_{n \in \N}$, $(\bar{G}^n)_{n \in \N}$ and $(\bar{X}^n)_{n \in \N}$ the piecewise linear interpolation of $F$, $G$ and $X$, respectively, along $\pi = (\pi^n)_{n \in \N}$. Thus, $(\bar{F}^n, F')$ and $(\bar{G}^n, G')$ are controlled by $\bar{X}^n$, with remainders $R^{\bar{F}^n}_{s,t} = \bar{F}^n_{s,t} - F'_s \bar{X}^n_{s,t}$ and $R^{\bar{G}^n}_{s,t} = \bar{G}^n_{s,t} - G'_s \bar{X}^n_{s,t}$, $(s,t) \in \Delta_T$, respectively. As shown in the proof of~\cite[Theorem~4.19]{Perkowski2016}, if $p' > p$ and $q' > q$ such that $2/p' + 1/q' > 1$, then $(\bar{F}^n,F',R^{\bar{F}^n})$ converges in $(q',p',r')$-variation to $(F,F',R^F)$, where $1/r' = 1/p' + 1/q'$.
    
  Since the sequence $(\bar{X}^n)_{n \in \N}$ has uniformly bounded $p$-variation and $\bar{X}^n$ converges uniformly to $X$ as $n \to \infty$, it follows by interpolation that $\bar{X}^n$ converges to $X$ with respect to the $p'$-variation norm, i.e., $\|\bar{X}^n - X\|_{p'} \to 0$ as $n \to \infty$. It follows similarly using \cite[Lemma~2.11]{Das2025} that $\|(\bar{\X}^{n,(2)}- (\X^{\gamma,\pi,(2)} + \frac{1}{2} [X]^{\gamma,\pi})\|_{\frac{p'}{2}} \to 0$ and, hence, that $\|(\bar{X}^n,\bar{\X}^{n,(2)}) - (X,\X^{\gamma,\pi,(2)} + \frac{1}{2} [X]^{\gamma,\pi})\|_{p'} \to 0$ as $n \to \infty$; analogously for $(G,G',R^G)$.
    
  It follows from \cite[Lemma~A.2]{Allan2025} that
  \begin{equation}\label{eq: rough integrals converge}
    \int_0^t \bar{F}^n_r \dd \bar{G}^n_r \longrightarrow \int_0^t F_r \dd G_r \qquad \text{as } n \longrightarrow \infty,
  \end{equation}
  where the convergence is uniform in $t \in [0,T]$. Note that in \eqref{eq: rough integrals converge} the integral $\int_0^t F^n_u \dd \bar{G}_u$ is defined relative to the rough path $(\bar{X}^n, \bar{\X}^{n,(2)})$, whilst the limiting rough integral $\int_0^t F_u \dd G_u$ is defined relative to $(X,\X^{\gamma,\pi,(2)} + \frac{1}{2} [X]^{\gamma,\pi})$.
    
  But, for every $t \in [0,T]$, it holds that
  \begin{align*}
    &\lim_{n \to \infty} \int_0^t \bar{F}^n_r \dd \bar{G}^n_r \\
    &\quad = \lim_{n \to \infty} \sum_{k=0}^{N_n-1} (F_{t^n_k} + \frac{1}{2} F_{t^n_k, t^n_{k+1}}) G_{t^n_k \wedge t, t^n_{k+1} \wedge t} \\
    &\quad = \lim_{n \to \infty} \Big( \sum_{k=0}^{N_n-1} (F_{t^n_k} + \gamma F_{t^n_k, t^n_{k+1}}) G_{t^n_k \wedge t, t^n_{k+1} \wedge t} + \frac{1}{2} (1 - 2\gamma) \sum_{k=0}^{N_n-1} F_{t^n_k, t^n_{k+1}} G_{t^n_k \wedge t, t^n_{k+1} \wedge t} \Big).
  \end{align*}
  Since $(F,F'), (G,G') \in \mathscr{C}^{p,q}_X$, it is immediate that the second term on the right-hand side converges to $\frac{1}{2} \int_0^t F'_s G'_s \dd [X]^{\gamma,\pi}_s$, $t \in [0,T]$.

  Then, we have that
  \begin{align*}
    &\lim_{n \to \infty} \sum_{k=0}^{N_n-1} (F_{t^n_k} + \gamma F_{t^n_k,t^n_{k+1}}) G_{t^n_k \wedge t, t^n_{k+1} \wedge t} \\
    &\quad = \lim_{n \to \infty} \int_0^t \bar{F}^n_r \dd G^n_r - \frac{1}{2} \int_0^t F'_r G'_r \dd [X]^{\gamma,\pi}_r \\
    &\quad = \int_0^t F_r \dd G_r - \frac{1}{2} \int_0^t F'_r G'_r \dd [X]^{\gamma,\pi}_r \\
    &\quad = \lim_{|\cP|\to 0} \sum_{[u,v]\in \cP} F_u G_{u,v} + F'_u G'_u (\X^{\gamma,\pi,(2)} + \frac{1}{2}[X]^{\gamma,\pi})_{u,v} - \frac{1}{2} \lim_{|\cP|\to 0} \sum_{[u,v]\in \cP} F'_u G'_u [X]^{\gamma,\pi}_{u,v} \\
    &\quad = \lim_{|\cP|\to 0} \sum_{[u,v]\in \cP} F_u G_{u,v} + F'_u G'_u \X^{\gamma,\pi,(2)}_{u,v} \\
    &\quad = \int_0^t F_r \dd G_r,
  \end{align*}
  where the limit is taken over any sequence of partitions $\cP$ of $[0,t]$ with vanishing mesh size.
\end{proof}

The following proposition states that the map mapping a rough path to its rough path bracket is continuous.

\begin{proposition}\label{prop: rough path to bracket is continuous}
  Let $p \in (2,3)$. The map
  \begin{equation*}
    \cC^p([0,T];\R^d) \ni \bX = (X,\X^{(2)}) \mapsto [\bX] \in C^{\p}([0,T];\R^d)
  \end{equation*}
  is continuous, where the rough path bracket $[\bX]$ of a rough path $\bX = (X,\X^{(2)})$ is defined by $[\bX]_t := X_{0,t} \otimes X_{0,t} - (\X^{(2)}_{0,t} + (\X^{(2)}_{0,t})^\top)$, $t \in [0,T]$.
\end{proposition}

\begin{proof}
  Let $\bX^n = (X^n, \X^{n,(2)})$, $n \in \N$, $\bX = (X,\X^{(2)})$ be rough paths such that
  \begin{equation*}
     \|\bX^n;\bX\|_p = \|X^n-X\|_p + \|\X^{n,(2)}-\X^{(2)}\|_{\p} \longrightarrow 0 \qquad \text{as} \quad n \to \infty.
  \end{equation*}
  We first have that
  \begin{equation*}
    \|\X^{n,(2)} + (\X^{n,(2)})^\top - (\X^{(2)} + (\X^{(2)})^\top)\|_{\p} \lesssim \|\X^{n,(2)} - \X^{(2)}\|_{\p}. 
  \end{equation*}
  Further, it holds that
  \begin{align*}
    &|X^n_{s,t} \otimes X^n_{s,t} - X_{s,t} \otimes X_{s,t}|^{\p} \lesssim |X^n_{s,t} \otimes (X^n_{s,t} - X_{s,t})|^{\p} + |(X^n_{s,t} - X_{s,t}) \otimes X_{s,t}|^{\p} \\
    &\quad \lesssim (|X^n_{s,t}|^{\p} + |X_{s,t}|^{\p}) |X^n_{s,t}-X_{s,t}|^{\p}
  \end{align*}
  for any $(s,t) \in \Delta_T$. Thus, by the Cauchy--Schwarz inequality,
  \begin{align*}
    &\|X^n \otimes X^n - X \otimes X\|_{\p}^{\p} \\
    &\quad = \sup_{\cP} \sum_{[s,t] \in \cP} |X^n_{s,t} \otimes X^n_{s,t} - X_{s,t} \otimes X_{s,t}|^{\p} \\
    &\quad \lesssim \sup_{\cP} \sum_{[s,t] \in \cP}  (|X^n_{s,t}|^{\p} + |X_{s,t}|^{\p}) |X^n_{s,t}-X_{s,t}|^{\p} \\
    &\quad \lesssim \sup_{\cP} \bigg(\sum_{[s,t] \in \cP} (|X^n_{s,t}|^p + |X_{s,t}|^p) \bigg)^{\frac{1}{2}}  \bigg(\sum_{[s,t] \in \cP} |X^n_{s,t} - X_{s,t}|^p \bigg)^{\frac{1}{2}} \\
    &\quad \lesssim (\|X^n\|_p^{\p} + \|X\|_p^{\p}) \|X^n-X\|_p^{\p}.
  \end{align*}
  This implies that
  \begin{equation*}
    \|[\bX^n] - [\bX]\|_{\p} \lesssim (\|X^n\|_p + \|X\|_p) \|X^n-X\|_p + \|\X^{n,(2)} - \X^{(2)}\|_{\p} \longrightarrow 0 \qquad \text{as} \quad n \to \infty,
  \end{equation*}
  therefore, $\bX \mapsto [\bX]$ is continuous.
\end{proof}

\bibliography{references}{}
\bibliographystyle{amsalpha}

\end{document}